\documentclass{article}
\usepackage[utf8]{inputenc}
\usepackage{graphicx}
\usepackage{xcolor}
\usepackage{amssymb,amsmath,amsthm,mathrsfs,bm}
\usepackage{geometry}
\usepackage{biblatex}
\newtheorem{theorem}{Theorem}
\newtheorem{problem}{Problem}
\newtheorem{definition}{Definition}
\newtheorem{lemma}[theorem]{Lemma}

\newtheorem{corollary}[theorem]{Corollary}

\title{On General 2-dimensional Lattice Spectra: Closedness, Hall's Ray, and Examples}
\author{Ruichong Zhang\\
Tsinghua University}
\date{June 2024}

\DeclareMathOperator{\diag}{diag}

\DeclareMathOperator{\sgn}{sgn}
\newcommand{\N}{\mathbb{N}}
\newcommand{\Z}{\mathbb{Z}}
\newcommand{\Q}{\mathbb{Q}}
\newcommand{\R}{\mathbb{R}}

\begin{document}

\maketitle

\begin{abstract}
    The Lagrange and Markov spectra have been studied since late 19th century, concerning badly approximable real numbers. The Mordell-Gruber spectrum has been studied since 1936, concerning the supremum of the area of a rectangle centered at the origin that contains no other points of a unimodular lattice. We develop techniques that incorporate unimodular lattices and integer sequences, providing the log-systole function which unifies four famous spectra. We compute the Mordell-Gruber spectrum in the two-dimensional case and generalize Perron's formulas behind some famous spectra. Furthermore, we generalize the sum of Cantor sets to prove that certain functions on cartesian product of two Cantor sets contain an interval. Combining the techniques, we prove closedness and existence of Hall's interval in several different applications.
\end{abstract}

\section{Introduction}

\subsection{The Lagrange and Markov Spectrum}
\textbf{The Lagrange Spectrum.} The Lagrange spectrum plays a central role in Diophantine Approximation, which concerns how "badly approximable" an irrational can be. The set of Lagrange spectrum in the context of number theory is shown as
$$\mathrm{L} := \left\{\limsup_{n \in \N} \left((n ||n\alpha||)^{-1}\right) \middle| \alpha \in \R \setminus \Q\right\} \quad ( \text{where } \lVert \_ \rVert := d(\_, \Z))$$
Or also
$$
\begin{aligned}
    \mathrm{L} &= \left\{ \left(\liminf_{k \to \infty} (q_k |q_k\alpha - p_k|) \right)^{-1} \middle| \alpha \in \R \setminus \Q \right\} \\
\end{aligned}
$$
Where $p_k / q_k$ is the $k$-th convergent of $\alpha$. \\
\textbf{The Markov Spectrum.}  Let $f(x,y) = ax^2 +bxy + cy^2$ be a quadratic form with determinant $\Delta_f = b^2 - 4ac = 1$. The Markov spectrum \cite{Cusick_Flahive_Markoff_Lagrange} is obtained by traversing all possible such quadratic forms:
$$
\mathrm{M} := \left\{ \left(\inf_{(x,y) \in \Z^2 \setminus \{(0,0)\}} |f(x,y)|\right)^{-1} \middle| \Delta_f = 1\right\}
$$
\textbf{Structure of Lagrange and Markov Spectra.} It's shown with Perron's formula \cite{Perron_Irrationalzahlen}
that Lagrange and Markov spectra can be computed via continuous fraction expansion. For example, let $\{a_n\} : \Z \to \N^+$ be a doubly infinite sequence, and let
$$
P_k^{[a]} = a_k + [a_{k-1}; \ a_{k-2}, a_{k-3}, \cdots] + [a_{k+1};\  a_{k+2}, a_{k+3}, \cdots]
$$
where we use continued fraction expansion of
$$[s_0;\  s_1, s_2, \cdots] = s_0 + \frac{1}{s_1+\frac{1}{s_2+\frac{1}{\cdots}}}$$
and the Lagrange and Markov spectra are obtained by taking different limits:
$$
\begin{aligned}
    \mathrm{L} &= \left\{\ \limsup_{k\in \Z} P_k^{[a]} \ \middle|\  a : \Z \to \N^+\right\}\\
    \mathrm{M} &= \left\{\ \sup_{k\in \Z}   P_k^{[a]} \ \middle|\  a : \Z \to \N^+\right\}\\
\end{aligned}
$$

It has been proved for the Lagrange and Markov spectra in \cite{Cusick_Flahive_Markoff_Lagrange} that both $\mathrm{L}$ and $\mathrm{M}$ are closed sets, and $\mathrm{L}$ is a proper subset of $\mathrm{M}$. 

The minimum of Lagrange and Markov spectra is $\sqrt{5}$, known as the Hurwitz's Constant. The lower part of Lagrange and Markov spectra the part smaller than 3, whereas $\mathrm{L} \cap [\sqrt{5}, 3) = \mathrm{M} \cap [\sqrt{5}, 3) $ is a discrete subset of $[\sqrt{5}, 3)$, so it is also known as the discrete part. 

The Hall's ray is a famous theorem stating that there exists a $K$ such that $[K, \infty] \subset \mathrm{L}$ and $[K, \infty] \subset \mathrm{M}$, which is elegantly shown by Hall in his 1947 \emph{Annals} paper \cite{Hall_Sum_Product_CF}. The best constant of $K$ has been given by Freiman \cite{freiman1975diophantine} in 1975, who showed that 
$$K_0 = \frac{2221564096 + 283748\sqrt{462}}{491993569} = 4.5278 \cdots$$
which is the minimum of $K$ for both $[K, \infty] \subset \mathrm{L}$ and $[K, \infty] \subset \mathrm{M}$. The part $\mathrm{L} \cap [K_0, \infty] = \mathrm{M} \cap [K_0, \infty] = [K_0, \infty]$ is called \emph{Hall's ray}.

It's worth noting that Freiman's proof involves more than 100 pages of computations and has never been verified. 

The Lagrange and Markov spectra restricted on the interval of $[\sqrt{5}, K_0)$ are called the intermediate part and have complex structures. This is the part where the Lagrange and Markov spectra differ: $\mathrm{M} \setminus \mathrm{L} \subset [\sqrt{5}, K_0)$. This part contains multiple intervals and gaps, of which the most famous gap is Perron's gap \cite{Perron_Irrationalzahlen}: $(\sqrt{12}, \sqrt{13}) \cap \mathrm{M} = \emptyset$.
Recently, it has been proved \cite{moreira2018geometricpropertiesmarkovlagrange, cerqueira2018continuityhausdorffdimensiongeneric} that the Hausdorff dimension of Lagrange spectrum and Markov spectrum intersected with half-lines are equal, and is a continuous function. Namely, $f(t) = \dim(\mathrm{L} \cap (-\infty, t)) = \dim(\mathrm{M} \cap (-\infty, t))$ is a continuous function with range $[0,1]$.
See \cite{Smirnov2023Freiman} for a comprehensive literature review about Lagrange and Markov spectra.


\subsection{Dirichlet and Mordell-Gruber Spectra}
\textbf{The Dirichlet Spectrum.} From \cite{Divis_Novak_Diophantine_Approximations}, the Dirichlet spectrum is defined as follows.

\begin{definition}
    The Dirichlet constant for a given irrational number $\alpha$ is defined as
    $$ \mu(\alpha) = \limsup_{t \in \N} t \cdot \left(\min_{p, q \in \N, 0<q \le t} |q\alpha - p| \right)$$
    The Dirichlet spectrum is defined as all possible values that the Dirichlet constant can take:
    $$ \mathrm{D} = \left\{ \mu(\alpha) | \alpha \in \R \setminus \Q \right\}$$
\end{definition}

An alternative definition is to use the second kind convergents:
$$\mu(\alpha) = \limsup_{k \in \N} q_{k+1} |q_k\alpha - p_k| $$
where $p_k/q_k$ is the $k$-th convergent of $\alpha$.

It is proved in \cite{Divis_Lagrange_Numbers} and \cite{Divis_Novak_Diophantine_Approximations} that 
$\mu(\alpha)$ has an alternative form:
\begin{theorem}
    Let $[a_0; a_1, \cdots]$ be the continued fraction expansion of $\alpha$, then the Dirichlet constant $\mu(\alpha)$ takes the following form:
    $$\mu(\alpha) = \limsup _{k \to \infty} \frac{1}{1+ \alpha_{k}^{[a]}\beta_{k+1}^{[a]}} $$
    where the notation $\alpha_{k}^{[a]} = [0; a_{k+1}, a_{k+2}, \cdots]$ and $\beta_{k}^{[a]} = [0; a_{k-1}, a_{k-2}, \cdots, a_1]$.
\end{theorem}
They also showed that $\mathrm{D}$ admits a counterpart of Hall's ray, whereas
$$\left[ \frac{83+18\sqrt{2}}{87+18\sqrt{2}} \approx 0.9645,\  1\right] \subset \mathrm{D}$$

The paper \cite{ivanov1982origin} has improved the bound quite a lot. However, to the best of our knowledge, the optimal bound (analog of Freiman's constant) for Hall's ray counterpart in Dirichlet spectrum has not been accurately computed. \\
\textbf{The Mordell-Gruber Spectrum.} 
Let $\Lambda \subset \R^2$ be a nondegenerate lattice, which admits a basis $\{A, B\}$. The lattice is then generated as $\Lambda = \Z A  + \Z B$, whose covolume is
$$\mathrm{covol}(\Lambda) = | \det(A;B) |$$
We call $\Lambda$ unimodular if it has a unit covolume, which is $\mathrm{covol}(\Lambda) = 1$.
 
The Mordell-Gruber Spectrum considers how large an admissible rectangle can be in a certain lattice.

\begin{definition} 
    Let $C$ be a rectangle centered at origin in $\R^2$, where $ C = [-x, x] \times [-y, y]$. 
    The rectangle $C$ is called admissible in a unimodular lattice $\Lambda$ if $C \cap \Lambda = \{(0,0)\}$, where $C$ contains no more lattice points than the origin.
\end{definition}

The Mordell constant of a unimodular lattice $\Lambda$, denoted $\kappa(\Lambda)$, is therefore
$$\kappa(\Lambda) := \sup_{[-x, x] \times [-y, y] \text{ admissible in }\Lambda} xy$$

The two-dimensional Mordell-Gruber spectrum \cite{Shapira_Weiss_Mordell_Gruber} is obtained by the set of all accessible values of $\kappa(\Lambda)$:
$$\mathrm{MG}_2 := \left\{ \kappa(\Lambda)\  \middle| \ \Lambda \text{ is a unimodular lattice in } \mathbb{R}^2\right\}$$

By Minkowski convex body theorem, $\mathrm{MG}_2 \subset (0,1]$, and it has been shown
that \cite{Szekeres_Lattice_Plane} $$\min (\mathrm{MG}_2) = \frac{5+\sqrt{5}}{10}$$

It's worth noting that the paper \cite{Shapira_Weiss_Mordell_Gruber} has not mentioned the Perron's formula or the Hall's ray, thus failed to prove the cardinality of $\mathrm{MG}_2$. Just like the relation of $\mathrm{L} \subset \mathrm{M}$ for Lagrange and Markov spectra, we will try to prove that $\mathrm{D} \subset \mathrm{MG}_2$ and show the analog of Hall's ray in this work.

\subsection{The Dual Problem}
From the dynamical system's point of view, all the four spectra above are corresponded with a single function, called the log-systole function. Here we adapt the concept of "systole" in geometry to study lattices under the evolution of diagonal flow. The four spectra correspond to taking different types of limits ($\liminf, \inf, \limsup, \sup$) of one log-systole function $W(t; \Lambda)$:
\begin{itemize}
    \item Lagrange spectrum $\mathrm{L}$ - $\liminf_t W(t; \Lambda)$;
    \item Markov spectrum $\mathrm{M}$ - $\inf_t W(t; \Lambda)$;
    \item Diriclet spectrum $\mathrm{D}$ - $\limsup_t W(t; \Lambda)$;
    \item Mordell-Gruber spectrum $\mathrm{MG}_2$ - $\sup_t W(t; \Lambda)$.
\end{itemize}
This is just a brief overview. In the next chapter, we will introduce the log-systole function and study its properties like local maxima and minima, as well as providing the corresponding Perron's formula (Theorem \ref{fourinone}). 

\subsection{Higher Dimensional Generalizations}
The above spectra, as well as some others, can be generalized to higher dimensions.

One generalization of Lagrange spectrum involves changing $n \lVert n\alpha \rVert$ to $n \lVert n\alpha \rVert \lVert n\beta \rVert$, which is a two dimensional simultaneous Diophantine approximation. The corresponding lattice is therefore 3-dimensional and could be generated by $(1,0,0), (0,1,0)$ and $(\alpha, \beta, 1)$. This is the famous Littlewood conjecture in Diophantine approximation: if $(\liminf_n (n \lVert n\alpha \rVert \lVert n\beta \rVert))^{-1}$ admits a Hall's ray, it would imply that the Littlewood conjecture is false. The existence of Hall's ray remains an open problem.

The Mordell-Gruber spectrum in higher dimensions is discussed in the paper \cite{Shapira_Weiss_Mordell_Gruber}, where the Mordell-Gruber spectrum for $n$-dimensional lattices is denoted by $\mathrm{MG}_n$. Since $\mathrm{MG}_n \subset \mathrm{MG}_{n+1}$, if we prove that $\mathrm{MG}_2$ admits an analog of Hall's ray, so does any $\mathrm{MG}_n$. However, a reduced version of Mordell-Gruber spectrum is also discussed in that paper, where the Mordell constant is taken over all "indecomposable" lattices of dimension $n$, denoted by
$$\hat{\mathrm{MG}}_n := \{ \kappa (\Lambda) \ | \ \Lambda \text{ is an indecomposable }n\text{-dimensional lattice}\}
$$
We indeed have $\hat{\mathrm{MG}}_2 = \mathrm{MG}_2$, but not necessarily $\hat{\mathrm{MG}}_n = \mathrm{MG}_n$. Therefore, it remains an open problem to prove or disprove the existence of Hall's ray in $\hat{\mathrm{MG}}_n$ for $n \ge 3$.

The Dirichlet spectrum can be generalized to higher dimensions as well. The paper 
\cite{akhunzhanov2021notedirichletspectrum}
is a comprehensive survey on the generalized Dirichlet spectrum for higher dimensions under arbitrary norm. The paper 
\cite{akhunzhanov2013twodimensionaldirichletspectrum}
discusses Dirichlet spectrum of two-dimensional approximation under the $\ell^2$ norm, or equivalently, the limit supremum volume of an admissible cylinder (shaped like ${x^2+y^2\le r^2} \times [-z,z]$) under the lattice in $\R^3$ generated by $(1,0,0), (0,1,0), (v_1, v_2, 1)$. It is shown that the spectrum $\mathrm{D}_2 = [0, 2/\sqrt{3}]$ which is an entire segment. However, the Dirichlet spectrum of $n$-dimensional approximation (admissible volume of $(n+1)$-dimensional lattices) for $n \ge 3$ is not well known. It would be interesting to show the existence of Hall's ray (segment) for this case.

\subsection{Some Notes on This Paper}
\textbf{Notations.} It's necessary to use different notations to distinguish between continued fraction expansions and sequences. We use round brackets to denote sequences: $(a_0, a_1, \cdots)$ or $(\cdots, a_{-1}, a_0, a_1, \cdots)$. Square brackets with colons are used for continued fractions:
$$[k_0; \ k_1, \  k_2, \cdots] = k_0+ \frac{1}{ k_1+ \frac{1}{ k_2+ \frac{1}{ \cdots }}}$$

Another important notation is the use of $\alpha$ and $\beta$. The notation $\alpha_t^{[a]}$ means the part of the sequence $\{a_n\}$ to the right of the offset $t$ (excluding $t$ itself) regarded as a continued fraction expansion:
$$\alpha_t^{[a]} := [0; \ a_{t+1}, a_{t+2}, a_{t+3}, \cdots] \in [0,1]$$
If there is only one sequence $a$ to be considered, we may use $\alpha_t = \alpha_t^{[a]}$ for convenience. Similarly, $\beta_t^{[a]}$ means the part of the sequence $\{a_n\}$ to the \emph{left} of offset $t$ (excluding $t$ itself) \emph{reversed and} regarded as a continued fraction expansion:
$$\beta_t^{[a]} := [0; \ a_{t-1}, a_{t-2}, a_{t-3}, \cdots] \in [0,1]$$
If $a: \N \to \N^+$ which has no term before $a_0$, then $\beta_t^{[a]}$ is a rational number:
$$\beta_t^{[a]} := [0; \ a_{t-1}, a_{t-2}, \cdots, a_0] \in \Q$$
Similarly, if only one sequence $a$ is considered, we use $\alpha_t = \alpha_t^{[a]}$. 
\\
\textbf{Important Theorems.} There are several important theorems in this work which might arouse interest:
\begin{enumerate}
    \item Theorem \ref{interval}: This theorem shows certain functions on the cartesian product of two certain Cantor sets has an interval range. This theorem is the key to prove the existence of Hall's ray in some spectra. The proof of the theorem is self-contained and does not rely on previous chapters.
    \item Theorem \ref{infclosed} and \ref{liminfclosed}: These two theorems prove the closedness of spectra given good continuity on a Perron function. The steps are mainly generalized from the book \cite{Cusick_Flahive_Markoff_Lagrange}.
    \item Theorem \ref{fourinone}: This theorem unifies four famous spectra with a single log-systole function.
\end{enumerate}
\ \\
\textbf{Caveats.} This paper is my undergraduate thesis at Tsinghua University. A translated Chinese version of this paper is available.

\section{Projective Invariance of 2-Dimensional Lattices}
Throughout this section, we develop a tool to study lattices called projective invariance, which corresponds positive integer sequences with lattices under certain flow action. 

We are particularly interested in \emph{bi-infinite} lattices, where a lattice has has no nontrivial points located on the axis. 
\begin{definition}
    A unimodular lattice $\Lambda$ is called bi-infinite if $\Lambda \cap \{x=0\} = \{(0,0)\}$ and $\Lambda \cap \{y=0\} = \{(0,0)\}$.
\end{definition}
One reason is that non-bi-infinite lattices have simpler structures and can be dealt directly.

\subsection{Pivots and Orders}
Given $\Lambda$ a unimodular lattice, first we consider a set of points that are sufficiently close to the origin, called pivots.
\begin{definition}
    Let $A = (x,y) \in \Lambda$ be a nonzero point in the lattice ($ (x,y) \ne (0,0)$). We define the rectangle box of $A$ as $R(A) := \left[-|x|, |x|\right] \times \left[-|y|, |y|\right]$. The point $A$ is called a pivot if there are no nonzero lattice points in the interior of $R(A)$, i.e., $\mathrm{Int}(R(A)) \ \cap \ \Lambda = \{ (0,0)\} $.
\end{definition}
It's obvious that $R(A) = R(-A)$, thus $A$ is a pivot if and only if $-A$ is a pivot. Without loss of generality, we only compute pivots above the x-axis. Define the pivot set
$$ \Pi(\Lambda) := \left\{ A \in \Lambda \cap \{y \ge 0\}\  \middle|\  A \text{ is a pivot}\right\}$$
and we can deduce some key properties about $\Pi(\Lambda)$.

\begin{lemma}
    The set $\Pi(\Lambda)$ has some key properties:
    \begin{enumerate}
        \item $\Pi(\Lambda) \ne \emptyset$.
        \item If $\Lambda$ is bi-infinite, then $\Pi(\Lambda)$ is a totally ordered set, where the order is induced by taking the $y$ component:
        $$ A < B \iff y_A < y_B$$
        \item If $\Lambda$ is bi-infinite, then the totally ordered set $(\Pi(\Lambda), < )$ is isomorphic to $\mathbb{Z}$.
        \item The order of $(\Pi(\Lambda), < )$ can be equivalently obtained by taking the absolute value of the $x$ coordinate:
        $$ A < B \iff y_A < y_B \iff |x_A| > |x_B|$$
    \end{enumerate}
\end{lemma}

\begin{proof}
    \begin{enumerate}
        \item We prove this by applying Zorn's lemma. Assign $\Lambda \setminus \{(0,0)\}$ a poset structure of
        $$A \le_\Lambda B \iff R(A) \subset R(B), \quad \forall A, B \in \Lambda \setminus \{(0,0)\}$$
        then the minimal elements under $\le_\Lambda$ are pivots of $\Lambda$. If we can prove that every chain admits a lower bound, we are done. Otherwise, suppose 
        $$R(A_1) \supset R(A_2) \supset \cdots$$
        There exists a subsequence of $\{A_i\}$ in the same quadrant, WLOG the first quadrant. We denote the new sequence by $\{B_i\}$, then 
        $$|x_{B_1}| \ge |x_{B_2}|  \ge \cdots \ge 0$$
        $$|y_{B_1}| \ge |y_{B_2}|  \ge \cdots \ge 0$$
        We can remove the absolute values since all $B_i$'s are in the first quadrant. By monotonic convergence of sequences, we obtain that $\lim_{i \to \infty} (x_{B_i}, y_{B_i}) = (x_0, y_0)$, therefore $(x_0, y_0)$ is an accumulation point of $\Lambda$. However, $\Lambda$ is a closed set, a contradiction. Thus $\Pi(\Lambda)$ is nonempty.
        \item If $\Lambda$ is bi-infinite, then for two different points $A, B \in \Pi(\Lambda)$, their $y$ coordinates may not be equal. Suppose otherwise that $y_A = y_B$, then $y_{(A-B)} = 0$, which implies that $(A-B)$ is on the x-axis, a contradiction of the bi-infinite definition.
        \item We will construct by induction the following map:
        $$\Psi: \mathbb{Z} \to \Pi(\Lambda)$$
        First, we have just proved above that $\Pi(\Lambda)$ is not an empty set. We arbitrarily take $A_0 \in \Pi(\Lambda)$ and define $\Psi(0) = A_0$.

        Now we construct $\Psi(i)$ by induction on $i (i = 0,1, \cdots)$.

        Assume we have defined $\Psi(0), \Psi(1), \cdots, \Psi(i)$, and denote by $\Psi(i) = (x_i,y_i)$. Consider a rectangle $C_{(y,i)} := [-|x_i|, |x_i|] \times [0,y]$, and take $y^*$ be the supremum of $y$ such that the interior of $C_{(y,i)}$ does not contain lattice points, i.e., $$y^* = \sup_{\mathrm{Int} \left(C_{(y,i)}\right) \cap \Lambda = \emptyset } y $$
        
        Using Minkowski convex body theorem, the area of $C_{(y,i)}$ should be no more than $2$, therefore $y^*$ is finite. 
        We show that there is a unique pivot point at $(-|x_i|, |x_i|) \times \{y^*\}$. Since $y^*$ is the supremum, for any $\epsilon > 0$ the region $(-|x_i|, |x_i|) \times [y^*, y^*+\epsilon)$ contains lattice points. It is clear that $[-|x_i|, |x_i|] \times \{y^*\}$ contains an accumulation point of $\Lambda$. Using discreteness, $ \left([-|x_i|, |x_i|] \times \{y^*\} \right) \cap \Lambda \ne \emptyset$. We exclude two endpoints because $\Psi(i)$ and $-\Psi(i)$ already occupy x coordinate of $\{-|x_i|, |x_i|\}$. Using bi-infinitude again, we can show that the point at $(-|x_i|, |x_i|) \times \{y^*\}$ is unique, which we assign to $\Psi(i+1)$.

        After obtaining $\Psi(i) (i \in \N)$, we obtain inductively $\Psi(-i) (i \in \N)$ using similar techniques. This time, we consider a rectangle $D_{(x,i)} := [-|x|,|x|] \times [0,y_{-i}]$ and let 
        $$x^* = \sup_{\mathrm{Int} \left(D_{(x,i)}\right) \cap \lambda = \emptyset } |x| $$
        By analogy of Minkowski convex body theorem, $x^*$ is finite, and a pivot lies in the set $\{-x^*, x^*\} \times (0,y_{-i})$. We assign $\Psi(-i-1)$ to it, and all negative indices are set.

        Notice that by constructing rectangles, we also showed a side-product that $|x_j| > |x_{j+1}|$ and $y_j < y_{j+1}$, indicating that $\Psi$ is injective. It remains to show that $\Psi$ is surjective.

        Assume $P = (x_P,y_P) \in \Lambda$ is a pivot, and without losing of generality, $y_P >y_0$. If $P$ is none of $\Psi(i) (i \in \Z)$, we show that it's impossible.
        \begin{enumerate}
            \item If all $i \in N$ satisfies $y_i \le y_p$, then $\Psi(i) (i \in \N)$ is a bounded sequence and $\Lambda$ admits a nontrivial accumulation point, which is impossible.
            \item Otherwise, we can always find $i$ such that $y_i \le y_p$ and $y_{i+1} > y_P$.
            If $|x_i| \le |x_p|$, this contradicts that $\Psi(i)$ is a pivot. Therefore $|x_P| < |x_i|$ and $P$ lies in interior of the rectangle $[-|x_i|, |x_i|] \times [0, y_{i+1}]$, but this again contradicts the construction of $C(y_{i+1}, i)$ such that $C(y_{i+1}, i)$ contains no lattice points in the interior.
        \end{enumerate}
        We have thus proved that $\Psi$ is an order isomorphism.
        
        \item As we have showed above, $y_{i} < y_{i+1}$ and $|x_i| > |x_{i+1}|$. Thus $y_A < y_B \iff \Psi^{(-1)}(A) < \Psi^{(-1)}(B) \iff |x_A| > |x_B|$.
    \end{enumerate}  
\end{proof}

Note: There is also a concept of "stable pivots", as shown in this paper \cite{Knitter_Complexity_Domain_Approximation}. We avoid this concept by discussing bi-infinite lattices.

\subsection{The Alternating Principle, Bases and Indices}
To acquire a deeper understanding of the structure of pivots, we analyze the relations between pivots. We first study two consecutive pivots.
\begin{definition}
    For two pivots $A$ and $B$ in $\Lambda$, $A$ and $B$ are called consecutive if for order isomorphism $\Phi: \Pi(\Lambda) \to \Z$, $|\Phi(A) - \Phi(B)| = 1 $.
\end{definition}
It's easy to check that the definition is well-defined and independent of the choice of $\Phi$.

For consecutive pivots we have the following properties:
\begin{lemma}
    If $A = (x_A, y_A)$ and $B=(x_B, y_B)$ are two consecutive pivots of a unimodular bi-infinite lattice $\Lambda$, and further assume $0 < y_A < y_B$. Then we have:
    \begin{enumerate}
        \item The points $A$ and $B$ do not fall in the same quadrant.
        \item The points $A$ and $B$ form a basis of $\Lambda$.
    \end{enumerate}
\end{lemma}
\begin{proof}
    \begin{enumerate}
        \item From proof the previous section, we know that there are no more lattice points in the interior of rectangle $ C = [-|x_A|, |x_A|]\times [0,y_B]$, and $0 < y_A < y_B$ implies $|x_A| > |x_B|$. Now, assume that $A$ and $B$ fall in the same quadrant, so $x_A$ and $x_B$ have the same sign and $|x_A - x_B| < |x_A|$. Combining with $0 < y_B - y_A < y_B $ we know that $(B-A)$ lies in the interior of rectangle $C$, contradiction.
        \item Consider parallelogram $V$ with vertices $A, B, -A, -B$. Since $V \subset [-|x_A|, |x_A|]\times [-y_B,y_B]$ and $A, B, -A, -B$ do not lie on the vertices of the rectangle, we obtain the inequality by (Minkowski's Theorem)
        $$4 \cdot \mathrm{Area}(\Delta_{A, (0,0), B})\le \mathrm{Area}(V) < 4\cdot |x_A|\cdot y_B \le 4$$
        Therefore $\mathrm{Area}(\Delta_{A, (0,0), B}) < 1$. 
        We have
        $$\lvert\det(A;B)\rvert = 2\cdot \mathrm{Area}(\Delta_{A, (0,0), B}) < 2$$
        and $\lvert\det(A;B)\rvert \in \N$, where $\lvert\det(A;B)\rvert \ne 0$ since $A$ and $B$ lie in different quadrants.
        We conclude that $\lvert\det(A;B)\rvert =1$, signifying that $A$ and $B$ is a basis of $\Lambda$.
    \end{enumerate}
\end{proof}

Now let's consider an order isomorphism $\Psi: \Z \to \Pi(\Lambda)$, and denote $\Psi(j) = (x_j, y_j)$. Since $\Psi(j)$ and $\Psi(j+1)$ never fall in the same quadrant, if we restrict $\Pi(\Lambda)$ above the $x$-axis we can deduce that $\Psi(j-1)$ and $\Psi(j+1)$ fall in the same quadrant. Thus, all $\Psi(2t) (t \in \Z)$ falls in one quadrant and $\Psi(2t+1)$ falls in the other. This is the alternating principle of pivots.

This alternating principle enables computation of the relation between three consecutive pivots:  $\Psi(j-1)$, $\Psi(j)$ and $\Psi(j+1)$, and in fact, they form an elegant relation:

\begin{lemma}
    \label{indexing}
    For any $j\in \Z$, there exists a positive integer $a_j$, such that 
    $$ \Psi(j+1) -a_j \Psi(j) - \Psi(j-1) = 0 $$
    The sequence $\{a_j\}(j \in \Z)$ is called the index sequence.
\end{lemma}

\begin{proof}
    From the last lemma, we already know that $\{\Psi(j+1), \Psi(j)\}$ and $\{\Psi(j-1), \Psi(j)\}$ form bases.
    Consider the linear transformation between two bases:
    $$
    \begin{bmatrix}
        \Psi(j+1) \\ \Psi(j)
    \end{bmatrix}
    =
    \begin{bmatrix}
        a & b \\ 0 & 1
    \end{bmatrix}
    \begin{bmatrix}
        \Psi(j-1) \\ \Psi(j)
    \end{bmatrix}, \quad a,b \in \Z
    $$
    
    Taking the determinant, $a = \pm 1$. 
    
    Taking the y coordinate, $y_{j+1} = ay_{j-1} +by_j$ where $y_{j+1} > y_j > y_{j-1} > 0$, therefore $b>0$. 
    
    Then we take the x coordinate: $x_{j+1} = ax_{j-1} +bx_j$, but their signs satisfy $\sgn(x_{j+1}) = \sgn(x_{j+1}) = -\sgn(x_{j})$. If $a=-1$, then $\sgn(x_{j+1}) = \sgn(ax_{j-1}+bx_{j}) = \sgn(x_j)$, which leads to contradiction, so $a=1$.

    We let $a_j := b$ here and we obtain
    $$\Psi(j+1) = a_j \Psi(j) + \Psi(j-1) = b \Psi(j) + \Psi(j-1) $$
\end{proof}

\subsection{Invariance under Flow Action}
From the last lemma, we can obtain the sequence $\{a_j\}(j \in \Z)$ from a unimodular bi-infinite lattice $\Lambda$, but this sequence is not well-defined, because we don't know where $\Psi(0)$ and $a_0$ is. In fact, $\Psi(0)$ could be any point in the pivot set $\Pi(\Lambda)$. This could either be resolved by adding a starting point $\Psi(0)$ first, or by quotient of shifting the sequence.

\begin{lemma}
    \label{orderiso}
    For any unimodular bi-infinite lattice $\Lambda$ and a pivot $A \in \Pi(\Lambda)$, there exists a unique order isomorphism $\Psi: \Z \to \Pi(\Lambda)$, such that $\Psi(0) = A$. Also, this admits a unique index sequence $\{a_j\} (j \in \Z)$ from lemma above. We will denote the index sequence by $(\Lambda; A)$.
\end{lemma}

\begin{proof}
    First, there exists an isomorphism $\Phi_0: \Pi(\Lambda) \to \Z$. For existence, define $\Psi(a) := \Phi_0^{(-1)} (a + \Phi_0(A))$. For uniqueness, suppose there are two different isomorphisms $\Psi_1, \Psi_2: \Z \to \Pi(\Lambda)$, then $\Psi_1^{(-1)} \circ \Psi_2$ is an order automorphism of $\Z$ that satisfies $\Psi_1^{(-1)} \circ \Psi_2 (0) = 0$, so $\Psi_1^{(-1)} \circ \Psi_2 = \mathrm{Id}$.
\end{proof}

Next, we consider the action
$$
g_t := \begin{bmatrix}
    \exp(t) & 0 \\ 0 & \exp(-t)
\end{bmatrix}
\quad (t \in \R)
$$
that sends lattice $\Lambda$ to $g_t \Lambda$.
It's obvious that $g_t$ is a flow (a one-parameter group) of parameter $t$, that satisfies $g_s g_t = g_{s+t}$. We show that pivots and indices are equivariant under flow of $g$.
\begin{lemma}
    Let $\Lambda$ be a unimodular bi-infinite lattice, and $t \in \R$ an arbitrary parameter, then:
    \begin{enumerate}
        \item $g_t \Lambda$ is a unimodular bi-infinite lattice;
        \item Equivariance: $g_t \Pi(\Lambda) = \Pi(g_t \Lambda)$;
        \item The index sequence of $(\Lambda; A)$ is equal to the index sequence $(g_t \Lambda; g_t A)$.
    \end{enumerate}
\end{lemma}

\begin{proof}
    \begin{enumerate}
        \item Note that $\det(g_t) = 1$ and for any point $K$, $K$ is on the axis iff $g_t K$ is on the axis.
        \item By definition of pivot, a point $B = (x,y)$ is a pivot in $\Lambda$ iff $(-|x|,|x|) \times (-|y|,|y|) \cap \Lambda = \{ (0,0)\}$, which is equivalent to 
        $$\left(-e^t|x|,e^t|x|\right) \times \left(-e^{-t}|y|,e^{-t}|y|\right)\  \cap \ \diag\left(e^t, e^{-t}\right)\Lambda = \{ (0,0)\}$$
        \item Note that orders are invariant: 
        $$ \forall C, D \in \Pi(\Lambda), \quad C < D \iff y_C < y_D \iff g_t C < g_t D$$
        therefore $g_t (\Psi (i))$ is equal to $\Psi_{g_t \Lambda}(i)$, and the relation $\Psi(j+1) = a_j \Psi(j) + \Psi(j-1)$ is equivariant under flow of $g_t$.
    \end{enumerate}
\end{proof}

Are there more actions that preserves the equivariance relation of pivots and indices of $\Lambda$? The answer is yes, and we can find another couple of actions:

\begin{lemma}
    Let $V_x = \diag(1,-1)$ and $V_y = \diag(-1, 1)$ be the reflection along x and y axis, respectively. Then
    \begin{enumerate}
        \item $V_x \Lambda$ and $V_y \Lambda$ are unimodular bi-infinite lattices;
        \item Pivots are equivariant: $V_x \Pi(\Lambda) = \Pi(V_x \Lambda)$, $V_y \Pi(\Lambda) = \Pi(V_y \Lambda)$;
        \item The index sequence of $(\Lambda; A)$ is equal to the index sequences $(V_x \Lambda; V_x A)$ and $(V_y \Lambda; V_y A)$.
    \end{enumerate}
\end{lemma}
\begin{proof}
    \begin{enumerate}
    \item Obviously.
    \item Recall that a point $B = (x,y)$ is a pivot in $\Lambda$ iff $(-|x|,|x|) \times (-|y|,|y|) \cap \Lambda = \{ (0,0)\}$. All the reflections of $B$ remains on the vertices of the box of $\{\pm x, \pm y\}$.
    \item It is obvious that the index sequence of $(\Lambda; A)$ is equal to $(V_y\Lambda; V_y A)$, since we haven't made any assumptions on the signs of x coordinates of pivots. To prove that $V_x$ is also an equivariant action, note that $V_x = -V_y$ and $\Lambda = -\Lambda = \diag(-1,-1) \Lambda$, where $\diag(-1,-1)$ keeps the whole lattice invariant.
    \end{enumerate}
\end{proof}

We define the equivariant group $G$ as follows:
\begin{definition}
    The equivariant group $G$ is the subgroup of $\mathrm{GL}_2(\R)$ generated by these elements:
    $$
    \begin{bmatrix}
        1 & 0 \\ 0 & -1
    \end{bmatrix}, \quad
    \begin{bmatrix}
        -1 & 0 \\ 0 & 1
    \end{bmatrix}, \quad
    \begin{bmatrix}
        \exp(t) & 0 \\ 0 & \exp(-t)
    \end{bmatrix} (t \in \R)
    $$
\end{definition}
It's obvious that this group is isomorphic to $\left(\Z/2\Z\right)^2 \times \R$, and the action of $G$ keeps the index sequence of $(\Lambda; A)$ equivariant, i.e., for any $h \in G$, the index sequences of $(\Lambda; A)$ and $(g\Lambda, gA)$ are equivariant under an action of $g$.

It is possible to use the action of $G$ on $(\Lambda; A)$ to reduce the lattice into a certain canonical form to simplify the computation. We choose $h \in G$ such that $hA$ locates on the ray of $x+y=0 (x<0)$.
\begin{lemma}
    \label{lem7}
    For any point $A \in \R^2$ not on the axis, there exists a unique $h \in G$ such that $hA \in \{ (x,-x) | x<0\}$.
\end{lemma}
\begin{proof}
    Suppose $A = (x_0, y_0)$, then $h$ is uniquely determined as
    $$h = \begin{bmatrix}
        -\sgn(x_0) & 0 \\ 0 & \sgn(y_0)
    \end{bmatrix}
    g_t
    $$
    where $ t= \frac12 \left( \ln (|y_0|) - \ln(|x_0|) \right)$.
\end{proof}

Now consider the lattice $\Lambda$ with a pivot $A_0 = (x_0, y_0) = (-y_0, y_0)$ at the ray $x+y=0 (x<0)$. We aim to discover more specific information about pivots under this setting.

\begin{lemma}
    \label{lattiprop}
    The lattice $\Lambda$ has the following properties:
    \begin{enumerate}
        \item All lattice points lie on the lines $y = -x + {k}/{y_0} \ (k \in \Z)$;
        \item Both the next pivot of $A_0$, denoted by $A_1 = (x_1, y_1)$, and the precedent pivot of $A_0$, denoted by $A_{-1} = (x_{-1}, y_{-1})$, locate in the first quadrant $\{(x,y) | x<0, y>0\}$ and on the line $y = -x + 1/{y_0}$.
    \end{enumerate}
\end{lemma}

\begin{proof}
\begin{enumerate}
    \item Another basis vector $B = (x_B, y_B)$ must satisfy $| \det (A; B)| = 1$, so $y_0 |y_B+x_B| = 1$ and $y_B+x_B = \pm 1/y_0$. Thus $mA+nB$ lies on the line of $ y = -x + \pm n/y_0$.
    \item By the previous alternating principle, we know that they are both in the first quadrant, and since $\{A_0, A_{-1}\}$ and $\{A_0, A_1\}$ are two sets of bases, we deduce that $A_{-1}$ and $A_1$ are either on the line of $y=-x+1/y_0$ or $y=-x-1/y_0$. Because $\{y=-x-1/y_0\}$ does not pass the first quadrant, both $A_{-1}$ and $A_1$ must be on the line of $y=-x+1/y_0$.
\end{enumerate}
\end{proof}

It is also a similar case to restrict $A_0$ on another ray $y=x (x>0)$, where we obtain similar results. Under this setting, $h \in G$ is also unique.

\subsection{Establishing the Correspondence}
From the last section, we know how to obtain the index sequence giving a lattice $\Lambda$ and a starting point. Now, we prove that it's possible to reconstruct the lattice from the index sequence and a starting point. We will provide an explicit formula to show the relationship between these lattices and continued fractions.

\begin{lemma}
    \label{recall}
    Let $\{a_n\}\  (n \in \Z, a_n \in \N^+)$ be any positive integer sequence over $\Z$. Then there exists a unimodular bi-infinite lattice $\Lambda \subset \R^2$, denoted as $\Lambda ( \{ a_n \} )$, and a point $A_0$, denoted as $A_0 ( \{ a_n \} )$, such that
    \begin{enumerate}
        \item $A_0$ is a pivot point: $A_0 \in \Pi(\Lambda)$;
        \item $A_0$ is on the ray: $A_0 \in \{(x, -x)| x<0\}$;
        \item The index sequence $(\Lambda; A_0)$ is exactly $\{a_n\}$.
    \end{enumerate}
\end{lemma}

\begin{proof}
    For convenience, we define these two sequences using continued fraction:
    $$\alpha_n := [0; a_{n+1}, a_{n+2}, \cdots] = \frac{1}{a_{n+1}+\frac{1}{a_{n+2} + \frac{1}{\cdots}}}$$
    $$\beta_n := [0; a_{n-1}, a_{n-2}, \cdots] = \frac{1}{a_{n-1}+\frac{1}{a_{n-2} + \frac{1}{\cdots}}}$$
    It's easy to see that both $u_n$ and $v_n$ are irrational and lie in the open interval $(0,1)$.
    We set the vector $A_0$ as
    $$A_0 = \left(\frac{-1}{\sqrt{a_0+\beta_0+\alpha_0}}, \frac{1}{\sqrt{a_0+\beta_0+\alpha_0}}\right)$$
    and another point $A_{-1}$ as
    $$A_{-1} = \left(\frac{a_0+\alpha_0}{\sqrt{a_0+\beta_0+\alpha_0}}, \frac{\beta_0}{\sqrt{a_0+\beta_0+\alpha_0}}\right)$$
    and let $\Lambda = \Z A_0 + \Z A_{-1}$.

    It is directly verifiable that $|\det(A_0; A_{-1})| = 1$, so $\Lambda$ is unimodular.

    Also, $\beta_0$ and $a_0+\alpha_0$ are irrational numbers, suggesting that none of $mA+nA_{-1}$ lies on the axis except $0$, implying that $\Lambda$ is bi-infinite.
    
    Consider the box 
    $$C_0 := \left[-\frac{a_0+\alpha_0}{\sqrt{a_0+\beta_0+\alpha_0}}, \frac{a_0+\alpha_0}{\sqrt{a_0+\beta_0+\alpha_0}} \right] \times \left[\frac{-1}{\sqrt{a_0+\beta_0+\alpha_0}}, \frac{1}{\sqrt{a_0+\beta_0+\alpha_0}} \right]$$
    All lattice points in $C$ must satisfy $|x+y| \le \frac{a_0+\alpha_0+1}{\sqrt{a_0+\beta_0+\alpha_0}}$, so the component of $A_{-1}$ is in $\{-1,0,1\}$. A simple computation shows that only $\pm A_0, \pm A_{-1}, 0$ belong to the box, and $\pm A, \pm A_{-1}$ lie on the boundary, which concludes that $A_0, A_{-1}$ are two consecutive pivots. 

    Now let 
    $$A_1 = A_{-1} + a_0 A_0 = \left(\frac{\alpha_0}{\sqrt{a_0+\beta_0+\alpha_0}}, \frac{a_0+\beta_0}{\sqrt{a_0+\beta_0+\alpha_0}}\right)$$
    One can similarly (by analogy of above since the formulas are symmetric) deduce that  $A_0, A_{1}$ are two consecutive pivots. 

    Now, we find $A_2, A_3, \cdots$ by induction.

    Let 
    $$g_{1} = \diag \left( \sqrt{\frac{a_0+\beta_0}{\alpha_0}}, \sqrt{\frac{\alpha_0}{a_0+\beta_0}}\right)$$

    Then $g_1$ pulls $A_1$ on $y=x (x\ge 0)$. We have
    $$
    \begin{aligned}
        g_1 A_1 &= \left( \frac{\sqrt{\alpha_0(a_0+\beta_0)}}{\sqrt{a_0+\beta_0+\alpha_0}}, \frac{\sqrt{\alpha_0(a_0+\beta_0)}}{\sqrt{a_0+\beta_0+\alpha_0}} \right) \\
        &= \left( \frac{1}{\sqrt{\frac{1}{a_0+\beta_0}+\frac{1}{\alpha_0}}},  \frac{1}{\sqrt{\frac{1}{a_0+\beta_0}+\frac{1}{\alpha_0}}}\right) \\
        &= \left( \frac{1}{\sqrt{\beta_1 + a_1 + \alpha_1}},  \frac{1}{\sqrt{\beta_1 + a_1 + \alpha_1}}\right) \\
    \end{aligned}
    $$
    and
    $$
    \begin{aligned}
        g_1 A_0 &= \left( \frac{-\sqrt{(a_0+\beta_0)}}{\sqrt{\alpha_0(a_0+\beta_0+\alpha_0)}}, \frac{\sqrt{\alpha_0}}{\sqrt{(a_0+\beta_0)(a_0+\beta_0+\alpha_0)}} \right) \\
        &= \left( \frac{1}{\alpha_0} \frac{1}{\sqrt{\beta_1 + a_1 + \alpha_1}},  \frac{1}{(a_0+\beta_0)} \frac{1}{\sqrt{\beta_1 + a_1 + \alpha_1}} \right) \\
        &= \left(- \frac{a_1+\alpha_1}{\sqrt{\beta_1 + a_1 + \alpha_1}}, \frac{\beta_1}{\sqrt{\beta_1 + a_1 + \alpha_1}} \right) \\
    \end{aligned}
    $$

    Similarly set

    $$g_1 A_2 = \left(- \frac{\alpha_1}{\sqrt{\beta_1 + a_1 + \alpha_1}}, \frac{a_1+\beta_1}{\sqrt{\beta_1 + a_1 + \alpha_1}} \right)$$
    
    Since $0 < \beta_1, \alpha_1 < 1$, it is easily verifiable that 

    $$C_1 := \frac{1}{\sqrt{a_0+\beta_0+\alpha_0}} \left(\ \left[-1,1 \right] \times \left[ -a_1-\beta_1,\  a_1+\beta_1\right]\  \right)$$
    has no nonzero lattice points in the interior of it. Therefore $A_1, A_2$ are two consecutive pivots, and the index here is exactly $a_1$ because $g_1(A_2 - A_0 - a_1 A_1) = 0$.

    By induction, if we have (note that products are group actions on the left and indices come with decreasing order)
    $$
    \begin{aligned}
        \left(\prod_{k=2j-1}^{1} g_k\right) A_{2j-2} &= \left(- \frac{a_{2j-1}+\alpha_{2j-1}}{\sqrt{\beta_{2j-1} + a_{2j-1} + \alpha_{2j-1}}}, \frac{\beta_{2j-1}}{\sqrt{\beta_{2j-1} + a_{2j-1} + \alpha_{2j-1}}} \right) \\
        \left(\prod_{k=2j-1}^{1} g_k\right) A_{2j-1} &= \left(\frac{1}{\sqrt{\beta_{2j-1} + a_{2j-1} + \alpha_{2j-1}}}, \frac{1}{\sqrt{\beta_{2j-1} + a_{2j-1} + \alpha_{2j-1}}} \right) \\
    \end{aligned}
    $$
    and $A_{2j-2}, A_{2j-1}$ are two consecutive pivots, then set $A_{2j} = A_{2j-2} +a_{2j-1}A_{2j-1}$ and
    $$
    \begin{aligned}
        \left(\prod_{k=2j-1}^{1} g_k\right) A_{2j} &= \left(- \frac{\alpha_{2j-1}}{\sqrt{\beta_{2j-1} + a_{2j-1} + \alpha_{2j-1}}}, \frac{a_{2j-1}+\beta_{2j-1}}{\sqrt{\beta_{2j-1} + a_{2j-1} + \alpha_{2j-1}}} \right)
    \end{aligned}
    $$
    Then it's possible to show that $A_{2j-1}, A_{2j}$ are consecutive pivots with index $a_{2j-1}$ and set
    $$g_{2j} = \diag \left( \sqrt{\frac{a_{2j-1}+\beta_{2j-1}}{\alpha_{2j-1}}}, \sqrt{\frac{\alpha_{2j-1}}{a_{2j-1}+\beta_{2j-1}}} \right)$$

    then
    $$
    \begin{aligned}
        \left(\prod_{k=2j}^{1} g_k\right) A_{2j-1} &= \left(\frac{a_{2j}+\alpha_{2j}}{\sqrt{\beta_{2j} + a_{2j} + \alpha_{2j}}}, \frac{\beta_{2j}}{\sqrt{\beta_{2j} + a_{2j} + \alpha_{2j}}} \right) \\
        \left(\prod_{k=2j}^{1} g_k\right) A_{2j} &= \left(-\frac{1}{\sqrt{\beta_{2j} + a_{2j} + \alpha_{2j}}}, \frac{1}{\sqrt{\beta_{2j} + a_{2j} + \alpha_{2j}}} \right) \\
    \end{aligned}
    $$
    then $A_{2j+1} = A_{2j-1} +a_{2j}A_{2j}$ and
    $$
    \begin{aligned}
        \left(\prod_{k=2j}^{1} g_k\right) A_{2j+1} &= \left(\frac{\alpha_{2j}}{\sqrt{\beta_{2j} + a_{2j} + \alpha_{2j}}}, \frac{a_{2j}+\beta_{2j}}{\sqrt{\beta_{2j} + a_{2j} + \alpha_{2j}}} \right)
    \end{aligned}
    $$

    All the above formulas hold for all $j \in \N^+$. 

    Now the patterns are obvious and one can compute similarly for the $j \le 0$ part by analogy, producing the exact index sequence of $\{a_n\}$.
\end{proof}

At this point, we can establish one of our master theorems. 

\begin{theorem}
    \label{corre}
    There is a one-to-one correspondence between
    \begin{equation}
        \label{lat3}
        \left\{(\Lambda, A_0) \ \middle |  \  \Lambda \text{ is a unimodular bi-infinite lattice, } A_0 \in \{ y= -x | y>0 \} \right\}
    \end{equation}
    and
    \begin{equation}
        \label{seq}
        \left\{\ \{ a_n\} \ \middle| \  a_n \in \N^+, \forall n \in \Z \right\}
    \end{equation}
\end{theorem}

\begin{proof}
    The map from the former to latter, named $\mathscr{F}$, is done with Lemma \ref{indexing}, and the map from the latter to former, denoted by $\mathscr{G}$, is the above Lemma \ref{recall}.

    The proof of Lemma \ref{recall} already shows that $\mathscr{F} \circ \mathscr{G} = \mathrm{Id}$ in the space of index sequences. Now we only need to prove that $\mathscr{G}$ is a surjective map.

    For any lattice $(\Lambda, A_0)$ such that $ \Lambda $ is a unimodular bi-infinite lattice, and $A_0 \in \{ y= -x | y>0 \}$,
    we choose $A_{-1}$ such that $y_{A_{-1}} < y_{A_0}$ and $A_0, A_{-1}$ are two consecutive pivots. By Lemma \ref{lattiprop} we know that $A_{-1}$ belongs to both the first quadrant and the line $y = -x + 1/{y_{A_0}}$.
    
    Set 
    $$\beta_0 = \frac{y_{A_{-1}}}{ y_{A_0}}$$
    and 
    $$a_0 = \left\lfloor \frac{|x_{A_{-1}}|}{|x_{A_0}|}\right\rfloor, \quad \alpha_0 = \frac{|x_{A_{-1}}|}{|x_{A_0}|} - a_0$$
    then there exists a constant $k>0$ such that $A_0 = k(-1,1)$ and $A_{-1} = k(a_0+\alpha_0, \beta_0)$. By the unimodularity constraint and computing the determinant, we prove that
    $$
    k = \left(\det \begin{bmatrix}
        -1 & 1\\ a_0+\alpha_0 & \beta_0
    \end{bmatrix}\right)^{-1/2} = \frac{1}{\sqrt{\beta_0 + a_0 + \alpha_0}}
    $$
    Therefore, 
    $$A_0 = \left(\frac{-1}{\sqrt{a_0+\beta_0+\alpha_0}}, \frac{1}{\sqrt{a_0+\beta_0+\alpha_0}}\right)$$
    and
    $$B = \left(\frac{a_0+\alpha_0}{\sqrt{a_0+\beta_0+\alpha_0}}, \frac{\beta_0}{\sqrt{a_0+\beta_0+\alpha_0}}\right) $$
    It's worth noting that $a_0+\alpha_0 \notin \Q$ and $\beta_0 \notin \Q$. 

    Finally, we obtain sequence $\{a_n\}$ by expanding $\beta_0$ and $\alpha_0$:
    
    $$
    \begin{aligned}
        &a_0 = a_0 \\
        &[0; \ a_{-1}, a_{-2}, \cdots] := \beta_0\\
        &[0; \ a_1, a_2, \cdots] := \alpha_0\\
    \end{aligned}
    $$

    This index sequence belongs to the preimage of $(\Lambda, A_0)$ under $\mathscr{G}$, which completes the proof that $\mathscr{G}$ is surjective.
\end{proof}

\subsection{A More General Correspondence}

A natural question arises as we pose Theorem \ref{corre}: Can we consider a more general setting under which pivots are not fixed on a ray? The answer is that the additional constraints on the ray can be replaced by certain equivalence of group actions. Now we pose the general correspondence as follows:

\begin{theorem}
    \label{corre2}
    There is a one-to-one correspondence between
    \begin{equation}
        \label{lat}
        \left\{(\Lambda, A_0) \ \middle |  \  \Lambda \text{ is a unimodular bi-infinite lattice} \right\} \Big/ \sim_1
    \end{equation}
    and
    \begin{equation}
        \label{seq2}
        \left\{\ \{ a_n\} \ \middle| \  a_n \in \N^+, \forall n \in \Z \right\}
    \end{equation}
    Where $\sim_1$ literally means belonging to the same coset under the action of $G$, or written as
    $$(\Lambda_1, A_{01}) \sim_1 (\Lambda_2, A_{02}) \ \iff \ \exists g \in G\ \  g (\Lambda_1, A_{01}) = (\Lambda_2, A_{02})$$
\end{theorem}

\begin{proof}
    This is directly obtained by Lemma \ref{lem7} that 
    $$
    \begin{aligned}
        &\left\{(\Lambda, A_0) \ \middle |  \  \Lambda \text{ is a unimodular bi-infinite lattice} \right\} \Big/ \sim_1 \\
        &\cong \left\{(\Lambda, A_0) \ \middle |  \  \Lambda \text{ is a unimodular bi-infinite lattice, } A_0 \in \{ y= -x | y>0 \} \right\}
    \end{aligned}
    $$
    And the rest part is obtained with Theorem \ref{corre}.
\end{proof}

The above lemma still treats $A_0$ as a special starting point for construction. What if we start from an arbitrary pivot? We show that is still possible to eliminate $A_0$ by quotient of another group action:

\begin{theorem}
    \label{corre3}
    There is a one-to-one correspondence between
    \begin{equation}
        \label{lat2}
        \left\{\Lambda \ \middle |  \  \Lambda \text{ is a unimodular bi-infinite lattice} \right\} \Big/ \sim_1
    \end{equation}
    and
    \begin{equation}
        \label{seq3}
        \left\{\ \{ a_n\} \ \middle| \  a_n \in \N^+, \forall n \in \Z \right\} \Big/ \sim_2
    \end{equation}
    Where $\sim_1$ is defined the same as above:
    $$\Lambda_1 \sim_1 \Lambda_2 \ \iff \ \exists g \in G\ \  g\Lambda_1 = \Lambda_2$$
    and $\sim_2$ is defined as the equivalence under shift of sequences:
    $$\{a_n\} \sim_2 \{b_n\} \ \iff \ \exists j \in \Z \quad \forall n \in \Z \quad a_n=b_{n+j} $$
\end{theorem}

\begin{proof}
    First, we recall a trivial fact:

    For any objects $\mathcal{P}$, $\mathcal{Q}$ equivalence $\sim$ on $\mathcal{P}$, $\sim'$ on $\mathcal{Q}$ and isomorphism $\mathscr{R}: \mathcal{P} \to \mathcal{Q}$, if $\mathscr{R}$ is isomorphism on every coset (i.e., $\mathscr{R}$ restricted on $\Tilde{a}$, the coset of $a \in \mathcal{P}$, gives isomorphism between $\tilde{a}$ and $\widetilde{\mathscr{R}(a)}'$ ), then $R$ induces isomorphism between $\mathcal{P}/\sim$ and $\mathcal{Q}/\sim'$.

    From Theorem \ref{corre2}, we consider equivalence in
    $$\left\{(\Lambda, A_0) \ \middle |  \  \Lambda \text{ is a unimodular bi-infinite lattice} \right\} \Big/ \sim_1$$
    where two objects are equivalent iff they have the same lattice, no matter where the pivot $A_0$ is. We show that Theorem \ref{corre2} provides isomorphism on every equivalence class of it.

    Consider an equivalence class of $(\Lambda, A_0) / \sim_1$. The pivot in the equivalence class can be chosen arbitrarily, so the equivalence class looks like:

    $$
    \left\{(\Lambda, A_j) / \sim_1 \ \middle| \ j \in \Z \right\}
    $$

    where $A_j = \Psi(j) \in \Pi(\Lambda)$ and $\Psi: \Z \to \Pi(\Lambda)$ is the order isomorphism given in Lemma \ref{orderiso}. 

    Suppose that the index sequence of $(\Lambda, A_0) / \sim_1$ is $\{a_n\}$. From Lemma \ref{indexing} we know that $A_j = \Psi(j)$ corresponds to index $a_j$ in this lattice, so $(\Lambda, A_j) / \sim_1$ has index sequence $\{a_{n-j}\}$. 

    Unsurprisingly, all $\{a_{n-j}\} \ (j \in \Z)$ exactly forms the equivalence class of $\{a_n\}$ under $\sim_2$. And so we establish that correspondence of Theorem \ref{corre2} is isomorphism on every equivalence class, concluding the proof.
\end{proof}

Theorem \ref{corre2} and \ref{corre3} are especially useful in both number theory and dynamical systems. We will use them later to prove important properties on the spectra of numbers and lattices.

\subsection{Projective Invariance of Mono-infinite Lattices}

The set of bi-infinite unimodular lattices is interesting, but that does not cover all cases we are interested in. 

\begin{definition}
    A unimodular lattice $\Lambda \subset \R^2$ is called mono-infinite if $\Lambda$ has nonzero points on the $x$ axis, but not on the $y$ axis.
\end{definition}

We can similarly define pivots and indices, and indeed, this mono-infinite case is deeply related to Diophantine approximation of numbers.

Instead of defining pivots at first, we derive the canonical form of mono-infinite lattices.

\begin{lemma}
    Let $G_1 = \diag (\exp(t), \exp(-t)) \ (t \in \R)$ be the diagonal group. If $\Lambda$ is a unimodular mono-infinite lattice, then in the coset $\{g_t \Lambda\}$, there exists a unique $t$ such that $g_t \Lambda$ is generated by $(1,0)$ and $(-\alpha, 1) \ (\alpha \in (\R \setminus \Q) \cap (0,1) )$.
\end{lemma}

\begin{proof}
    Let $a > 0$ be the minimum number such that $(a,0) \in \Lambda$, and let $t = -\ln (a)$, then $(1,0)$ is one basis vector of $g_t \Lambda$. By computing the determinant, we know that another basis vector has $y$ coordinate of $\pm 1$, and all points in $g_t \Lambda$ have integer $y$ coordinates.

    Now consider points on the line $y=1$, there exists a point whose $x$ coordinate falls in $(-1, 0]$. But since there are no points on the $y$ axis except $(0,0)$, the $x$ coordinate must fall in $(-1, 0)$. So the next point $(-\alpha, 1)$ is obtained. Obviously, $\alpha$ is irrational.
\end{proof}

We can define pivots in this case, but the definitions are a bit different.

\begin{definition}
    For a unimodular mono-infinite lattice $\Lambda$ with basis $(1,0)$ and $(-\alpha, 1)$, For a point $A = (x, y) \in \Lambda$, if the rectangle $C = [-|x|, |x|] \times [-|y|, |y|]$ contains only the lattice points of $\{ A, 0, -A\}$, then $A$ is called a pivot. We further add two special pivots: $(1,0)$ and $(-\alpha, 1)$ are pivots. The pivot set is still denoted by $\Pi(\Lambda)$.
\end{definition}

It's trivial to check that this definition, excluding two special pivots, coincide with the bi-infinite case of pivots.

(Note: By forcing two pivots, we continue to avoid the concept of "stable pivots" shown in \cite{Knitter_Complexity_Domain_Approximation}. This coercion is actually advantageous in preserving the alternating principle that two consecutive pivots fall in different quadrants.)

Still, the order of pivots is similarly defined.

\begin{definition}
    Pivots are ordered by their absolute value of $y$ coordinates from small to large. Specifically, we define $A_{-1} = (1,0)$ and $A_{0} = (-\alpha, 1)$, and $A_1, A_2, \cdots$ are remaining pivots from small to large. 
\end{definition}

\begin{lemma}
    If $\Lambda$ is mono-infinite, the pivot set $(\Pi(\Lambda), <)$ is order isomorphic to $\N$.
\end{lemma}

\begin{proof}
    The proof is an analogy of Lemma \ref{orderiso}.

    We construct the map $\Psi: (\N \cup \{-1\}) \to \Pi(\Lambda)$ by induction: $\Psi(-1) = (1,0), \Psi(0) = (-\alpha, 1)$, and if $\Psi(i) = (x_i, y_i) (i \ge 0)$ is defined, we take $C_{(y,i)} := [-|x_i|, |x_i|] \times [0,y]$, and let 
    $$y^* = \sup_{\mathrm{Int} \left(C_{(y,i)}\right) \cap \Lambda = \emptyset } y $$

    It can be shown similarly that the upper edge of $C_{(y^*,i)}$, $(-|x_i|, |x_i|) \times \{y^*\}$, contains a unique lattice point. (By analogy we show that it contains a lattice point; if $|x_i| \le 1/2$, then it contains at most one lattice point; if $|x_i| > 1/2$, this only happens when $i=0$ and $1/2 < \alpha < 1$, and $(1-\alpha, 1)$ is the unique point)

    Let $\Psi(i+1)$ be that point, and $\Psi$ is defined by induction. The proof that $\Psi$ is surjective is no different from above. Thus we can prove that $\Psi$ gives the isomorphism from $(\N \cup \{-1\})$ to $\Phi(\Lambda)$. Also, $(\N \cup \{-1\})$ is order isomorphic to $\N$ by a shift.
\end{proof}

The indices and alternating principle can also be established. We use the convention that $A_{i} = \Psi(i-1) (i \in \N)$ are the pivots, and $a_i (i \in \N)$ satisfies that $A_{i+1} - A_{i-1} = a_i A_i$.  Under this convention, all $A_{2j} (j \in \N)$ are located in the second quadrant, and $A_{2j-1} (j \in \N)$ are located in the first quadrant.

The correspondence theorem is stated as follows:

\begin{theorem}
    \label{corre4}
    There is a one-to-one correspondence between
    \begin{equation}
        \label{lat1}
        \left\{\Lambda \ \middle |  \  \Lambda \text{ is a unimodular mono-infinite lattice} \right\} \Big/ \sim_0
    \end{equation}
    and
    \begin{equation}
        \label{seq1}
        \left\{\ \{ a_n\} \ \middle| \  a_n \in \N^+, \forall n \in \N \right\} 
    \end{equation}
    Where $\sim_0$ is defined as
    $$\Lambda_1 \sim_0 \Lambda_2 \ \iff \ \exists g = \diag(e^{t}, e^{-t}), \ \  g\Lambda_1 = \Lambda_2$$
\end{theorem}

Note that this equivalence does not involve the action of $\diag(-1,1)$ because we forced $(1,0)$ and $(-\alpha, 1)$ to be pivots, but $(1-\alpha, 1)$ is not necessarily a pivot.

\begin{proof}
    We already know how to compute indices, and we denote by $\mathscr{F}$ the map from lattice to indices. Now we recover the lattice:

    For simplicity, we still use the notation that
    $$\alpha_n := [0; a_{n+1}, a_{n+2}, \cdots] = \frac{1}{a_{n+1}+\frac{1}{a_{n+2} + \frac{1}{\cdots}}}$$
    But for $\beta$ a bit different:
    $$\beta_n := [0; a_{n-1}, a_{n-2}, \cdots, a_0] \ (n \ge 1); \quad \beta_0 = 0$$
    
    Suppose that $\{a_i\}_{i \ge 0}$ is a index sequence. We 
    let $\Lambda$ be the lattice with basis $\{ (1,0), (-\alpha_{-1}, 1)\}$. We show that $\Lambda$ has the index sequence of $\{a_i\}_{i \ge 0}$.

    Let $h = \diag\left( \alpha_{-1} ^{-1/2}, \alpha_{-1} ^{1/2}\right)$, then
    $$h A_0 = \left(\frac{-1}{\sqrt{a_0+\beta_0+\alpha_0}}, \frac{1}{\sqrt{a_0+\beta_0+\alpha_0}} \right) \quad \text{(note that } \beta_0 = 0 \text{)}$$
    $$h A_{-1} = \left(\frac{-1}{\sqrt{\alpha_{-1}}}, \ 0 \right) =  \left(\frac{a_0+\alpha_0}{\sqrt{a_0+\beta_0+\alpha_0}},\  0\right)$$
    and $$hA_1 = \left(\frac{\alpha_0}{\sqrt{a_0+\beta_0+\alpha_0}}, \frac{a_0 + \beta_0}{\sqrt{a_0+\beta_0+\alpha_0}} \right)$$

    The rest ($A_2, A_3, \cdots$) can be computed similarly as above Lemma \ref{recall}.
\end{proof}

\subsection{Importance in Number Theory}

Pivots not only corresponds to indices, but also relate closely to the approximant of the second kind.

Set $\alpha$ an irrational in $(0,1)$. Recall that for $k ,l \in \N^+$, $(l/k)$ is called a $\alpha$'s best approximant of second kind if $|k \alpha - l| < |m \alpha - n|$ holds for all $m, n \in \N^+, m < k$.

The following theorem tells us that any pivot not on the axis of a mono-infinite lattice has a best approximant of second kind behind it.
\begin{theorem}
    If $\Lambda = \Z (1,0) + \Z (-\alpha, 1)$ is a mono-infinite lattice, where $\alpha$ is an irrational in $(0,1)$, and let its pivot sequence be $A_{-1} = (1,0), \ A_0 = (-\alpha, 1),\  A_1, A_2, \cdots$. Then every $A_j (j \ge 0)$ induces a best approximant of second kind. Specifically, if $A_j = p(1,0) + q(-\alpha, 1)$, then $p/q$ is a best approximant of second kind, and vice versa.
\end{theorem}

\begin{proof}
    If $A_j = p(1,0) + q(-\alpha, 1) = (p-q\alpha, q)$ is a pivot, then there exists no points in the rectangle of $C = [-|p-q\alpha|, |p-q\alpha|] \times [0, q]$ except $(0,0)$ and $(p-q\alpha, q)$. therefore for any $m, n \in \N^+$, if $n<q$, then $(m-n\alpha, n)$ does not belong to the rectangle $C$, so $|m-n\alpha| > |p-q\alpha|$.

    Vice versa, if $p/q$ is a best approximant of second kind, then consider $(p-q\alpha, q)$ and rectangle $C = [-|p-q\alpha|, |p-q\alpha|] \times [0, q]$. Any point $(m-n\alpha, n)$ in $C$ except $(0,0)$ and $(p-q\alpha, q)$ must satisfy $n < q$, and since $|m-n\alpha| > |p-q\alpha|$, the point does not exist. And so we have proved the correspondence.
\end{proof}

It's worth noting that we can also write $|p-q\alpha|$ as $\lVert q\alpha \rVert$, where $\lVert x \rVert$ means the distance from $x$ to integers $\Z$.

How about bi-infinite lattices? We show that pivots actually relate to the best approximant of two irrationals.

Consider a pair $(\Lambda, A_0)$ where $\Lambda$ is a unimodular bi-infinite lattice, and $A_0$ is a pivot on $y=-x \ (x<0)$, whose index sequence is $\{a_n\}$. By the correspondence of Theorem \ref{corre} we know that
$$A_0 = \left(\frac{-1}{\sqrt{a_0+\beta_0+\alpha_0}}, \frac{1}{\sqrt{a_0+\beta_0+\alpha_0}}\right)$$
$$A_{-1} = \left(\frac{a_0+\alpha_0}{\sqrt{a_0+\beta_0+\alpha_0}}, \frac{\beta_0}{\sqrt{a_0+\beta_0+\alpha_0}}\right)$$

\begin{theorem}
    \label{baosk}
    Pivots of $\Lambda$ correspond to best approximants of second kind of either $\alpha_{-1}$ or $\beta_{0}$. Specifically, $A_k (k\ge 0)$ corresponds to best approximants of second kind of $\alpha_{-1}$, and $A_k (k< 0)$ corresponds to that of $\beta_{0}$.
\end{theorem}

\begin{proof}
We set
$$
\begin{aligned}
    g &= \diag\left( \frac{1}{\sqrt{a_0 + \alpha_0}}, \ \sqrt{a_0 + \alpha_0}\right) \\
    gA_0 &= \left(\frac{-1}{\sqrt{(a_0 + \alpha_0)(a_0+\beta_0+\alpha_0)}}, \  \frac{\sqrt{a_0 + \alpha_0}}{\sqrt{a_0+\beta_0+\alpha_0}}\right) \\
    &= \sqrt{\frac{1}{1+\frac{\beta_0}{a_0+\alpha_0}}} \left(\frac{-1}{a_0+\alpha_0}, 1\right) \\
    &= \sqrt{\frac{1}{1+\alpha_{-1} \beta_0}} \left(-\alpha_{-1}, 1\right) \\
    gA_{-1} &= \left(\frac{\sqrt{a_0 + \alpha_0}}{\sqrt{a_0+\beta_0+\alpha_0}}, \  \frac{\beta_0\sqrt{a_0 + \alpha_0}}{\sqrt{a_0+\beta_0+\alpha_0}}\right) \\
    &= \sqrt{\frac{1}{1+\alpha_{-1} \beta_0}} \left(1, \beta_0\right) \\
\end{aligned}
$$
We already know that $gA_0$ and $gA_{-1}$ form a basis of $\Lambda$.

Consider pivot $A_k = (x_k, y_k) \in \Lambda$ or $gA_k = (gx_k, gy_k) \in g\Lambda$. As $gA_k$ is a pivot, there are no points inside rectangle $gC_k = [-|gx_k|, |gx_k|] \times [-|gy_k|, |gy_k|]$ except $gA_k, 0, -gA_k$. We further assume that $A_k = qA_0+pA_{-1}, p, q \in \Z$. Then
$$ 
gA_k = \sqrt{\frac{1}{1+\alpha_{-1} \beta_0}} \left(p-q\alpha_{-1}, p\beta_0+q\right)
$$

If we consider $k\ge 0$, then $|p-q\alpha_{-1}|$ no larger than $|\alpha_{-1}|$, but $|p\beta_0+q|$ is large. With our convention that $p\beta_0+q$ is always positive, both $p$ and $q$ are nonnegative.

We show that $p/q$ is a second kind best approximant. Since there are no other points inside the rectangle of $C_k$, we show that
for any $m, n \in \N^+$, if $n<q$, then $\sqrt{\frac{1}{1+\alpha_{-1} \beta_0}}(m-n\alpha_{-1}, m\beta_0+n)$ does not belong to the rectangle $C_k$.

For $n<q$, we show that $m-n\alpha_{-1} \notin [-|p-q\alpha_{-1}|, |p-q\alpha_{-1}|]$. Suppose otherwise, then $m\beta_0+n \notin [0, p\beta_0+q]$. This indicates that $m>p$, and $m-n\alpha_{-1} \ge (1+\alpha_{-1}) + p-q\alpha_{-1} \ge (1+\alpha_{-1}) \ge 1 > \alpha_{-1} > |p-q\alpha_{-1}|$, a contradiction occurs.

Similarly, we can consider pivots of $k < 0$, which corresponds to best approximants of $\beta_0$.

\end{proof}

\subsection{The Log-systole Function: Four Spectra, One Function}
In this section, we adapt the concept of "systole" in geometry to study lattices. We equip $\R^2$ with the $\ell^\infty$ norm, and denote the norm by $\lVert \_ \rVert_\infty$. This notation differs from the previous $\lVert \_ \rVert$ by a $\infty$, where the latter means the distance to $\Z$.

\begin{definition}
    Given $\Lambda \subset \R^2$ a lattice, the systole of $\Lambda$, $\mathrm{sys}(\Lambda)$, is defined as the smallest $\ell^\infty$ norm of lattice points, or,
    $$\mathrm{sys}(\Lambda) = \min_{A \in \Lambda, A \ne (0,0)} \lVert A \rVert_\infty$$
\end{definition}

\begin{lemma}
    \label{systole}
    For any unimodular $\Lambda$ we have
    \begin{enumerate}
        \item $\mathrm{sys}(\Lambda)$ is well-defined;
        \item $\mathrm{sys}  (\Lambda) = \sup \left\{r\ge 0 \ | \ [-r,r]\times [-r,r] \cap \Lambda = (0,0)\right\}$;
        \item $0 < \mathrm{sys}(\Lambda) \le 1$.
        \item If $\Lambda$ is also bi-infinite, then $\lVert A \rVert_\infty = \mathrm{sys}  (\Lambda)$ implies that $A$ is a pivot.
    \end{enumerate}
\end{lemma}

\begin{proof}
    \begin{enumerate}
        \item It's obvious that $\Lambda$ is a discrete set, under which $\lVert A \rVert_\infty$ has a minimum.
        \item Suppose that $A$ satisfies $\lVert A \rVert_\infty = \mathrm{sys}  (\Lambda)$. If we take the square ($\ell^\infty$ ball) $C = [-\lVert A \rVert_\infty, \lVert A \rVert_\infty] \times [-\lVert A \rVert_\infty, \lVert A \rVert_\infty]$, then $A$ is on the border of $C$, and there are no points inside $C$ except $(0,0)$. Therefore for any $\ell^\infty$ ball $C(0,r)$, $r < \lVert A \rVert_\infty$ if and only if $C(0,r) \cap \Lambda = \{(0,0)\}$. Therefore
        $$\sup \ \left\{ r \ | \ C(0,r) \cap \Lambda = \{(0,0)\}\right\} = \lVert A \rVert_\infty = \mathrm{sys}  (\Lambda)$$
        \item By Minkowski convex body theorem, the area of $C(0,r)$ is no more than $4$, so $4r^2 \le 4$ and $r \le 1$. Because $(0,0)$ is a discrete point in $\Lambda$, $\mathrm{sys}  (\Lambda) > 0$.
        \item By directly applying the definition of pivot.
    \end{enumerate}
\end{proof}

Now consider a function called the log-systole function:

\begin{definition}
    Given $\Lambda \subset \R^2$ a lattice, and $g_t = \diag (e^t, e^{-t})$ the diagonal action. The log-systole function of $\Lambda$, denoted by $W(t; \Lambda) : \R \to \R$, is defined as
    $$W(t; \Lambda) := \ln (\mathrm{sys}(g_t \Lambda))$$
\end{definition}

\begin{lemma}
    \label{sysfun}
    The function $W(t; \Lambda)$ for $t$ is a piecewise linear function with slope $\pm 1$. 
\end{lemma}

\begin{proof}
For simplicity, we denote $\mathrm{sys}(g_t \Lambda)$ by $u_t$. We take the square $C(0, u_t)$. There are points either on the horizontal border $\{u_t, u_t\} \times [-u_t, u_t]$ or the vertical ones $[-u_t, u_t] \times \{-u_t, u_t\}$. 
\begin{itemize}
    \item If there are points only on the horizontal border, suppose $g_t A = (x,y)$ is on the border where WLOG $x>0$. then $x = u_t = e^t x_0$ and $|y| < u_t$.
    $$
    \frac{\partial}{\partial t} \left( W(t; \Lambda) \right) = \frac{\partial}{\partial t} \left( \ln (e^t x_0 ) \right) = \frac{\partial}{\partial t} \left( t+\ln x_0 \right) = 1
    $$
    \item Similarly, if there are points only on the vertical border, then  $\frac{\partial}{\partial t} \left( W(t; \Lambda) \right) = -1$.
    \item If there are points on both the horizontal and the vertical border:
    \begin{itemize}
        \item If points are located at the corner $(u_t, u_t)$ or $(-u_t, u_t)$, then $W(t; \Lambda)$ admits a minimal point here.
        \item If points are not located at the corner, then both $\{u_t, u_t\} \times (-u_t, u_t)$ and $(u_t, u_t) \times\{-u_t, u_t\}$ contains points. In this case, $W_t$ achieves a maximal point.
    \end{itemize}
\end{itemize}
\end{proof}

Later, we will relate the minimal and maximal points to the spectra of lattices.

\begin{lemma}
    Let $\Lambda$ be a bi-infinite unimodular lattice. Then
    \begin{enumerate}
        \item The minimal points of $W(\_; \Lambda)$ correspond to pivots.
        \item The maximal points of $W(\_; \Lambda)$ correspond to the box of two consecutive pivots.
    \end{enumerate}
\end{lemma}

\begin{proof}
    The results are a combination of Lemma \ref{systole} and \ref{sysfun}.
\end{proof}

It's therefore interesting to compute the value of the minimal and the maximal points. The formulas are called Perron's formula.

\begin{lemma}
    \label{compute}
    Let $\Lambda$ be a unimodular bi-infinite lattice, and $A_0$ a pivot. The sequence of pivots is denoted by $\cdots, A_{-1}, A_0, A_1, \cdots$, with corresponding indices ${a_n} (n \in \Z)$. Then:
    \begin{enumerate}
        \item The local minimal point of $W(\_; \Lambda)$ at $A_j$ is 
        $$\ln \left( \frac{1}{\sqrt{a_j +\beta_j + \alpha_j}}\right)$$
        \item The local maximal point of $W(\_; \Lambda)$ between $A_{j-1}$ and $A_j$ is
        $$\ln \left( \frac{1}{\sqrt{1 +\beta_j\alpha_{j-1}}}\right)$$
    \end{enumerate}
\end{lemma}

\begin{proof}
    Computations follow Lemma \ref{recall} and Theorem \ref{baosk}. \footnote{Based on the procedure of Lemma \ref{recall}, we can also obtain a side-product: an algorithm to efficiently compute $W(t; \Lambda)$ for a given $t \in \R$, in time complexity $O(|t|)$ and space complexity $O(1)$. Since it's not directly related to our topic, this algorithm is left for the reader as an exercise.}
\end{proof}


And finally we have the theorem that corresponds all four types of spectra using a single function:
\begin{theorem}
    \label{fourinone}
    The four spectra are obtained through one log-systole function in the following way:
    \begin{enumerate}
        \item Lagrange spectrum: 
        $$\mathrm{L} = \left\{\exp(-2\liminf_t W(t; \Lambda))\middle| \Lambda \text{ is a unimodular bi-infinite lattice}\right\} $$
        \item Markov spectrum: $$\mathrm{M} = \left\{\exp(-2\inf_t W(t; \Lambda))\middle| \Lambda \text{ is a unimodular bi-infinite lattice}\right\} $$
        \item Dirichlet spectrum: $$\mathrm{D} = \left\{\exp(2\limsup_t W(t; \Lambda))\middle| \Lambda \text{ is a unimodular bi-infinite lattice}\right\} $$
        \item Mordell-Gruber Spectrum: $$\mathrm{MG}_2 = \left\{\exp(2\sup_t W(t; \Lambda))\middle| \Lambda \text{ is a unimodular bi-infinite lattice}\right\}$$
    \end{enumerate}
\end{theorem}

\begin{proof}
    \ \\
    \begin{enumerate}
        \item Lagrange spectrum:
        Using the last lemma, Note that
        $$
        \begin{aligned}
            \exp \left(-2\liminf_t W(t; \Lambda)\right) &= 
            \exp \left(-2\liminf_{t_j \text{ is local minimal point}} W(t_j; \Lambda)\right) \\
            & (\text{the declension to } \limsup \text{ is because } e^{-2x} \text{ is decreasing for } x) \\
            &= \limsup_{j \in \Z} (a_j +\beta_j + \alpha_j) \quad (\{a_n\} \text{ is a bi-infinite sequence})
        \end{aligned}
        $$
        The rest part is well-known: see \cite{Cusick_Flahive_Markoff_Lagrange} for a proof. 
        \item Markov spectrum: recall that
        $$
        \mathrm{M} := \left\{ \left(\inf_{(x,y) \in \Z^2 \setminus \{(0,0)\}} |f(x,y)|\right)^{-1} \middle| \Delta_f = 1\right\}
        $$
        where $f(x,y) = ax^2+bxy+cy^2$. Reconsider $f$ as a matrix form:
        $$f = [x, y]\begin{bmatrix}
            a & b/2 \\ b/2 & c
        \end{bmatrix}[x, y]^{\mathrm{T}}$$
        and rewrite it by simultaneous transformation: There exists $\bm{A} \in \mathbf{SL}_2(\R)$ such that
        $$f = [x, y] \bm{A} \begin{bmatrix}
            0 & 1/2 \\ 1/2 & 0
        \end{bmatrix}\bm{A}^{\mathrm{T}}[x, y]^{\mathrm{T}}$$
        Reconsider $[x, y] \bm{A} \ ((x,y) \in \Z^2)$ as a unimodular lattice, then
        $$
        \begin{aligned}
            \mathrm{M} &= \left\{ \left(\inf_{P = (x_P,y_P) \in \Lambda \setminus \{(0,0)\}} |x_Py_P|\right)^{-1} \middle| \Lambda \text{ unimodular}\right\} \\
            &= \left\{ \left(\inf_{t \in \R} \exp(2W (t; \Lambda)) \right)^{-1} \middle| \Lambda \text{ unimodular}\right\} \\
            &= \left\{ \exp(-2 \inf_{t \in \R} W (t; \Lambda)) \middle| \Lambda \text{ unimodular}\right\} \\
        \end{aligned}
        $$
        If $\Lambda$ is not bi-infinite, then $\inf_{t \in \R} W(t; \Lambda) = -\infty$, and $\infty$ is usually considered as an element of the Markov spectrum. So
        $$
        \begin{aligned}
            \mathrm{M} &= \left\{ \exp(-2 \inf_{t \in \R} W (t; \Lambda)) \middle| \Lambda \text{ unimodular and bi-infinite}\right\} \\
        \end{aligned}
        $$
        \item Dirichlet spectrum: Note that
        $$
        \begin{aligned}
            \exp(2\limsup_t W(t; \Lambda)) &= 
            \exp \left(2\limsup_{t_j \text{ is local maximal point}} W(t_j; \Lambda)\right) \\
            &= \limsup_{j \in \Z} \left(\frac{1}{1 +\beta_j\alpha_{j-1}} \right) \quad (\{a_n\} \text{ is a bi-infinite sequence})
        \end{aligned}
        $$
        For the rest part, see \cite{Divis_Novak_Diophantine_Approximations} for a proof.
        \item Mordell-Gruber spectrum: Note that
        $$
        \begin{aligned}
            \exp(2\sup_t W(t; \Lambda)) &= \limsup_t \ (\mathrm{sys}(g_t \Lambda))^2 \\
            &= \sup_t \sup_{[-x,x] \times [-x,x] \text{ admissible in } g_t \Lambda} (x^2) \\
            &= \sup_{(t\in \R,x>0): [-e^t x, e^t x] \times [-e^{-t}x, e^{-t}x] \text{ admissible in }\Lambda} (x^2) \\
            &= \kappa(\Lambda)
        \end{aligned}
        $$
        If a lattice is not bi-infinite, the Mordell constant is trivial: $\kappa(\Lambda) = 1$. Thus it's sufficient to discuss the bi-infinite case:
        $$\mathrm{MG}_2 = \left\{\exp(2\sup_t W(t; \Lambda))\middle| \Lambda \text{ is a unimodular bi-infinite lattice}\right\} \cup \{1\}$$
        Also, a Mordell constant of $1$ is indeed achievable by bi-infinite lattice. This is achieved via the index sequence of $a_n = |n| +1$. So
        $$\mathrm{MG}_2 = \left\{\exp(2\sup_t W(t; \Lambda))\middle| \Lambda \text{ is a unimodular bi-infinite lattice}\right\}$$
    \end{enumerate}
\end{proof}

\subsection{Projective Invariance of All Planar Unimodular Lattices}
We have so long discussed about projective invariance of bi-infinite and mono-infinite lattices, but how about general lattices? 

\begin{lemma}
    For any general planar unimodular lattice $\Lambda$, the log-systole function $W(t; \Lambda)$ admits a local maximal point.
\end{lemma}
\begin{proof}
    This is obvious, since we have already discussed mono-infinite and bi-infinite lattices. For the finite case, there is a $t$ such that $g_t \Lambda$ is generated by $(1,0) and (-\alpha, 1)$ where $\alpha \in \Q$. The $t$ here is a local maximum of $W(t; \Lambda)$.
\end{proof}
At a local maximal point of the log-systole function $W(t; \Lambda)$, the lattice $g_t \Lambda$ has certain structure: If there is some $t\in \R$ such that $W(t; \Lambda)$ is a local maximum, we consider the square $[-w, w]^2$ where $w = W(t; \Lambda)$. There are at least 2 linear independent points on the upper and right edges respectively of $w$ (otherwise $W$ is differentiable with derivative $\pm 1$, contradicting with the local maximum). Suppose the two points are $A$ and $B$, where $A$ is on the upper side of the square.
\begin{lemma}
    \label{arbitraryreconstruction}
    For $A, B$ we have:
    \begin{itemize}
        \item $A$ and $B$ form a basis. 
        \item the formulas are either:
        $$A = \frac{(-\alpha, 1)}{\sqrt{1+\alpha\beta}}, \quad B = \frac{(1, \beta)}{\sqrt{1+\alpha\beta}}$$
        or:
        $$A = \frac{(\alpha, 1)}{\sqrt{1+\alpha\beta}}, \quad B = \frac{(-1, \beta)}{\sqrt{1+\alpha\beta}}$$
        where
        $$\alpha, \beta \in [0,1], \quad w = \frac{1}{\sqrt{1+\alpha\beta}}$$
    \end{itemize}
\end{lemma}
\begin{proof}
    \begin{itemize}
        \item First consider the triangle formed by $(0,0)$, $A$ and $B$. The triangle has no points inside and on its edges, so by Pick's theorem it has an area of $1/2$, which implies that $A$ and $B$ form a basis. 
        \item Note that $A$ and $B$ are not in the same quadrant, otherwise $A-B$ has a smaller $\ell^\infty$ distance. so it must be one of the two cases: $A = w(-\alpha, 1), B=(1, \beta)$ or $A = w(\alpha, 1), B=(-1, \beta)$. The value of $w$ is obtained by the restriction of covolume.
    \end{itemize}
\end{proof}

How do we obtain corresponding sequences of lattices? A simple way is to compute the continued fraction expansions of $\alpha$ and $\beta$, and stick them head-to-head. Namely, let
$$
\begin{aligned}\relax
    [0;\  a_0,\  a_1,\  \cdots] &= \alpha \\
    [0;\  a_{-1},\  a_{-2},\  \cdots] &= \beta
\end{aligned}
$$
and let $a = (\cdots,\  a_{-2},\  a_1; \ a_0,\  a_1,\  \cdots)$, where each $a_i \in \N^+ \cup \{ \infty \}$. 

\begin{definition}
    \label{expansionseq}
    We call a sequence $(a_0,\  a_1,\  \cdots)$ an "expansion sequence" if:
    \begin{itemize}
        \item all $a_i \in \N^+ \cup \{ \infty \}$;
        \item Either all $a_i < \infty$, or there exists $a_{i+1} = \infty$ and $a_0, a_1, \cdots, a_i < \infty$. We do not define terms beyond $a_{i+1}$ when $a_{i+1} = \infty$.
    \end{itemize}
\end{definition}

There is a natural surjection from the set of all expansion sequences to the set of all numbers in $[0,1]$, by taking the formula of continued fraction expansion. However, this map is not bijective, like $(3,1,\infty)$ and $(4, \infty)$ both corresponds to the number $1/4$.
\begin{theorem}
    \label{surjcorre}
    There's a surjection from
    $$\left\{ a: \Z \to \N^+ \cup \{ \infty \} \middle| \text{both } (a_0, a_1, \cdots) \text{ and } (a_{-1}, a_{-2}, \cdots) \text{ are expansion sequences}\right\}$$
    to
    $$\left\{\Lambda \ \middle |  \  \Lambda \text{ is a unimodular lattice} \right\} \Big/ \sim_1$$
    where $\sim_1$ is the quotient under action $g_t$.
\end{theorem}
\begin{proof}
    Just set
    $$
    \begin{aligned}
        \alpha &= [0;\  a_0,\  a_1,\  \cdots] \\
        \beta &= [0;\  a_{-1},\  a_{-2},\  \cdots]
    \end{aligned}
    $$
    and apply the above lemma \ref{arbitraryreconstruction} to get the lattice.
\end{proof}
Although being a surjection, not a bijection, this allows us to lift the spectra on lattices to the discussions on sequences, which simplifies certain discussions like closedness and Hall's ray. See chapters below for applications.


\section{The Mordell-Gruber Spectrum}

\subsection{The Mordell Constant in Another View}
We know that
$$
\begin{aligned}
    \kappa(\Lambda) &:= \sup_{[-x, x] \times [-y, y] \text{ admissible in }\Lambda} xy \\
    &= \exp (2 \sup_{t \in \R} W(t; \Lambda))
\end{aligned}
$$
and is equivariant under the diagonal flow action.

Using Theorem \ref{fourinone}, we can use the index sequence to study this case:
\begin{equation}
    \label{eq:mg2}
    \kappa(\Lambda) = \sup_{j \in \Z} \left( \frac{1}{1 +\beta_j\alpha_{j-1}}\right)
\end{equation}
which allows us to compute the Mordell constant efficiently. 

By using the equivalence we have established, we translate our discussion into solely bi-infinite sequences throughout this chapter.

\subsection{The Lower Part}
Using the formula \ref{eq:mg2}, we can compute both the lower bound and the lowest few points of the Mordell-Gruber spectrum.
\begin{theorem}
    The minimum of $\kappa(\Lambda)$ is
    $$ \min _{\Lambda \text{ unimodular}} \kappa(\Lambda) = \frac{5+\sqrt{5}} {10}$$
    and is only achievable when the index sequence is all $1$ (i.e. $\forall i \in \Z \quad a_i = 1$).
\end{theorem}

\begin{proof}
    It's by direct computation that when all $a_i = 1$, then the Mordell constant is
    $$\frac{1}{1+\left( \frac{\sqrt{5}-1}{2} \right)^2} = \frac{5+\sqrt{5}} {10}$$
    Otherwise, consider the largest value of $\{a_j\}$. If $\{a_j\}$ is unbounded, then $\beta_j\alpha_{j-1}>0$ can be arbitrarily small, which means $\kappa(\Lambda) = 1$. Now WLOG suppose $a_0 \ge 2$ is the largest integer in $\{a_j\}$, we prove that $$\beta_{0} \alpha_{-1} < \left( \frac{\sqrt{5}-1}{2} \right)^2$$
    \begin{enumerate}
        \item If $a_0 \ge 3$: then $$\beta_{0} \alpha_{-1} < 1/3 \cdot 1 < 0.334 < \left( \frac{\sqrt{5}-1}{2} \right)^2$$
        \item If $a_0 = 2$: then $\alpha_{-1} < [0; \overline{1,2}]$ and $\beta_0 < [0; 2, \overline{2,1}]$. So
        $$\beta_{0} \alpha_{-1} < [0; \overline{1,2}]\cdot [0; 2, \overline{2,1}] <  \frac{4}{\sqrt{3}}-2 < 0.31 < \left( \frac{\sqrt{5}-1}{2} \right)^2$$
    \end{enumerate}
\end{proof}

Although this theorem is already proved in 1937 by Szekeres \cite{Szekeres_Lattice_Plane}, it serves as an example of showing the effectiveness of our above tools, which enables us to complete the proof in a few lines.

It's also feasible to study the spectrum in the region $[\frac{5+\sqrt{5}}{10}, \frac{1+\sqrt{5}}{4}]$. We have the following theorem:
\begin{theorem}
    When and only when $\kappa(\Lambda) \in [\frac{5+\sqrt{5}}{10}, \frac{1+\sqrt{5}}{4}]$, the index sequence is one of these, up to a shift:
    \begin{itemize}
        \item $(\overline{1})$;
        \item $(\overline{1}, \ 2, \  \overline{1})$;
        \item $(\overline{2,\ 1^{(k)}})$, where $k$ is an odd number.
    \end{itemize}
\end{theorem}

\begin{proof}
    Recall a lemma of comparing two continued fractions:

    Let $S = (s_0; s_1, s_2, \cdots)$ and $T = (t_0; t_1, \cdots)$ be two continued fractions, and let $i$ be the smallest number such that $s_i < t_i$, and $s_j = t_j$ for every $j < i$. If $i$ is even, then $S<T$, otherwise $S>T$.

    The next step is to find infimum of $\beta_j\alpha_{j-1}$, because
    $$\kappa(\Lambda) \le \frac{1+\sqrt{5}}{4} \iff \inf_{j \in \Z} \beta_j\alpha_{j-1} \ge \sqrt{5} -2$$
    
    WLOG let $a_0$ be the largest number in $\{a_n\}$, since $\kappa(\Lambda)$ is obviously $1$ if $\{a_n\}$ is unbounded.

    \begin{itemize}
        \item $a_0 \ge 5$: $\alpha_{-1} \le \frac{1}{5}$, $\alpha_{-1}\beta_0 \le \frac{1}{5} < \sqrt{5} -2$, which is impossible.
        \item $a_0 = 4$: $\alpha_{-1} \le [0; 4,5]$ and $\beta_{-1} \le [0; 1,5]$, $\alpha_{-1}\beta_0 \le \frac{25}{126} < \sqrt{5} -2$, which is also impossible.
        \item $a_0 = 3$: $\alpha_{-1} \le [0; 3,4]$ and $\beta_{-1} \le [0; 1,4]$, $\alpha_{-1}\beta_0 \le \frac{16}{65} < \sqrt{5} -2$, which is also impossible.
        \item $a_0 = 1$: As discussed above, all $a_n =1$, and $\kappa(\Lambda)  = \frac{5+\sqrt{5}}{10}$.
        \item $a_0 = 2$: Then all $a_n$ must be either $1$ or $2$. 
        \begin{itemize}
            \item If there is only one $2$ in the sequence: The sequence is therefore $(\overline{1}, 2, \overline{1})$ with $\inf_{j \in \Z} \beta_j\alpha_{j-1} = \beta_0 \alpha_{-1} = [0; \overline{1}]\cdot [0; 2, \overline{1}] = \sqrt{5} -2$ and $\kappa(\Lambda)  = \frac{1+\sqrt{5}}{4}$.
            \item If there are at least two terms of $2$ in the sequence:
            First we show that there can't exist two consecutive $2$'s: Assume $a_{-1} = a_0 = 2$, then $\alpha_{-1} < [0; 2,3] = \frac{3}{7}$, so is $\beta_0$, and $\alpha_{-1} \beta_0 < \frac{9}{49} < \sqrt{5} - 2$.
            \begin{itemize}
                \item If there exists $j$ such that $a_j=2$ and for all $i<j$, $a_i = 1$: WLOG let $j=0$. 
                $$ ( \overline {1}, 2, 1^{(k)}, 2, \cdots) (k \in \N^+)$$
                \begin{itemize}
                    \item If $k$ is odd: then $\alpha_{-1} = [0; 2, 1^{(k)}, 2, \cdots] < [0; 2, \overline{1}] = \frac{3-\sqrt{5}}{2}$, but $\beta_0 = \frac{\sqrt{5}-1}{2}$, $\alpha_{-1}\beta_0 < \sqrt{5} -2$, impossible.
                    \item If $k$ is even: then $\alpha_{0} = [0; 1^{(k)}, 2, \cdots] < [0; \overline{1}] = \frac{\sqrt{5}-1}{2}$ and $\beta_1 = [0; 2, \overline{1}] = \frac{3-\sqrt{5}}{2}$. So $\alpha_0\beta_1 < \sqrt{5} -2$, impossible.
                \end{itemize}
                \item If there exists $j$ such that $a_j=2$ and for all $i>j$, $a_i = 1$:
                $$ (\cdots, 2, 1^{(k)}, 2, \overline {1}) (k \in \N)$$
                The discussion is same as above by swapping $\alpha$ and $\beta$.
                \item The remaining sequences look like:
                $$ (\cdots, 2, 1^{(k_1)}, 2(*), 1^{(k_2)}, 2, \cdots) (k_1, k_2 \in \N)$$
                Where WLOG let $a_0$ be the position of $*$ in the sequence.
                
                We prove that $k_1 = k_2$ and is an odd number.

                \begin{itemize}
                    \item If $k_1$ is odd, $k_2$ is even: 
                    $$\beta_1 = [0; 2, 1^{(k_1)}, 2, \cdots] < \frac{3-\sqrt{5}}{2}$$
                    $$\alpha_0 = [0; 1^{(k_2)}, 2, \cdots] < \frac{\sqrt{5}-1}{2}$$
                    $$\alpha_0 \beta_1 < \sqrt{5}-2 $$
                    \item If $k_1$ is odd, $k_2$ is even: similar as above.
                    \item If $k_1$ and $k_2$ are odd but not equal: WLOG $k_1 \ge k_2 + 2$, then
                    $$\beta_0 = [0; 1^{(k_1)}, 2, \cdots] < [0; 1^{(k_2+2)}, 3]$$
                    $$\alpha_{-1} = [0; 2, 1^{(k_2)}, 2, 1, \cdots] < [0; 2, 1^{(k_2)}, 2, 2]$$
                    $$\alpha_{-1} \beta_0 < [0; 1^{(k_2+2)}, 3] \cdot [0; 2, 1^{(k_2)}, 2, 2] < \sqrt{5}-2 $$
                    \item If $k_1$ and $k_2$ are even but not equal: WLOG $k_1 \ge k_2+2$, then
                    $$\alpha_{0} = [0; 1^{(k_2)}, 2, 1, \cdots] < [0; 1^{(k_2)}, 2, 2]$$
                    $$\beta_1 = [0;2, 1^{(k_1)}, 2, \cdots] < [0; 2, 1^{(k_2+2)}, 3]$$
                    $$\alpha_{0} \beta_1 < [0; 2, 1^{(k_2+2)}, 3] \cdot [0; 1^{(k_2)}, 2, 2] < \sqrt{5}-2 $$
                    (For a detailed proof of the above two inequalities, see \cite{Lesca_Diophantine_Approximations}.)
                    \item If all $k_1 = k_2$, but are even: 
                    $$\beta_0 = [0; \overline {1^{(k)}, 2}], \quad \alpha_{-1} = [0; \overline {2, 1^{(k)}}]$$
                    $$\alpha_{-1} \beta_0 = [0; \overline {1^{(k)}, 2}] [0; \overline {2, 1^{(k)}}]$$
                    In this case, we rewrite $\alpha_{-1} \beta_0$ as
                    $$\alpha_{-1} \beta_0 = \frac{h_k}{2+h_k}, \quad \text{where } h_k = [0; \overline {1^{(k)}, 2}] \in (0,1)$$
                    as a monotonic function about $h_k$.
                    Since $\{h_{2j}\}_j$ is a monotonic increasing sequence with limit $\frac{\sqrt{5}-1}{2}$, we obtain that $\alpha_{-1} \beta_0 < \sqrt{5} -2$.
                    \item If all $k_1 = k_2$, but are odd: This is possible, see below for a verification.
                \end{itemize}
            \end{itemize}
        \end{itemize}
    \end{itemize}
    Now, it's also necessary to verify that these 3 kinds of sequences do satisfy the condition above. 
    \begin{itemize}
        \item If the sequence is $(\overline{1})$: Obviously.
        \item If the sequence is $(\overline{1}, \ 2, \  \overline{1})$, where WLOG $a_0 = 2$. We have verified that $\beta_0 \alpha_{-1} = [0; \overline{1}]\cdot [0; 2, \overline{1}] = \sqrt{5} -2$, while for a different offset, $[0; \overline{1}] \cdot [0; 1^{(l)}, 2, \overline{1}] > [0; \overline{1}]\cdot [0; 2, \overline{1}] = \sqrt{5} -2$.
        \item $(\overline{2,\ 1^{(k)}})$, where $k$ is an odd number and WLOG $a_0 = 2$.
        Similarly, 
        $$\alpha_{-1} \beta_0 = \frac{h_k}{2+h_k}, \quad \text{where } h_k = [0; \overline {1^{(k)}, 2}] \in (0,1)$$
        Since $\{h_{2j-1}\}_j$ is a monotonic decreasing sequence with limit $\frac{\sqrt{5}-1}{2}$, we obtain that $\alpha_{-1} \beta_0 > \sqrt{5} -2$.

        To verify that the infimum $\inf_{j \in \Z} \beta_j\alpha_{j-1} = \alpha_{-1} \beta_0$, consider a different offset $\alpha_j \beta_{j+1}$, where WLOG $j = 0,1,2, \cdots, k-1$. 
        
        For $j=0$, we have $\alpha_0 = \beta_0$ and $\beta_1 = \alpha_{-1}$. 
        
        For $j = 1, 2, \cdots, k-1$:
        $$\alpha_j = [0;\  1^{(k-j)}, \ \overline{2, \ 1^{(k)}}]$$
        $$\beta_{j+1} = [0;\  1^{(j)},\  \overline{2, \ 1^{(k)}}]$$
        \begin{itemize}
            \item If $j$ is odd, then $k-j$ is even. 
            $$\alpha_j = [0;\  1^{(k-j)}, \ \overline{2, \ 1^{(k)}}] > [0;\ \overline{2, \ 1^{(k)}}] = \alpha_{-1}$$
            $$\beta_{j+1} = [0;\  1^{(j)},\  \overline{2, \ 1^{(k)}}] > [0;\  1^{(k)},\  \overline{2, \ 1^{(k)}}] = \beta_{0}$$
            $$\alpha_{j}\beta_{j+1} > \alpha_{-1}\beta_{0}$$
            \item If $j$ is even, then $k-j$ is odd. Similarly, $\alpha_j > \beta_0$ and $\beta_{j+1} > \alpha_{-1}$.
        \end{itemize}
    \end{itemize}
\end{proof}

Based on the above theorem, the first few terms in the Mordell-Gruber spectrum are:
\begin{itemize}
    \item $(\overline{1})$: $\kappa(\Lambda) = \frac{5+\sqrt{5}}{10} \approx 0.7236$
    \item $(\overline{2,\ 1})$: $\kappa(\Lambda) = \frac{3+\sqrt{3}}{6} \approx 0.7887$
    \item $(\overline{2,\ 1,\  1,\  1})$: $\kappa(\Lambda) = \frac{4+\sqrt{6}}{8} \approx 0.8062$
    \item $\cdots$
    \item $(\overline{1}, \ 2, \  \overline{1})$: $\kappa(\Lambda) = \frac{1+\sqrt{5}}{4} \approx 0.8090$
\end{itemize}

The closed form of the Mordell constant of $(\overline{2,\ 1^{(2t-1)}})$ is computed in the following procedure.

Let $\{F_n\}$ denote the Fibonacci numbers, i.e., $F_0 = 0, F_1 = 1, F_{n+1} = F_n + F_{n-1} (n \in \N^+)$.

Set $x_t = [1; \ \overline{1^{(t-1)},\  2, \ 1}]$, then $x_t$ is a quadratic rational, and is the unique positive root of the following quadratic equation (by simple induction):
$$x_t = \frac{F_{2t+2} x_t + F_{2t}}{F_{2t+1}x_t + F_{2t-1}}, \quad \text{or } F_{2t+1} x_t^2 - 2F_{2t}x_t + F_{2t} = 0$$
and
$$\kappa(\Lambda) = \frac{1}{1+\frac{1}{x_t}\cdot \frac{1}{2+\frac{1}{x_t}}} = \frac{2x_t+1}{2x_t+2} = \frac{1}{2} + \frac{1}{2} \cdot \frac{x_t}{x_t+1}$$
where
$$\frac{x_t}{x_t+1} = \frac{F_{2t} + \sqrt{F_{2t}^2 + F_{2t}F_{2t+1}}}{F_{2t} + F_{2t+1} + \sqrt{F_{2t}^2 + F_{2t}F_{2t+1}}} = \sqrt{\frac{F_{2t}}{F_{2t}+F_{2t+1}}} = \sqrt{\frac{F_{2t}}{F_{2t+2}}}$$
So
$$\kappa(\Lambda) = \frac{1+\sqrt{\frac{F_{2t}}{F_{2t+2}}}}{2} \quad (t = 1, 2, \cdots)$$

In summary, we have the following theorem:
\begin{theorem}
    The spectrum $\mathrm{MG}_2$ satisfies the following:
    \begin{itemize}
        \item $\min (\mathrm{MG}_2) = \frac{5+\sqrt{5}}{10}$;
        \item The lower part:
        $$\mathrm{MG}_2 \cap \left(\frac{5+\sqrt{5}}{10} , \frac{1+\sqrt{5}}{4}\right) = \left\{ \frac{1+\sqrt{\frac{F_{2t}}{F_{2t+2}}}}{2}\ \middle| \ t = 1, 2, \cdots \right\}$$
        \item The smallest accumulation point of $\mathrm{MG}_2$ is $\frac{1+\sqrt{5}}{4} \in \mathrm{MG}_2$.
    \end{itemize}
\end{theorem}

Note: In 1972, Divis \cite{Divis_Lagrange_Numbers} studied a very similar spectrum, where $2+\sqrt{5}$ (corresponding to $\frac{1+\sqrt{5}}{4}$ in our case) doesn't belong to the spectrum. It's worth noting that only one-sided infinite sequences are taken into account in the spectrum of Divis. Since we discuss bi-infinite sequences, it's worth reconsidering if the Mordell-Gruber spectrum (in the bi-infinite sense) is closed, as pointed out in a revision of \cite{Shapira_Weiss_Mordell_Gruber}. We will prove later that the Mordell-Gruber spectrum is actually closed.

\subsection{Existence of Hall's Ray or Segment}
In the Lagrange Spectrum $\mathrm{L}$ \cite{Cusick_Flahive_Markoff_Lagrange}, it's proved that there exists a constant $C$, such that $[ C, \infty] \subset \mathrm{L}$. The part is called \emph{Hall's ray}, which was proved in Hall's famous paper \cite{Hall_Sum_Product_CF}. The minimum of $C$ such that $[ C, \infty] \subset \mathrm{L}$ is called \emph{Freiman's Constant}.

We are interested in whether the Mordell-Gruber spectrum admits a Hall's ray. Since we already know that $\mathrm{MG}_2 \subset \left[\frac{5+\sqrt{5}}{10}, 1\right]$, we prove that there exists a segment $[C, 1] \subset \mathrm{MG}_2$.

Consider a special case of the sequence $\{a_n\}$, such that $a_{-1} \ge 5,\  a_0 \ge 5$, and for all $t \notin {0,1}$, $a_t \le 4$.

\begin{lemma}
    In this case, $\inf_{j \in \Z} \beta_j\alpha_{j-1} = \beta_{0} \alpha_{-1}$.
\end{lemma}

\begin{proof}
    First, we check that
    $$\beta_0 \alpha_{-1} = [0;\  a_0,\  a_1,\  \cdots][0;\  a_{-1},\  a_{-2},\  \cdots] < [0;\ a_0,\  5][0;\  a_{-1},\  5]$$
    $$\beta_{1} \alpha_{0} = [0;\  a_1,\ a_2, \  \cdots][0;\ a_0, \  a_{-1},\  \cdots] > [0; 5][0;\  a_{0},\  a_{-1}]$$
    To check that $[0; 5] [0;\  a_{0},\  a_{-1}] > [0;\ a_0,\  5] [0;\  a_{-1},\  5]$, it suffices to check 
    $$5\left(a_0 + \frac{1}{a_{-1}}\right) < \left(a_0 + \frac15 \right)\left(a_{-1} + \frac15 \right)$$
    which is obvious when $a_0, a_{-1} \ge 5$. 

    Similarly, $[0; 5] [0;\  a_{-1},\  a_{0}] > [0;\ a_0,\  5] [0;\  a_{-1},\  5]$, which implies $\beta_{-1} \alpha_{-2} > \beta_0 \alpha_{-1}$.

    For other offsets, 
    $$\beta_j \alpha_{j-1} > [0;\ 5] [0;\ 5] = \frac{1}{25} > [0;\ a_0,\  5] [0;\  a_{-1},\  5]$$
    Which finishes the proof.
\end{proof}

\begin{lemma}
    In the above case, $\beta_0 \alpha_{-1}$ can be any value in $\left(0, \frac{4}{83+18\sqrt{2}}\right]$.
\end{lemma}

\begin{proof}
    This is exactly shown by Hall \cite{Hall_Sum_Product_CF} at Theorem 3.2, where we take minimum of $F(5,4)$.
\end{proof}

Combining the above lemmas, we prove the Hall's segment:
\begin{theorem}
    Mordell-Gruber Spectrum contains a Hall's segment
    $$\left[ \frac{83+18\sqrt{2}}{87+18\sqrt{2}} \approx 0.9645,\  1\right] \subset \mathrm{MG}_2$$
\end{theorem}

\begin{corollary}
    Mordell-Gruber Spectrum is an uncountable set.
\end{corollary}

Of course, the bound can be substantially improved, and even analog of Freiman's constant is computable. But we will not focus on details of computing the constant in this paper.

\begin{problem}
    Compute the exact infimum value of $C$ such that $[C,1] \subset \mathrm{MG}_2$.
\end{problem}

\subsection{Closedness of the Spectrum}
In this section, we prove that $\mathrm{MG}_2$ is a closed set.

Before proving that, we need a lemma about the convergence rate of Diophantine approximants.

\begin{lemma}
    \label{localcontrol}
    Suppose $x = [0;\  a_1,\  a_2, \ \cdots]$, where $0 \le x \le 1$, and $\eta_n$ be an arbitrary number in $[[0;\ a_1, \ a_2, \ \cdots, \ a_n], [0;\ a_1, \ a_2, \ \cdots,\  a_n +1] ]$ or $[[0;\ a_1, \ a_2, \ \cdots, \ a_n +1], [0;\ a_1, \ a_2, \ \cdots,\  a_n] ]$. Then $|\eta_n - x| \le 2^{-n}$.
\end{lemma}
\begin{proof}
    We prove by induction of $n$.
    \begin{itemize}
        \item $n=0$: Obviously $ \eta \in [0,1], \ x \in [0,1],\ |\eta - x| \le 1 $.
        \item $n=1$: Obviously the length of the interval $[\frac{1}{a_1+1}, \frac{1}{a_1}]$ is no more than $1/2$.
        \item $n \to n+2$: Consider function
        $$f_{c,d} (x) = \frac{1}{c+\frac{1}{d+x}} \quad(c,d \in \N^+, x \in [0,1])$$
        then
        $$\frac{\mathrm{d}}{\mathrm{d} x} f_{c,d} (x) = \frac{1}{(cx+cd+1)^2} \in \left[ 0, \frac{1}{4}\right]$$
        so for any $c, d \in \N^+$, $f_{c,d}$ is a contraction map with contraction constant $\frac{1}{4}$.
        Now rewrite
        $$x = \frac{1}{a_1+\frac{1}{a_2+\frac{1}{x'}}} = f_{a_1,a_2} (x')$$
        $$\eta = \frac{1}{a_1+\frac{1}{a_2+\frac{1}{\eta'}}} = f_{a_1,a_2} (\eta')$$
        By induction hypothesis, we have $|\eta' - x'| \le 2^{-n}$.
        So 
        $$|\eta - x| = |f_{a_1,a_2} (\eta') - f_{a_1,a_2} (x')| \le \frac{1}{4} |\eta' - x'| \le 2^{-n-2}$$
    \end{itemize}
    which finishes the induction.
\end{proof}

In the following few lemmas, we need to discuss more than one index sequence. We need a notation template for convenience, where $\square$ is a template:

$$
\begin{aligned}
    \alpha_n^{[\square]} &:= [0;\  \square_{n+1},\  \square_{n+2},\  \cdots] \\
    \beta_n^{[\square]} &:= [0;\  \square_{n-1},\  \square_{n-2},\  \cdots]
\end{aligned}
$$
And the Perron's formula as follows:
$$
\begin{aligned}
    P_n^{[\square]} &:= \frac{1}{1+\alpha_{n-1}^{[\square]}\beta_{n}^{[\square]}}
\end{aligned}
$$

\begin{lemma}
    \label{lem6bookanalog}
    For a bi-infinite sequence $\{a_n\}$, if it is uniformly bounded, i.e., $ \forall n \quad a_n \le C$, then there exists a sequence $\{b_n\}$ such that
    $$ \sup_{n \in \Z} P_n^{[a]} = \sup_{n \in \Z} P_n^{[b]} = P_0^{[b]}$$
\end{lemma}

\begin{proof}
    This proof mainly follows from Lemma 6 in the book \cite{Cusick_Flahive_Markoff_Lagrange}. 
    We discuss two cases:
    \begin{itemize}
        \item There exists $i$ such that
        $$\sup_{n \in \Z} P_n^{[a]} = P_{i}^{[a]}$$
        then simply let $b_n = a_{n+i}$.
        \item Such $i$ does not exist.
        By the definition of the supremum, we can find a sequence $\{m_j\}_j$ such that
        $$\sup_{n \in \Z} P_n^{[a]} = \lim _{j \to \infty} P_{m_j}^{[a]}$$

        WLOG, we can also assume that $P_{m_j}^{[a]}$ converges monotonically to $ \sup_{n \in \Z} P_{n}^{[a]}$. 

        We define a new set of sequences as $a(j)_k := a_{m_j + k}$ for $j \in \N^+$ and $k \in \Z$. We construct $b_k$ by induction.

        \begin{itemize}
            \item $k=0$: Since $\{a_n\}$ is bounded, there exists a number that appears infinitely times in $a(j)_0$. Let $b_0$ be the number.
            \item $k-1 \to k$: Assume we have fixed $(b_{-k+1}, \cdots, b_0, \cdots, b_{k-1})$, There exists $b_{-k}$ and $b_{k}$ such that 
            $$ (b_{-k}, \cdots, b_0, \cdots, b_{k}) = (a(j)_{-k}, \cdots,  a(j)_0, \cdots, a(j)_{k})$$ 
            for infinitely many $j$.
        \end{itemize}
        
        We show that the sequence of $\{b_k\}$ is just as we desired.

        First, we prove that 
        $$ \sup_{n \in \Z} P_n^{[a]} = P_0^{[b]}$$
        Which is equivalent to 
        $$ \lim _{j \to \infty} P_{m_j}^{[a]} = \lim _{j \to \infty} P_{0}^{[a(m_j)]} = P_0^{[b]}$$
        For any $\epsilon > 0$, let $k = \lfloor - \ln \epsilon \rfloor + 2$, there exists $j$ such that
        $$(b_{-k}, \cdots, b_0, \cdots, b_{k}) = (a(m_j)_{-k}, \cdots,  a(m_j)_0, \cdots, a(m_j)_{k})$$
        Consider this formula
        $$
        \begin{aligned}
            \left\lvert P_0^{[a(m_j)]} - P_0^{[b]} \right\rvert &= \left\lvert \frac{1}{1 + \alpha_{-1}^{[a(m_j)]} \beta_{0}^{[a(m_j)]}} - \frac{1}{1 + \alpha_{-1}^{[b]} \beta_{0}^{[b]}} \right\rvert \\
            & \le \left\lvert - \alpha_{-1}^{[a(m_j)]} \beta_{0}^{[a(m_j)]} +  \alpha_{-1}^{[b]} \beta_{0}^{[b]} \right\rvert \\
            & = \left\lvert \left(\alpha_{-1}^{[b]} - \alpha_{-1}^{[a(m_j)]} \right) \beta_{0}^{[a(m_j)]} +  \alpha_{-1}^{[b]} \left(\beta_{0}^{[b]} - \beta_{0}^{[a(m_j)]}\right) \right\rvert \\
            & \le \left\lvert \alpha_{-1}^{[b]} - \alpha_{-1}^{[a(m_j)]} \right\rvert + \left\lvert \beta_{0}^{[b]} - \beta_{0}^{[a(m_j)]} \right\rvert
        \end{aligned}
        $$
        Now, since the first $k$ terms of the continued fraction expansions of both $\beta_{0}^{[b]}$ and $\beta_{0}^{[a(m_j)]}$ are $[0; b_{-1}, \cdots, b_{-k}]$, we obtain by Lemma \ref{localcontrol} that 
        $$\left\lvert \beta_{0}^{[b]} - \beta_{0}^{[a(m_j)]} \right\rvert \le 2^{-k} \le \frac {\epsilon}{2}$$
        Similarly
        $$\left\lvert \alpha_{-1}^{[b]} - \alpha_{-1}^{[a(m_j)]} \right\rvert \le 2^{-k-1} \le \frac {\epsilon}{2}$$
        So
        $$\left\lvert P_0^{[a(m_j)]} - P_0^{[b]} \right\rvert \le \frac {\epsilon}{2} + \frac {\epsilon}{2} = \epsilon $$

        By monotonicity (as assumed before), any $N \ge j$ satisfies
        $$\left\lvert P_0^{[a(m_N)]} - P_0^{[b]} \right\rvert \le \epsilon $$
        which finishes the proof for this part.
        
        The next part is to prove that
        $$\sup_{n \in \Z} P_n^{[b]} = P_0^{[b]}$$
        equivalently,
        $$\forall n \in \Z \quad P_n^{[b]} \le  P_0^{[b]}$$
        Suppose the other way, $P_n^{[b]} >  P_0^{[b]} + \epsilon$.
        The proof is similar to above, but we choose $k = |n| + \lfloor - \ln \epsilon \rfloor + 4$. Then it can be proved that
        $$\left\lvert P_n^{[a(m_j)]} - P_n^{[b]} \right\rvert \le \frac {\epsilon}{4}$$
        However, 
        $$P_n^{[a(m_j)]} \le  \sup_{n \in \Z} P_n^{[a]} = P_0^{[b]}$$
        This contradicts with $P_n^{[b]} >  P_0^{[b]} + \epsilon$.
    \end{itemize}
\end{proof}

Now it's ready to prove the closedness.

\begin{theorem}
    \label{mg2closed}
    The Mordell-Gruber spectrum $\mathrm{MG}_2$ is a closed set.
\end{theorem}

\begin{proof}
    We prove that $\mathrm{MG}_2$ is closed in $[0, 1)$, since we have already proved that $1 \in \mathrm{MG}_2$. Specifically, we prove that, for any set of index sequences $\{a(j)_k\}$, where $j = 1,2, \cdots$, and $k \in \Z$, if
    $$ \lim _{j \to \infty} \sup_{k \in \Z} P_{k}^{[a(j)]} = y <1$$
    then there exists a sequence $\{c_l\}_{l \in \Z}$ such that
    $\sup_{k \in \Z} P_{k}^{[c]} = y$.

    WLOG, we also suppose that $\sup_{k \in \Z} P_{k}^{[a(j)]} < \frac{1+y}{2}$.
    Then $a(j)_k$ are uniformly bounded: 
    $$\forall j \in \N^+ \quad \forall k \in \Z \quad a(j)_k \le C $$
    (where $C = \frac{1+y}{1-y}$ is enough in this case). 

    From the last lemma, we can obtain another set of sequences $\{b(j)_k\}$ such that $\forall j \in \N^+$, 
    $$\sup_{k \in \Z} P_{k}^{[a(j)]} = \sup_{k \in \Z} P_{k}^{[b(j)]} = P_{0}^{[b(j)]}$$
    Then obviously, all terms of $\{b(j)_k\}$ are still bounded by $C$.
    
    Now construct $c_l$ by induction:
    \begin{itemize}
        \item $l=0$: Since $\{b(j)_k\}$ is bounded, there exists a number that appears infinitely times in $b(j)_0$. Let $c_0$ be the number.
        \item $l-1 \to l$: Assume we have fixed $(c_{-l+1}, \cdots, c_0, \cdots, c_{l-1})$, There exists $c_{-l}$ and $c_{l}$ such that 
        $$ (c_{-l}, \cdots, c_0, \cdots, c_{l}) = (b(j)_{-l}, \cdots,  b(j)_0, \cdots, b(j)_{l})$$ 
        for infinitely many $j$.
    \end{itemize}
    The rest steps are very similar to the proof of the above lemma. In fact, one can show that
    $$P_0^{[c]} = \lim _{j \to \infty} P_{0}^{[b(j)]} = y$$
    and also
    $$P_0^{[c]} = \sup_{l \in \Z} P_l^{[c]}$$
    Therefore $y \in \mathrm{MG}_2$, so $\mathrm{MG}_2$ is closed in the region $[0,1)$. So $\mathrm{MG}_2$ is a closed set.
\end{proof}

\begin{corollary}
     $\mathrm{MG}_2$ is a compact set.
\end{corollary}

\begin{corollary}
     $\mathrm{MG}_2$ is a Borel measurable set.
\end{corollary}

Our results of Mordell-Gruber spectrum can be concluded in a figure.

\begin{figure}
    \centering
    \includegraphics[width=\textwidth]{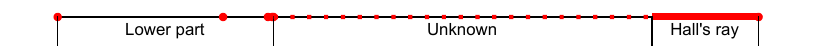}
    \caption{The structure of Mordell-Gruber spectrum, concluded in this chapter, drawn proportional to numerical values.}
    \label{fig:MG2}
\end{figure}

It would be interesting to study the total measure or the Hausdorff dimension of the spectrum.

\begin{problem}
    Compute the total Lebesgue measure of $\mathrm{MG}_2$.
\end{problem}

\begin{corollary}
    The lower and upper bound of the Lebesgue measure is given by
    $$ \int_{\R} \mathbf{1}_{\mathrm{MG}_2} \ \mathrm{d}x \in \left[  \frac{4}{87+18\sqrt{2}}, \ \frac{3-\sqrt{5}}{4} \right]$$
\end{corollary}

\begin{problem}
    Study the Hausdorff dimension of $\mathrm{MG}_2$ at a certain point or region.
\end{problem}

\section{Generalized View of Spectra: Inclusion and Closedness}

\subsection{Generalized Perron's Formula and Spectrum}

We generalize Perron's formula in the following way. Let $\{a_n\}_{n \in \Z}$ denote the index sequence, and $\alpha_k^{[a]}$ and $\beta_k^{[a]}$ inherited from the previous definition. The generalized Perron's formula has the following form:
$$P_k^{[a]} = f(\alpha_{k+l-1}^{[a]}, \ \beta_{k}^{[a]},\  \  a_{k}, a_{k+1}, \cdots, a_{k+l-1})$$
Where $f: [0,1]^2 \times \left(\N^+ \cup \{\infty\} \right)^{l} \to (\R \cup \pm \infty )$ or $[0,1]^2 \times \left(\N^+\right)^{l} \to (\R \cup \pm \infty )$ is a general function. The part of $a_{k}, a_{k+1}, \cdots, a_{k+l-1}$ is called integral parameters, and $\alpha_{k+l-1}^{[a]}, \ \beta_{k}^{[a]}$ is called fractional parameters. The count of integral parameters $l \in \N$ is a non-negative integer, and it is possible that $l=0$.

We define "one side" of the sequence $\{a_n\}$ to be $(a_0, a_1, \cdots)$, and "the other side" is $(a_{-1}, a_{-2}, \cdots)$.
We specially allow $\infty$ to be a possible value in the sequence for discussing special lattices. When there is only one $\infty$ at one side in the sequence, the corresponding lattice is mono-infinite; when there are two terms equaling to $\infty$ from both sides in the sequence, the lattice is therefore finite. Also, we do note define terms beyond $\infty$ when $\infty$ appears in a side of the sequence. As an example, If $a_0, a_1, \cdots, a_k$ are defined such that $a_k$ is the first occurrence of $\infty$, we do not define $a_{k+1}$ or terms beyond that. This is the same of "expansion sequence" appeared in Definition \ref{expansionseq}.

Note that the limit of $\inf$ and $\sup$ can be defined on both finite and infinite sequences (i.e., sequence with and without occurrence of $\infty$), but $\liminf$ and $\limsup$ should only be defined on infinite sequences.

The induced spectrum takes the following form:
$$\sigma_{f, \mathrm{limit}} := \left\{ \mathrm{limit}_k  \left( P_k^{[a]} \right) \ \middle| \ a:\Z \to  \left(\N^+ \cup \{\infty\} \right) \ \right\}$$
Where the limit takes one of the following types:
$$ \inf, \quad \sup$$
and
$$\sigma'_{f, \mathrm{limit}} := \left\{ \mathrm{limit}_k  \left( P_k^{[a]} \right) \ \middle| \ a:\Z \to   \N^+ \ \right\}$$
Where the limit takes one of the following types:
$$ \liminf, \quad \limsup$$

\subsection{Good Continuity}
Since such a generalized Perron's function $f$ is too general to deduce anything meaningful, we add some continuity conditions on $f$:

\begin{definition}
    \label{goodcontinuity}
    We consider two kinds of continuity of Perron function $f$:
    \begin{itemize}
        \item Uniform limit at infinity for integral parameters: There exists a limit $I \in [-\infty, \infty]$, for any $\epsilon > 0$, if $\max(n_0, n_1, \cdots, n_{l-1}) > C $, then 
        $$f(\alpha, \beta, n_0, n_1, \cdots, n_{l-1}) \in B(I, \epsilon)$$
        where $B(I, \epsilon)$ is the neighborhood of $I$ and convergence rate is only dependent of $C$. Moreover, $f(\alpha, \beta, n_0, n_1, \cdots, n_{l-1}) = I$ if one of $n_0, n_1, \cdots, n_{k-1}$ takes $\infty$.
        \item Continuity for fractional parameters: For each set of $n_0, \cdots, n_{l-1} < \infty $, $f(\alpha, \beta, n_0, n_1, \cdots, n_{l-1})$ is continuous with respect to $(\alpha, \beta) \in [0,1] \times [0,1]$. This also implies uniform continuity since $[0,1] \times [0,1]$ is a compact set.
    \end{itemize}
    If a Perron function $f$ satisfies two conditions above, we say that $f$ has a good continuity.
\end{definition}
We will later show that good continuity induces many interesting properties, such as closedness.

\subsection {Bidirectional Accumulation Sequence}
Recall the topology of $\N^+ \cup \{\infty\}$:
\begin{lemma}
    The set $\N^+ \cup \{\infty\} \subset [-\infty, \infty]$ is closed and compact. Moreover, for any $n \in \N^+$, $\left(\N^+ \cup \{\infty\} \right)^n$ is a closed and compact set. So any sequence on it admits a converging subsequence.
\end{lemma}
This compact property is used later in our construction.

Assume that $\{a(j)_k\} : (j:\N^+) \to (k: \Z) \to \left(a(j)_k: \N^+ \cup \{ \infty \} \right) $ is a family of sequences, where each $a(j)$ is a bidirectional "expansion sequence" which takes the form appeared in Theorem \ref{surjcorre}. We construct its bidirectional accumulation sequence $\{c_r\}_{r \in \Z}$ by recursive definition, where $c_r \in \N^+ \cup \{\infty\}$, in the following steps:
\begin{enumerate}
    \item $r=0$: 
    We construct $c_{0}$ and $c_{-1}$ simultaneously. Note that $(a(j)_{0}, a(j)_{-1}) \in \left(\N^+ \cup \{ \infty \} \right)^2 $ is a point sequence in a compact set. We set $(c_0, c_-1)$ to an accumulation point of $\{(a(j)_{0}, a(j)_{-1})\}_j$. Since $\left(\N^+ \cup \{ \infty \} \right)^2$ is a closed set, $(c_{0}, c_{-1}) \in \left(\N^+ \cup \{ \infty \} \right)^2 $.
    \begin{itemize}
        \item If one of $c_0, c_{-1}$ attains $\infty$, directly jump to step 3.
        \item If not, then $(c_{-1}, c_0) = (a(j)_{-1}, a(j)_{0})$ occurs infinitely many times for $j$.
    \end{itemize}
    \item $r-1 \to r$:
    The induction assumption is that the set
    $$\left\{ j \ \middle| \  c_{-r} = a(j)_{-r} < \infty,\ \ \cdots, \ \ c_{r-1} = a(j)_{r-1} < \infty\right\}$$
    is infinite.
    \begin{itemize}
        \item We take $c_{r}$ such that $$(c_{-r}, \cdots, c_0, \cdots, c_{r}) = (a(j)_{-r}, \cdots, a(j)_{0}, \cdots, a(j)_{r})$$ occurs infinitely many times for $j$, and $\infty$ if otherwise. Formally, we can set
        $$c_r = \liminf_{\left\{ j \ \middle| \  c_{-r} = a(j)_{-r} < \infty,\ \ \cdots, \ \ c_{r-1} = a(j)_{r-1} < \infty\right\}} a(j)_r$$
        When $c_r = \infty$, jump out of the loop towards step 3.
        \item If $c_r \ne \infty$, we take $c_{-r-1}$ such that $$(c_{-r-1}, \cdots, c_0, \cdots, c_{r}) = (a(j)_{-r-1}, \cdots, a(j)_{0}, \cdots, a(j)_{r})$$ occurs infinitely many times for $j$, and $\infty$ if otherwise. When $c_{-r} = \infty$, jump out of the loop towards step 3.
    \end{itemize}
    \item When one direction reaches $\infty$, WLOG $c_{-r-1} = \infty$, 
    we take $s$ from $r+1$ to $\infty$ to construct $c_s$. The induction assumption is that the set
    $$\left\{ j \ \middle| \  c_{-r} = a(j)_{-r} < \infty,\ \ \cdots, \ \ c_{s-1} = a(j)_{s-1} < \infty\right\}$$
    is infinite.
    We take $c_s$ such that 
    $$(c_{-r}, \cdots, c_0, \cdots, c_{s}) = (a(j)_{-r}, \cdots, a(j)_{0}, \cdots, a(j)_{s})$$ occurs infinitely many times for $j$, or $\infty$ if otherwise. Formally,
    $$c_s = \liminf_{\left\{ j \ \middle| \  c_{-r} = a(j)_{-r} < \infty,\ \ \cdots, \ \ c_{s-1} = a(j)_{s-1} < \infty\right\}} a(j)_s$$
    When $c_{s} = \infty$, jump out of the loop, and finish the construction.
\end{enumerate}
The bidirectional accumulation sequence $\{c_r\}_{r \in \Z}$ has one of the following forms:
\begin{enumerate}
    \item Finite: 
    $$c_{-s-1} = \infty,\ \   c_{-s}, \cdots, c_0, \cdots, c_{t-1}, \ \  c_{t} = \infty \quad(s, t \in \N)$$
    \item Mono-infinite (extends in one direction): 
    $$c_{-s-1} = \infty, c_{-s}, \cdots \quad(s \in \N)$$ or $$\cdots, c_{t-1}, c_{t} = \infty \quad(t \in \N)$$
    \item Bi-infinite (extends in both directions).
\end{enumerate}
In any of these cases, $P_0^{[c]}$ is well-defined. Note that $c$ is also an "expansion sequence" like the ones in Theorem \ref{arbitraryreconstruction}.

We prove the following lemma, which is useful in proving the closedness of certain spectra:
\begin{lemma}
    \label{accumulation}
    Assume that $\{a(j)_k\} : (j:\N^+) \to (k: \Z) \to \left(a(j)_k: \N^+ \cup \{ \infty \} \right) $ is a series of sequences, with bidirectional accumulation sequence $\{c_r\}_{r \in \Z}$. If $c_k < \infty$ for $k \in [-s, t]$, then for any Perron function with good continuity and any $k \in [-s-l, t+1]$, and if the limit $\lim_{j \to \infty} P_k^{[a(j)]}$ exists, we have
    $$ \lim_{j \to \infty} P_k^{[a(j)]} = P_k^{[c]}$$
\end{lemma}

\begin{proof}
    We consider two different cases:
    \begin{itemize}
    \item If one of $c_k, \cdots, c_{k+l-1}$ attains $\infty$ (WLOG $c_v = \infty$, this happens when $k \in [s-l, s-1] \cup [t-l+2, t+1]$): then $a(j)_{v} \to \infty$ and by uniform limit we know that $$\lim _{j \to +\infty} P_{k}^{[a(j)]} = P_k^{[c]} = I$$
    \item If none of $c_k, \cdots, c_{k+l-1}$ attains $\infty$ (this happens when $k \in [s, t-l+1]$), then 
    $$
    \begin{aligned}
        & \lim _{j \to +\infty} P_{k}^{[a(j)]} \\
        &= \lim _{j \to +\infty} f(\alpha_{k+l-1}^{[a(j)]}, \ \beta_{k}^{[a(j)]},\  \  a(j)_k, a(j)_{k+1}, \cdots, a(j)_{k+l-1})\\
        &= \lim _{j \to +\infty} f(\alpha_{k+l-1}^{[a(j)]}, \ \beta_{k}^{[a(j)]},\  \  c_k, c_{k+1}, \cdots, c_{k+l-1})\\
        &=  f(\lim _{j \to +\infty} \alpha_{k+l-1}^{[a(j)]}, \ \lim _{j \to +\infty} \beta_{k}^{[a(j)]},\  \ c_k, c_{k+1}, \cdots, c_{k+l-1}) \\
        &= f(\alpha_{k+l-1}^{[c]}, \ \beta_{k}^{[c]},\  \  c_k, c_{k+1}, \cdots, c_{k+l-1})\\
        &= P_k^{[c]}\\
    \end{aligned}
    $$
    \end{itemize}
\end{proof}

\subsection{Generalized Properties}
We generalize the proof appeared in \cite{Cusick_Flahive_Markoff_Lagrange} to prove some key properties related to the generalized spectra. The only condition needed is the good continuity condition.
\subsubsection{Directional Equivalence}
We consider sequences that are infinite in one direction (mono-infinite), and in both directions (bi-infinite), under the limit of $\liminf$ or $\limsup$, but not $\inf$ or $\sup$. Mono-infinite sequences are usually related to continued fraction expansions in number theory.

\begin{theorem}
    Assume continuity for fractional parameters of $f$, then these two are equal:
    $$\left\{ \liminf_{k \in \Z}  \left( P_k^{[a]} \right) \ \middle| \ a:\Z \to  \N^+  \ \right\} = \left\{ \liminf_{k \in \N}  \left( P_k^{[b]} \right) \ \middle| \ b:\N \to  \N^+  \ \right\}$$
    The same relation holds if $\liminf$ is replaced with $\limsup$.
\end{theorem}

\begin{proof}
    The proof is simple. For $a: \Z \to \N^+$, WLOG $\{n_j\} \to +\infty$ satisfies $$\liminf_{k \in \Z}  \left( P_k^{[a]} \right) = \lim _{j \to +\infty} P_{n_j}^{[a]}$$
    consider $b: \N \to \N^+$ such that $b_n = a_n$, then 
    $$\liminf_{k \in \Z}  \left( P_k^{[a]} \right) = \liminf_{k \in \N}  \left( P_k^{[b]} \right)$$
    Conversely, $b: \N \to \N^+$, we simply construct $a$ as $a_n = b_{|n|}$ would work.
\end{proof}

\subsubsection{Inclusion}
We also prove that the $\liminf$ spectrum is included in the $\inf$ spectrum, and similarly that $\limsup$ spectrum is included in the $\sup$ spectrum.
\begin{theorem}
    \label{liminclusion}
    Let $f$ be a Perron function with good continuity. Then the following holds:
    $$\sigma'_{f, \liminf} \subset \sigma_{f, \inf}$$
    Similarly
    $$\sigma'_{f, \limsup} \subset \sigma_{f, \sup}$$
\end{theorem}
\begin{proof}
    We only prove the first line concerning $\liminf$ and $\inf$, and the second line is done analogously.
    Suppose that the sequence of $\{ n_j\}$ satisfies that
    $$\liminf_{k \in \Z}  \left( P_k^{[a]} \right) = \lim _{j \to +\infty} P_{n_j}^{[a]}$$
    We define $a(j)_k := a_{n_j+k}$ and let $\{c_r\}$ be its bidirectional accumulation sequence. Directly applying Lemma \ref{accumulation} we have
    $$\liminf_{k \in \Z}  P_k^{[a]} = \lim _{j \to +\infty} P_{n_j}^{[a]} = P_0^{[c]}$$
    Then we show that $P_0^{[c]} = \inf_{k \in \Z} P_k^{[c]}$, i.e., 
    $$\forall k \quad P_0^{[c]} \le P_k^{[c]}$$
    This is because that similarly we have a subsequence $\{j_t\}$ such that
    $\lim_{t \to \infty} P_k^{[a(j_t)]}$ exists:
    $$P_k^{[c]} = \lim_{t \to \infty} P_k^{[a(j_t)]} = \lim_{t \to \infty} P_{n_{j_t} + k}^{[a]} \ge \liminf_{k \in \Z}  P_k^{[a]} = P_0^{[c]}$$
    The above tells us that any $\liminf_{k \in \Z}  \left( P_k^{[a]} \right) \in \sigma'_{f,\liminf}$ is contained in $\sigma_{f,\inf}$, proving the inclusion.
\end{proof}

\subsubsection{Closedness}
We also show that closedness is a universal property across the four spectra:
$$\sigma'_{f,\liminf}, \ \sigma'_{f,\limsup},\quad  \sigma_{f,\inf},\  \sigma_{f,\sup}$$
The proof relies on only the good continuity of $f$, so this is a very general result. 

\begin{lemma}
    \label{seqconstruct}
    Assume good continuity of $f$, for sequence $\{a_n\}$, there exists a sequence $\{c_n\}$ such that
    $$ \inf_{n \in \Z} P_n^{[a]} = \inf_{n \in \Z} P_n^{[c]} = P_0^{[c]}$$
    The $\inf$ in the statement can be safely replaced with with $\sup$, and the proposition still holds.
\end{lemma}

\begin{proof}
    This proof is analogous to Theorem \ref{lem6bookanalog} above, or Lemma 6 in book \cite{Cusick_Flahive_Markoff_Lagrange}.

    If there exists $n_0$ such that $\inf_{n \in \Z} P_n^{[a]} = P_{n_0}^{a}$, we simply set $c_n = a_{n-n_0}$. Otherwise, the sequence $a$ is at least infinite in one direction.

    Assume the sequence $\{n_j\}$ satisfies that
    $$ \inf_{n \in \Z} P_n^{[a]} = \lim _{j \to \infty}  P_{n_j}^{[a]}$$
    and moreover, we assume that $P_{n_j}^{[a]}$ is monotonically decreasing for $j \in \N^+$. We define $a(j)_k := a_{n_j+k}$ and define the sequence $\{c_n\}$ as its bidirectional accumulation sequence.
    
    \begin{itemize}
        \item First, using Lemma \ref{accumulation},
        $$\lim _{j \to \infty}  P_{n_j}^{[a]} = P_0^{[c]}$$.
        \item Then we prove that
        $$P_0^{[c]} = \inf_{k \in \Z} P_k^{[c]}$$
        This is because that similarly we can show for a subsequence $\{j_t\}$ that 
        $$\lim _{t \to +\infty} P_{n_{j_t}+k}^{[a]} = P_k^{[c]}$$
        also
        $$\lim _{t \to +\infty} P_{n_{j_t}+k}^{[a]} \ge \inf_{n \in \Z} P_n^{[a]} = P_0^{[c]}$$
    \end{itemize}
    The similar arguments can be made for $\sup$ spectrum.
\end{proof}

\begin{theorem}
    \label{infclosed}
    Under the assumption of good continuity of $f$, these 2 spectra are closed sets:
    $$\sigma_{f,\inf}, \quad \sigma_{f,\sup}$$
\end{theorem}

\begin{proof}
    This theorem is proved by applying Lemma \ref{seqconstruct}, and is analogous to Theorem \ref{mg2closed}.

    For any set of index sequences $\{a(j)_k\}$, where $j = 1,2, \cdots$, and $k \in \Z$, if
    $$ \lim _{j \to \infty} \inf_{k \in \Z} P_{k}^{[a(j)]} = y$$
    we prove that there exists a sequence $\{c_k\}_{k \in \Z}$ such that
    $\inf_{k \in \Z} P_{k}^{[c]} = y$.
    
    Using Lemma \ref{seqconstruct}, we can further assume that
    $$ \inf_{k \in \Z} P_{k}^{[a(j)]} = P_0^{[a(j)]}$$

    We let the sequence $\{c_k\}_{k \in \Z}$ be the bidirectional accumulation sequence of $a(j)_k$. All we need is to verify that
    \begin{itemize}
        \item The Perron's formula for sequence $c$ at offset $0$ is exactly $y$: $P_{0}^{[c]} = y$. This is obvious by Lemma \ref{accumulation}.
        \item Verify $\inf_{k \in \Z} P_{k}^{[c]} = P_{0}^{[c]}$.
        By Lemma \ref{accumulation}, it's direct to show for a subsequence
        $$ P_{k}^{[c]} = \lim_{t \to \infty} P_{k}^{[a(j_t)]} \ge \lim_{t \to \infty} P_{0}^{[a(j_t)]} = y $$
    \end{itemize}
    The rest (proving the $\sup$ spectrum) can be done with same steps, or just replacing the Perron function $f$ with $-f$.
\end{proof}

The proof of closedness for the other 2 spectra requires explicit construction of closure.

\begin{definition}
    We define some periodic spectral sets as follows:
    $$\tau'_{f, \liminf} := \left\{ \liminf_k  \left( P_k^{[a]} \right) \ \middle| \ a:\Z \to   \N^+ \text{ is a periodic sequence} \right\}$$
    $$\tau'_{f, \limsup} := \left\{ \limsup_k  \left( P_k^{[a]} \right) \ \middle| \ a:\Z \to   \N^+ \text{ is a periodic sequence} \right\}$$
\end{definition}
Since $a$ is a periodic sequence, $P_k^{[a]}$ has only finitely many attainable values for offset $k$. Thus it is equivalent to replace $\liminf$ with $\inf$ or $\limsup$ with $\sup$ in the above definition.

\begin{theorem}
    \label{liminfclosed}
    Assume good continuity of $f$. The spectra $\sigma'_{f, \liminf}$ and $\sigma'_{f, \limsup}$ are respectively the closure of $\tau'_{f, \liminf}$ and $\tau'_{f, \limsup}$, or
    $$\sigma'_{f, \liminf} = \mathrm{cl} \left( \tau'_{f, \liminf} \right), \quad \sigma'_{f, \limsup} = \mathrm{cl} \left( \tau'_{f, \limsup} \right)$$
\end{theorem}

\begin{proof}
    The proof mainly follows Theorem 2 in Chapter 3 of book \cite{Cusick_Flahive_Markoff_Lagrange}.

    For simplicity, we can further assume that the Perron function $f$ takes value between the interval $[-1, 1]$ (which is topologically homeomorphic to $[-\infty, \infty]$). This simplifies the discussion about neighborhoods.

    First, for any sequence $a$ and arbitrarily small $\epsilon > 0$, we construct periodic sequence $p$ such that
    $$ \left\lvert \liminf_k  \left( P_k^{[a]} \right) -   \liminf_k  \left( P_k^{[p]} \right) \right\rvert < \epsilon$$
    Like in the proof of Theorem \ref{liminclusion}, we assume that
    $$ \liminf_k  \left( P_k^{[a]} \right) = \lim_{j\to \infty}  P_{n_j}^{[a]}$$
    We construct bidirectional accumulation sequence $\{c_k\}$ of $\{a(j)_k := a_{n_j+k}\}$ according to Lemma \ref{accumulation}, and choose a subsequence of $\{n_j\}$, denoted by $\{d_j\}$ that satisfies the following:
    \begin{itemize}
        \item $d_{j+1} - d_j > 3j+3$;
        \item $d_{j+1} - d_j$ is an increasing sequence for $j$.
    \end{itemize}
    The periodic sequence, $\{p(j)\}_k$, is defined as:
    $$p(j)_k = a_{d_j + (k \!\! \mod (d_{j+1} - d_j))}$$
    where $p(j)$ has the period of
    $$(\overline{a_{d_j}, a_{d_j+1}, \cdots, a_{d_{j+1} - 1}})$$
    We will later show that if $j$ is large enough, then 
    $$ \left\lvert \liminf_k  \left( P_k^{[a]} \right) -   \liminf_k  \left( P_k^{[p(j)]} \right) \right\rvert < \epsilon$$
    The choice of $d_j$ requires further dependencies on the structure of $\{c_k\}$, and we will discuss by cases below:
    \begin{itemize}
        \item (Bi-infinite) $\{c_k\}$ extends in both directions.
        We further require that for any $j$, the following are equal
        $$\left(a_{d_j - j}, \cdots, a_{d_j + j}\right) = \left(c_{-j}, \cdots, c_j\right)$$
        For this part, the proof coincides with Theorem 2 in Chapter 3 of book \cite{Cusick_Flahive_Markoff_Lagrange}.
        
        \item (Mono-infinite) $\{c_k\}$ extends in only one direction. WLOG, we consider the case where 
        $$c_{-s-1} = \infty, c_{-s}, \cdots, c_0, \cdots$$
        and further assume that for any $j$, the following holds:
        $$\left(a_{d_j - s}, \cdots, a_{d_j + j}\right) = \left(c_{-s}, \cdots, c_j\right)$$
        $$a_{d_j-s-1} > 2^j$$
        We show that for $j$ sufficiently large, and for $u = 0, 1, \cdots, d_{j+1} - d_j$, 
        $$\left\lvert P_{d_j+u}^{[a]} -  P_{u}^{[p(j)]}  \right\rvert < \epsilon / 2$$
        Since $a_{d_j + u} = p(j)_{u}$ for $u = -s, \cdots, 0, \cdots, d_{j+1} - d_j +j$, we can show that $\left\lvert \alpha_{d_j+u+l-1}^{[a]} - \alpha_{u+l-1}^{[p(j)]}\right\rvert < 2^{l-j}$ and $\left\lvert \beta_{d_j+u}^{[a]} - \alpha_{u}^{[p(j)]}\right\rvert < 2^{-j}$, controlling the term of $\left\lvert P_{d_j+u}^{[a]} -  P_{u}^{[p(j)]}  \right\rvert$ (by uniform continuity of $f$ to $(\alpha, \beta)$).
        
        \item (Finite) $\{c_k\}$ has finitely many terms, which are:
        $$c_{-s-1} = \infty,\ \   c_{-s}, \cdots, c_0, \cdots, c_{t}, \ \  c_{t+1} = \infty$$
        so we can further assume that for any $j$, the following holds:
        $$\left(a_{d_j - s}, \cdots, a_{d_j + t}\right) = \left(c_{-s}, \cdots, c_t\right)$$
        $$a_{d_j-s-1} > 2^j,\ \  a_{d_j+t+1} > 2^j$$
        We also show that for $j$ sufficiently large, and for $u = 0, 1, \cdots, d_{j+1} - d_j$, 
        $$\left\lvert P_{d_j+u}^{[a]} -  P_{u}^{[p(j)]}  \right\rvert < \epsilon / 2$$
        This is done though discussing two branches:
        \begin{itemize}
            \item For $u \in [0, t-l+1] \cup [t+2, (n_{j+1} - n_j) -s-l-1]$, the proof follows from above. 
            \item For $u \in [t-l+2, t+1] \cup [(n_{j+1} - n_j) -s-l, (n_{j+1} - n_j)]$, by uniform continuity of $f$'s integral parameters, $P_{d_j+u}^{[a]}$ and $P_{u}^{[p(j)]}$ converges to $I$ uniformly.
        \end{itemize}
        
    \end{itemize}
    The above shows that $\sigma'_{f, \liminf} \subset \mathrm{cl} \left( \tau'_{f, \liminf} \right)$. 
    
    Now we prove the reverse inclusion.

    Suppose that $p(j)_k$ is a series of periodic sequences, such that
    $P_0^{[p(j)]} = \liminf_{k} P_k^{[p(j)]} $, with $\lim_{j \to \infty} P_0^{[p(j)]} = y$. We construct a sequence $\{a_k\}$ such that $\liminf_{k \to \infty} P_k^{[a]} = y$.

    We construct bidirectional accumulation sequence $\{c_k\}$ of $\{p(j)_k\}$ according to Lemma \ref{accumulation}, and discuss by cases of $\{c_k\}$:

    \begin{itemize}
        \item (Bi-infinite) $\{c_k\}$ extends in both directions.
        By resetting $\{p(j)\}$ to be a subsequence of itself, we further require that for any $j$, the following are equal
        $$\left(p(j)_{-j}, \cdots, p(j)_{j}\right) = \left(c_{-j}, \cdots, c_j\right)$$
        we also require that the period of $p(j)$ should be increasing and greater than $3j$.
        
        This part coincides with Theorem 2 in Chapter 3 of book \cite{Cusick_Flahive_Markoff_Lagrange}.
        
        \item (Mono-infinite) $\{c_k\}$ extends in only one direction. We further assume that for any $j$,
        $$\left(p(j)_{- s}, \cdots, p(j)_{j}\right) = \left(c_{-s}, \cdots, c_j\right)$$
        $$p(j)_{-s-1} > 2^j$$
        we also require that the period of $p(j)$ should be increasing and greater than $2j+s+1$.
        
        \item (Finite) $\{c_k\}$ has finitely many terms, which are:
        $$c_{-s-1} = \infty,\ \   c_{-s}, \cdots, c_0, \cdots, c_{t}, \ \  c_{t+1} = \infty$$
        so we assume that:
        $$\left(p(j)_{-s}, \cdots, p(j)_{t}\right) = \left(c_{-s}, \cdots, c_t\right)$$
        $$p(j)_{-s-1} > 2^j,\ \  p(j)_{t+1} > 2^j$$
        we also require that the period of $p(j)$ should be increasing and greater than $j+s+t+2$.
    \end{itemize}

    We construct $\{d_j\}$ and $\{a_n\}$ as:
    $$
    \begin{aligned}
        d_1 &= 1, \\
        d_{j+1} &= d_j + (\text{period length of } p(j)), \\
        \left(a_{d_j}, \cdots, a_{d_{j+1} - 1} \right) &= \left(p(j)_0, \cdots, p(j)_{d_{j+1} - d_{j} - 1} \right) \\
    \end{aligned}
    $$
    It's by analogy from above to show that
    $$\left\lvert P_{d_j+u}^{[a]} -  P_{u}^{[p(j)]}  \right\rvert < \epsilon / 2 \quad (u = 0, \cdots, d_{j+1} - d_{j} - 1)$$
    and 
    $$\liminf_{k \to \infty} P_k^{[a]} = \lim_{j \to \infty} P_{d_j}^{[a]} = \lim_{j \to \infty} P_0^{[p(j)]} = y$$
    
    

    
\end{proof}

\begin{corollary}
    The Perron's formula for Lagrange and Markov spectra is:
    $$f(\alpha, \beta, a_0) = \alpha + \beta + a_0$$
    which satisfies the good continuity condition. Thus, the Lagrange spectrum is contained in the Markov spectrum, and they are both closed.
\end{corollary}

\begin{corollary}
    The Perron's formula for Dirichlet and Mordell-Gruber spectra is:
    $$f(\alpha, \beta) = \frac{1}{1+\alpha\beta}$$
    which satisfies the good continuity condition. Thus, the Dirichlet spectrum is contained in the Mordell-Gruber spectrum, and they are both closed.
\end{corollary}

\section {Generalized Hall's Segment or Ray}
It's also feasible to study whether a specific spectrum contains an interval, which we call it Hall's segment if it has a nonzero length but does not tend to infinity, and Hall's ray if it tends to infinity.

The proof of Hall's interval is more challenging and specific than the proof of closedness.
Although the existence of Hall's interval is a general result for spectra, its proof requires specific structures of the Perron function. In this chapter, we prove a lemma that extends the work of \cite{Hall_Sum_Product_CF}, which still allows us to handle more functions than just sums and products.

\subsection{Aperture Ratio for Cantor Sets}

Recall that a general Cantor set $C \subset [0,1]$ is constructed by repeating the process:
\begin{itemize}
    \item $C_0 = [0,1]$;
    \item Assume $C_n = \bigcup_{k=1}^{2^n} [a_{n,k}, d_{n,k}]$, we construct 
    $$C_{n+1} = \bigcup_{k=1}^{2^n} \left( [a_{n,k}, b_{n,k}] \cup [c_{n,k}, d_{n,k}] \right)$$
    where $a_{n,k} < b_{n,k} < c_{n,k} < d_{n,k}$.
    \item The $C$ is the limit of $\{C_n\}$:
    $$C = \bigcap_{n=0}^{\infty} C_n = \lim_{n \to \infty} C_n$$
\end{itemize}

We introduce the definition of \emph{aperture ratio} for a Cantor set: 
\begin{definition}
    For a general Cantor set $C \subset [0,1]$ constructed from the process above, the \emph{aperture ratio}, denoted by $\mathrm{Ap}(C)$, is given by the following formula:
    $$\mathrm{Ap}(C) = \sup _{n,k} \left( \max \left( \frac{c_{n,k} - b_{n,k}} {b_{n,k} - a_{n,k}} , \frac{c_{n,k} - b_{n,k}} {d_{n,k} - c_{n,k}} \right)\right)$$
\end{definition}
Note that the definition takes the supremum, which is dependent of the global structure of the Cantor set, and does not directly relate to measures or Hausdorff dimensions. To better comprehend the definition, recall that the typical Cantor ternary set has an aperture ratio of 1, since in the Cantor ternary set
$$c_{n,k} - b_{n,k} = b_{n,k} - a_{n,k} = d_{n,k} - c_{n,k} = \frac{d_{n,k} - a_{n,k}}{3}$$
In the paper of \cite{Hall_Sum_Product_CF}, it is proved that the sum of general Cantor set with itself with aperture ratio not greater than 1 contains an interval. Namely, if $C \subset [0,1]$ is a general Cantor set with $0 \le \mathrm{Ap}(C) \le 1$, then the sum with itself satisfies $$C+C = [0,2]$$

Note: we found that a concept of "lateral thickness" is defined in \cite{honary2005stable}, which is exactly the reciprocal of aperture ratio defined in this section. However, only affine Cantor sets are discussed in that paper, which is less general than piecewise $C^1$ increasing functions in the theorem below. We decide to keep the concept and notation of aperture ratio unchanged in fear of introducing inconsistencies.

\subsection{Hall's Segment Existence Theorem}
In the next step, we extend the formula of $C+C$ to the following general form
$$\left\{ g(\alpha,\beta) \ \middle | \ \alpha \in C, \ \beta \in D \right \}$$

\begin{theorem}
    \label{interval}
    Let $C, D \subset [0,1]$ with endpoints $0$ and $1$ be two general Cantor sets, whose aperture ratio is denoted as $\mathrm{Ap}(C)$ and $\mathrm{Ap}(D)$ respectively. Let $g : [0,1] \times [0,1] \to \R$ be a continuous function with piecewise $C^1$ differentiability.\footnote{i.e. the set where differentiability fails is at most a finite union of differentiable arcs. An arc is either a straight segment or intersects at most finitely many times with any segment. This ensures that $f$ restricted on a straight segment is a piecewise $C^1$ differentiable function on a closed interval, satisfying the Newton-Leibniz formula.} Furthermore, $g(\alpha, \beta)$ has positive partial derivatives at every differentiable region
    $$\frac{\partial g}{\partial\alpha} > 0, \quad \frac{\partial g}{\partial\beta} > 0$$
    so that $g$ is monotonic increasing for both $\alpha$ and $\beta$, and $g(0,0)$ is the minimum and $g(1,1)$ the maximum. Consider the level set of $\{ (\alpha, \beta)\  | \ g(\alpha, \beta) = h \}$ where $h \in (g(0,0), g(1,1))$, which is a piecewise $C^1$ differentiable curve, on which we have
    $$ \mathrm{d} g = \frac{\partial g}{\partial \alpha} \mathrm{d} \alpha + \frac{\partial g}{\partial \beta} \mathrm{d} \beta = 0$$
    If for any such $(\alpha, \beta) \in [0,1]^2$ (except piecewise discontinuities) we have
    $$ \left\lvert \frac{\partial g}{\partial \alpha} \middle/ \frac{\partial g}{\partial \beta} \right\rvert \in \left[\mathrm{Ap}(C), \mathrm{Ap}(C)^{-1} \right] \cap \left[\mathrm{Ap}(D), \mathrm{Ap}(D)^{-1} \right]$$
    Then the set
    $$G = \left\{ g(\alpha,\beta) \ \middle | \ \alpha \in C, \ \beta \in D \right \}$$
    saturates the whole interval: $G = [g(0,0), g(1,1)]$. 

    It's worth noting that the interval $[0,1]$ is solely for convenience, and it's possible to substitute it with any $[x_1, x_2]$ such that $x_1 < x_2$. 
\end{theorem}

\begin{proof}
    This proof does not really resemble the steps of  \cite{Hall_Sum_Product_CF}, as we develop different techniques.  
    The essence of this proof is to construct nested rectangular regions, denoted by $B_0, B_1, \cdots$ that satisfies $B_0 = [0,1]^2 \supset B_1 \supset \cdots$, and prove nonempty intersection that
    $$\forall s \in \N \quad B_s \cap g^{-1}(h) \ne \emptyset$$
    so that
    $$B_\infty = \bigcap_{s=0}^{\infty} B_s \ne \emptyset, \quad B_\infty \cap g^{-1}(h) \ne \emptyset, \quad B_\infty \cap g^{-1}(h) \cap (C \times D)\ne \emptyset$$
    
    For simplicity, we first define $r = \max(\mathrm{Ap}(C), \mathrm{Ap}(D))$, so that 
    $$ \left\lvert \frac{\partial g}{\partial \alpha} \middle/ \frac{\partial g}{\partial \beta} \right\rvert \in \left[r, 1/r\right]$$
    
    We inductively construct the rectangles $B_{s} \ (s \in \N)$, where $B_s = [a_s, d_s] \times [a'_s, d'_s]$. Suppose that $$C_n = \bigcup_{k=1}^{2^n} [a_{n,k}, d_{n,k}]$$ where $[a_s, d_s]$ is one of its intervals, and $$D_m = \bigcup_{k=1}^{2^m} [a'_{m,k}, d'_{m,k}]$$ where $[a'_s, d'_s]$ is one of its intervals. we call $n$ and $m$ respectively the "level" of $[a_s, d_s]$ and $[a'_s, d'_s]$.

    We need some auxiliary induction hypotheses to help us with inductive construction:
    \begin{itemize}
        \item $B_s \cap g^{-1}(h) \ne \emptyset$;
        \item Let $e = \min(d_s-a_s, d'_s-a'_s)$, then $g^{-1}(h)$ passes one of the two squares:
        $$Q_1 = [a_s, a_s+e] \times [a'_s, a'_s+e], \quad  Q_2 = [d_s-e, d_s] \times [d'_s-e, d'_s]$$
    \end{itemize}
    
    Now we begin our construction of $\{B_s\}$ alongside with inductively proving the auxiliary hypotheses:
    \begin{itemize}
        \item $s=0$: $B_0 = [0,1] \times [0,1]$. 
        Obviously: $B_s \cap g^{-1}(h) = [0,1]^2 \cap g^{-1}(h) \ne \emptyset$, also $[a_0, a_0+e] \times [a'_0, a'_0+e] = [0,1] \times [0,1]$
        which satisfies our induction hypotheses.
        \item $s \to s+1$: 
        Without losing of generality, we assume $d_s - a_s \ge d'_s - a'_s$, so that $e = d'_s - a'_s$. We further assume that $g^{-1}(h)$ passes $Q_1 = [a_s, a_s+e] \times [a'_s, a'_s+e] = [a_s, a_s+d'_s-a_s] \times [a'_s, d'_s]$.
        
        Suppose that $[a_s, d_s]$ is further divided into $[a_s, b_s] \cup [c_s, d_s] \subset C_{n+1}$. We take these 2 sub-rectangles of $B_s$:
        $$
        \begin{aligned}
            V_1 &= [a_s, b_s] \times [a'_s, d'_s] \quad \text{(left),} \\
            V_2 &= [c_s, d_s] \times [a'_s, d'_s] \quad \text{(right)} \\
        \end{aligned}
        $$

        We show that $g^{-1}(h)$ intersects one of $V_1$ or $V_2$. By the intermediate value theorem of continuous functions, if $g(b_s, d'_s) \ge h$ or $g(c_s, a'_s) \le h$, then $g^{-1}(h)$ intersects $V_1$ or $V_2$, respectively. If it fits none of the two cases, then $g(b_s, d'_s) < g(c_s, a'_s)$. 

        Now let $$\gamma: [0,1] \to \R^2, \quad \gamma(t) = (t c_s + (1-t) b_s, ta'_s+(1-t)d'_s)$$
        so that
        $$
        \begin{aligned}
            \frac{\mathrm{d}}{\mathrm{d}t} g(\gamma(t)) = \frac{\partial g}{\partial \alpha} \frac{\mathrm{d} \gamma_x}{\mathrm{d} t} + \frac{\partial g}{\partial \beta} \frac{\mathrm{d} \gamma_y}{\mathrm{d} t} &= \frac{\partial g}{\partial \alpha} (c_s-b_s)  + \frac{\partial g}{\partial \beta} (a'_s - d'_s) \\
        \end{aligned}
        $$
        If $(c_s-b_s) \le r (d'_s - a'_s)$, then 
        $$
        \begin{aligned}
            \frac{\mathrm{d}}{\mathrm{d}t} g(\gamma(t)) &= \frac{\partial g}{\partial \alpha} (c_s-b_s)  + \frac{\partial g}{\partial \beta} (a'_s - d'_s) \le 0\\
        \end{aligned}
        $$
        which implies that $g(b_s, d'_s) \ge g(c_s, a'_s)$, contradiction; so 
        $(c_s-b_s) > r (d'_s - a'_s)$.
        Since by aperture ratio, ${(c_s-b_s)}/{(b_s-a_s)} < r$, which implies $(b_s - a_s) > (d'_s - a'_s)$, and $(d_s - c_s) > (d'_s - a'_s)$. So $Q_1 \subset V_1$ and $Q_2 \subset V_2$. By our assumption, $g^{-1}(h) \cap Q_1 \subset g^{-1}(h) \cap V_1 \ne \emptyset$, which is a contradiction.

        Now, we let $e_1 = \min(b_s - a_s, d'_s-a'_s)$, $e_2 = \min(d_s - c_s, d'_s-a'_s)$ and take four smaller squares
        $$
        \begin{aligned}
            Q_3 &:= [a_s, a_s+e_1] \times [a'_s, a'_s+e_1], \quad  Q_4 := [b_s-e_1, b_s] \times [d'_s-e_1, d'_s] \\
            Q_5 &:= [c_s, c_s+e_2] \times [a'_s, a'_s+e_2], \quad  Q_6 := [d_s-e_2, d_s] \times [d'_s-e_2, d'_s] \\
        \end{aligned}
        $$
        so that $Q_3, Q_4 \subset V_1$ and $Q_5, Q_6 \subset V_2$. Their relationships are visualized in the following figure \ref{fig:visualcantorset}.

        \begin{figure}
            \centering
            \includegraphics[width=0.5\textwidth]{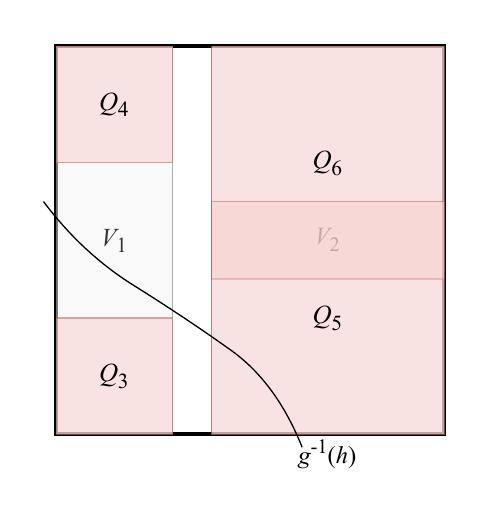}
            \caption{A visualization of the relationships between objects in this proof.}
            \label{fig:visualcantorset}
        \end{figure}
        
        \begin{itemize}
            \item For $b_s - a_s \ge d'_s-a'_s$, then $Q_3 = Q_1$. We take $B_{s+1} = V_1$.
            
            \item For $b_s - a_s < d'_s-a'_s$, but $d_s - c_s \ge d'_s-a'_s$, we show that $g^{-1}(h)$ passes $Q_3$ or $Q_5$. In other words, by the intermediate value theorem, we need to show
            $$h \in \left[g(a_s, a'_s),g(a_s+e_1, a'_s+e_1)\right] \cup \left[g(c_s, a'_s),g(c_s+e_2, a'_s+e_2)\right]$$
            Note that
            $$g^{-1}(h) \cap C_1 \ne \emptyset \iff h \in [g(a_s, a'_s), g(a_s+e, a'_s+e)]$$
            also $$g(a_s+e, a'_s+e) < g(c_s+e_2, a'_s+e_2)=g(c_s+e, a'_s+e)$$ so we only need to show that 
            $$g(a_s+e_1, a'_s+e_1) = g(b_s, a'_s+e_1) \ge g(c_s, a'_s)$$
            Since $c_s - b_s \le re_1 = r(b_s-a_s)$, this is proved by constructing a segment $\gamma: [0,1] \to \R^2, \gamma(0)=(b_s, a'_s+e_1), \gamma(1) = (c_s, a'_s)$ and showing that $(g \circ \gamma)' \le 0$.

            If $g^{-1}(h)$ passes $Q_3$, we take $B_{s+1} = V_1$. Otherwise, $g^{-1}(h)$ passes $Q_5$, so we take $B_{s+1} = V_2$.
            
            \item For $b_s - a_s < d'_s-a'_s$, and $d_s - c_s < d'_s-a'_s$, we show that $g^{-1}(h)$ passes one of $Q_3, Q_4, Q_5, Q_6$. 

            Consider $g(Q_3)$ and $g(Q_5)$:
            $$\left[g(a_s, a'_s),g(a_s+e_1, a'_s+e_1)\right] \cup \left[g(c_s, a'_s),g(c_s+e_2, a'_s+e_2)\right]$$
            Exactly same from above, $g(a_s+e_1, a'_s+e_1) = g(b_s, a'_s+e_1) \ge g(c_s, a'_s)$.
            
            Similarly, considering $g(Q_4)$ and $g(Q_6)$, it can be shown that $g(b_s, d'_s) \ge g(d_s -e_2, d'_s-e_2) = g(c_s, d'_s-e_2)$. Now the only thing remaining is to show that
            $$g(c_s+e_2, a'_s+e_2) = g(d_s, a'_s+e_2) \ge g(b_s - e_1, d'_s - e_1) = g(a_s, d'_s - e_1) $$
            This is simple, since 
            $$
            \begin{aligned}
                d'_s - e_1 - (a'_s+e_2) &\le d_s - e_1 - a_s - e_2 \\
                &= c_s - b_s \le r(b_s - a_s) \\
                &< r(d_s - a_s)
            \end{aligned}
            $$
            the final step is to construct the segment $\gamma$ from $(d_s, a'_s+e_2)$ to $g(a_s, d'_s - e_1)$ and show that $(g \circ \gamma)' \ge 0$.

            If $g^{-1}(h)$ passes $Q_3$ or $Q_4$, we take $B_{s+1} = V_1$. Otherwise, $g^{-1}(h)$ passes $Q_5$ or $Q_6$, so we take $B_{s+1} = V_2$.
        \end{itemize}
    \end{itemize}
    We obtain $B_\infty$ as the intersection of all $B_s$, which is a closed rectangle by the nested interval theorem. 
    
    However, please note that we have not really finished our proof, since not necessarily $B_\infty \subset C\times D$. Now, denote that
    $$B_\infty = B_{\infty, 0} = [a_{\infty, 0}, d_{\infty, 0}] \times [a'_{\infty, 0}, d'_{\infty, 0}]$$
    Recalling our definition of "level" above, we are not sure if both the level of $[a_{s}, d_{s}]$ and $[a'_{s}, d'_{s}]$ tends to infinity. Actually, if they both tend to infinity, then we are done by $B_\infty \subset C\times D$ and $B_\infty \cap g^{-1}(h) \ne \emptyset$. If not, we assume that the level of $[a_{s}, d_{s}]$ does not tend to infinity. Then:
    \begin{itemize}
        \item When $s$ is sufficiently large, $[a_{s}, d_{s}]$ remains unchanged:
        $$\exists s_0 \in \N^+ \quad \forall s\ge s_0 \quad [a_{s}, d_{s}] = [a_{s_0}, d_{s_0}] = [a_{\infty, 0}, d_{\infty, 0}]$$
        \item The level of $[a'_{s}, d'_{s}]$ tends to infinity.
        \item The following inequality holds: $$d'_{\infty, 0} - a'_{\infty, 0} \ge d_{\infty, 0} - a_{\infty, 0}$$
        \item The intersection:
        $$(C \times D) \cap B_{\infty} = \left([a_{s_0}, d_{s_0}] \cap C\right) \times [a'_{\infty,0}, d'_{\infty,0}]$$
    \end{itemize}

    We resolve this case by inductively constructing $B_{\infty, t}$:
    \begin{itemize}
        \item $t=0$: $B_{\infty, 0} = B_\infty$;
        \item $t \to t+1$: Assume that 
        $$B_{\infty, t} = [a_{\infty, t}, d_{\infty, t}] \times [a'_{\infty, 0}, d'_{\infty, 0}]$$
        such that $B_{\infty, t} \cap g^{-1}(h) \ne \emptyset$. We break the interval
        $$[a_{\infty, t}, d_{\infty, t}] \to [a_{\infty, t}, b_{\infty, t}] \cup [c_{\infty, t}, d_{\infty, t}]$$
        and let
        $$ 
        \begin{aligned}
            R_1 &= [a_{\infty, t}, b_{\infty, t}] \times [a'_{\infty, 0}, d'_{\infty, 0}] \\ R_2 &= [c_{\infty, t}, d_{\infty, t}] \times [a'_{\infty, 0}, d'_{\infty, 0}]
        \end{aligned}
        $$
        To verify that $g^{-1}(h)$ passes $R_1$ or $R_2$, one just show that
        $$g(b_{\infty, t}, d'_{\infty, 0}) \ge g(c_{\infty, t}, a'_{\infty, 0})$$
        This is easily solved because 
        $$c_{\infty, t} - b_{\infty, t} < r(d_{\infty, t} - a_{\infty, t}) \le r(d_{\infty, 0} - a_{\infty, 0}) \le r(d'_{\infty, 0} - a'_{\infty, 0}) $$
        And finally, let $B_{\infty, t+1}$ be one of $R_1$ or $R_2$ which is intersected by $g^{-1}(h)$.
    \end{itemize}

    The very last step is obvious: Let $B_{\infty, \infty} := \bigcap _{t=0}^{\infty} B_{\infty, t}$, which is a rectangle contained in $C \times D$, and $B_{\infty, \infty} \cap g^{-1}(h) \ne \emptyset$, which finishes our proof.

\end{proof}

\subsection{The Set $F(4)$}
What is the aperture ratio of the following set researched in \cite{Hall_Sum_Product_CF}?
$$F(4) = \{ [0; a_1, a_2, \cdots] | a_i \in \{1,2,3,4\}\}$$
According co the computations in \cite{Hall_Sum_Product_CF}, we have
$$
\begin{aligned}
    \mathrm{Ap}(F(4)) \le \max &\left( \frac{2\sqrt{2} - 1}{3}, \ \frac{6\sqrt{2} - 2}{17},\  \frac{2\sqrt{2} - 1}{3}, \ \right. \\
    &\left. \frac{11\sqrt{2} -2}{34},\  \frac{2\sqrt{2} - 1}{3},\  \frac{10\sqrt{2} +2}{21} \right)
\end{aligned}
$$
\begin{lemma}
    The set $F(4)$ has aperture ratio
    $$
        \mathrm{Ap}(F(4)) \le \frac{10\sqrt{2} +2}{21} < 0.7687
    $$
\end{lemma}
This means that set $F(4)$ of bounded continued fractions is available for proving the existence of Hall's segment for certain Perron's functions.

\section{Some Applications}
\subsection{Application 1: Analog of $q_{k+m} \lVert q_k \alpha \rVert$ Spectra} 
In this application, we research the following spectra for arbitrary $m$:
$$ I_m = \left\{ \liminf_{k \to \infty} q_{k+m} \lVert q_k \alpha \rVert \ \middle| \ \alpha \in [0,1] \setminus \Q \right\}  \quad (\text{for } m \le 0)$$
$$ S_m = \left\{ \limsup_{k \to \infty} q_{k+m} \lVert q_k \alpha \rVert \ \middle| \ \alpha \in [0,1] \setminus \Q \right\} \quad (\text{for } m \ge 1)$$
Where $q_k$ means the denominator of the $k$-th best approximant of $\alpha$.

When $m=0$, $I_0$ is the reciprocal of the Lagrange spectrum. When $m=1$, $S_1$ is the Dirichlet spectrum.

The question rises that how can we compute the Perron's formula for $q_{k+m} \lVert q_k \alpha \rVert$ spectrum efficiently.
\begin{itemize}
    \item For $m > 0$, we have the following identity:
    $$
    \begin{aligned}
        q_{k+m} \lVert q_k \alpha \rVert &= \frac{\prod_{i=k}^{k+m-1}  q_{i+1} \lVert q_i \alpha \rVert}{\prod_{i=k+1}^{k+m-1}  q_{i} \lVert q_i \alpha\rVert}
    \end{aligned}
    $$
    \item For $m \le 0$, we have the following identity:
    $$
    \begin{aligned}
        q_{k+m} \lVert q_k \alpha \rVert &= \frac{\prod_{i=k+m}^{k} q_{i} \lVert q_i \alpha\rVert }{\prod_{i=k+m}^{k-1}  q_{i+1} \lVert q_i \alpha \rVert}
    \end{aligned}
    $$
\end{itemize}
The above tools directly enables us to prove that $I_m$ and $S_m$ are closed and contain intervals.

\begin{theorem}
    Both $I_m (m \le 0)$ and $S_m (m \ge 1)$ are closed.
\end{theorem}
\begin{proof}
    Below we denote $\alpha$'s continued fraction expansion as $[0; a_1, a_2, \cdots]$.
    For $S_m(m \ge 1)$ part, we have already proved the case where $m=1$. Now consider $S_m (m \ge 2)$. Using the identity above, we rewrite
    $$
    \begin{aligned}
        \limsup_{k \to \infty}  \left( q_{k+m} \lVert q_k \alpha \rVert \right) &= \limsup \frac{\prod_{i=k}^{k+m-1} \frac{1}{1+ \alpha_i^{[a]} \beta_{i+1}^{[a]}} }{\prod_{i=k+1}^{k+m-1} \frac{1}{\alpha_i^{[a]} + a_i + \beta_{i}^{[a]}}} \\
        &= \limsup_{k \to \infty} \frac{\prod_{i=k+1}^{k+m-1} \left(\alpha_i^{[a]} + a_i + \beta_{i}^{[a]} \right) }{\prod_{i=k}^{k+m-1} \left(1+ \alpha_i^{[a]} \beta_{i+1}^{[a]} \right) } \\
        &= \limsup_{k \to \infty} \frac{\prod_{i=k+1}^{k+m-1} \left( a_i + \beta_{i}^{[a]} \right) }{\left(1+ \alpha_k^{[a]} \beta_{k+1}^{[a]} \right) } \\
        &= \limsup_{k \to \infty} f_m \left( \alpha_{k+m-1}^{[a]},\beta_{k+1}^{[a]}, \ \ a_{k+1}, \cdots,  a_{k+m-1} \right)
    \end{aligned}
    $$

    Claim: $f_m$ satisfies good continuity.
    \begin{itemize}
        \item If one of $a_i \in \{ a_{k+1}, \cdots,  a_{k+m-1}\}$ is greater than $C$, then one term $\left(\alpha_i^{[a]} + a_i + \beta_{i}^{[a]} \right)$ is greater than $C$, so that 
        $$f_m \left( \alpha_{k+m-1}^{[a]},\beta_{k+1}^{[a]}, \ \ a_{k+1}, \cdots,  a_{k+m-1} \right) > C/2$$
        Indeed, $f_m$ uniformly tends to infinity as one of $\{ a_{k+1}, \cdots,  a_{k+m-1}\}$ tends to infinity.
        \item If we fix $\{ a_{k+1}, \cdots,  a_{k+m-1}\}$, then $\alpha_i^{[a]}$ is a smooth function over $\alpha_{k+m-1}^{[a]} \in [0,1]$ by
        $$\alpha_i^{[a]} = 1/\left(a_{i+1} + 1/\left(\cdots + 1/\left(a_{k+m-1} + \alpha_{k+m-1}^{[a]}\right)\cdots \right)\right)$$
        Similarly, $\beta_i^{[a]}$ is a smooth function over $\beta_{k+1}^{[a]} \in [0,1]$. This means that $f_m$ is uniformly continuous for $\left(\alpha_{k+m-1}^{[a]}, \beta_{k+1}^{[a]}\right)$.
    \end{itemize}
    By Theorem \ref{liminfclosed}, we show that $S_m$ is a closed set for any $m>0$. Same methods can be used to prove that $I_m (m \le 0)$ is a closed set.
\end{proof}

\begin{theorem}
    Both $I_m (m \le 0)$ and $S_m (m \ge 1)$ contain intervals.
\end{theorem}

\begin{proof}
    We only prove for $S_m (m \ge 1)$. If $m=1$, we have already solved the Dirichlet spectrum. If $m \ge 2$, we first construct bidirectional sequence of $\{c_n\}$ such that
    $$P_1^{[c]} =  f_m \left( \alpha_{m-1}^{[c]},\beta_{1}^{[c]}, \ \ c_{1}, \cdots,  c_{m-1} \right) = \max_{k \in \Z} P_k^{[c]}$$
    We fix $c_1 = c_2 = \cdots = c_{m-1} = 10$, and $c_j \le 4$ for $j \ne 1,2, \cdots, m-1$, so that
    $$\left( c_i + \beta_{i}^{[c]} \right)_ {i \in \{1,2, \cdots, m-1\}} \ge 10 \ge 2\left( c_i + \beta_{i}^{[c]} \right)_ {i \notin \{1,2, \cdots, m-1\}}$$
    So
    $$\prod_{i=1}^{m-1} \left( c_i + \beta_{i}^{[c]} \right) \ge 2\prod_{i=k+1}^{k+m-1} \left( c_i + \beta_{i}^{[c]} \right)$$
    and therefore $P_1^{[c]} \ge P_k^{[c]}$ for all $k \in \Z$.
    
    Since $f_m$ has the reflection symmetry:
    $$f_m (\alpha, \beta, c_1, \cdots, c_{m-1}) = f_m (\beta, \alpha, c_{m-1}, \cdots, c_2, c_{1})$$
    Rename $g(\alpha, \beta) := f_m (\alpha, \beta, 10, 10, \cdots, 10)$ for convenience, we have $g(\alpha, \beta) = g(\beta, \alpha)$, so
    $$\frac{(\partial g/ \partial \alpha)_{\alpha=\beta}} {(\partial g/ \partial \beta)_{\beta=\alpha}} = 1$$
    as long as $(\partial g/ \partial \alpha)_{\alpha=\beta} \ne 0$, which is obvious since
    $$\frac{\partial}{\partial \alpha_k^{[c]}} \frac{\prod_{i=k}^{k+m-1} \left( 10 + \beta_{i}^{[c]} \right) }{\left(1+ \alpha_k^{[c]} \beta_{k+1}^{[c]} \right)} < 0$$
    Since $g$ is smooth over $\alpha$ and $\beta$, consider a small square $[x-\epsilon, x+\epsilon]$ containing $\alpha = \beta = x$, where we can let $x=\sqrt(2) - 1 = [0; \overline{2}]$ such that for any $\epsilon$, $[x-\epsilon, x+\epsilon] \cap F(4) \ne \emptyset$.
    We further choose $\epsilon$ to control the term
    $$\frac{(\partial g/ \partial \alpha)} {(\partial g/ \partial \beta)} \in [4/5, 5/4]  \quad (\forall \alpha, \beta \in [x-\epsilon, x+\epsilon])$$
    
    Then, we can apply Theorem \ref{interval} on function $-g$ (since $g$ is decreasing for $\alpha$ and $\beta$) to show that the attainable values of $P_1^{[c]}$ contains an interval.

    The final step is to obtain sequence $\{a_n\}_{n \in \N^+}$ from $\{c_n\}_{n \in \Z}$. We do so in the following manner:
    $$ a = (c_0, \ \  c_{-1}, c_0, c_1, \ \  c_{-2}, c_{-1}, c_0, c_1, c_2, \ \  \cdots , \ \  c_{-m}, \cdots, c_0, \cdots, c_m, \ \ \cdots)$$
    We can show that
    $$\lim_{n \to \infty} P_{n^2+n+2}^{[a]} = P_1^{[c]} = \limsup_{k \to \infty} P_{k}^{[a]}$$
    and we show that $S_m$ contains an interval.
\end{proof}

\subsection{Application 2: The First Quadrant Mordell-Gruber Spectrum}

In this application, we will also consider a modified Mordell-Gruber spectrum where rectangles are limited to the first quadrant.

Let $C$ be a rectangle located in the first quadrant in $\R^2$, where $ C = [0, x] \times [0, y]$. $C$ is called admissible in a unimodular lattice $\Lambda$ if $C \cap \Lambda = \{(0,0)\}$, and redefine modified Mordell constant as 
$$\kappa^+(\Lambda) := \sup_{[0, x] \times [0, y] \text{ admissible in }\Lambda} xy$$

We obtain the first quadrant Mordell-Gruber spectrum by
\begin{definition}
    The first quadrant Mordell-Gruber spectrum is
    $$\mathrm{MG}^+_2 := \left\{ \kappa^+(\Lambda)\  \middle| \ \Lambda \text{ is a unimodular lattice in } \mathbb{R}^2\right\}$$
\end{definition}
we aim to prove some key properties relating to $\mathrm{MG}^+_2$, including closedness and the existence of Hall's ray
$$ \exists \ K > 0, \quad [K, \infty] \subset \mathrm{MG}^+_2$$

\begin{lemma}
    The Perron's formula for the first quadrant Mordell-Gruber spectrum is 
    $$
    \begin{aligned}
        P_0^{[a]} &= f(\alpha_0, \beta_0, \ a_0)\\
        &= \left\{ 
        \begin{aligned}
            &\frac{(a_0 + 2\beta_0 + 1) (a_0 + 2\alpha_0+ 1)}{4(a_0+\alpha_0+\beta_0)}, \quad &(a_0 \text{ is odd})\\
            & \frac{(a_0 + 2\beta_0 + 2) (a_0 + 2\alpha_0)}{4(a_0+\alpha_0+\beta_0)}, \quad &(a_0 \text{ is even and } \alpha_0 \ge \beta_0)\\
            & \frac{(a_0 + 2\beta_0) (a_0 + 2\alpha_0 + 2)}{4(a_0+\alpha_0+\beta_0)}, \quad &(a_0 \text{ is even and } \alpha_0 \le \beta_0)\\
        \end{aligned}
        \right.
    \end{aligned}
    $$
\end{lemma}

\begin{proof}
    First, $ C = [0, x] \times [0, y]$ doesn't contain other lattice points than the origin. Consider two consecutive pivots in the first quadrant, and by alternating principle, we further know that they have the form $\{A_{2n-1}, A_{2n+1}\}\ (n \in \Z)$. let $n$ be the largest number such that the point $A_{2n-1} = (x_{2n-1}, y_{2n-1})$ has an $x$-coordinate no smaller than $x$, so that for $A_{2n+1} = (x_{2n+1}, y_{2n+1})$ we have $0 \le x_{2n+1} \le x \le x_{2n-1}$. Also, since $A_{2n+1}$ does not fall in $C$, its $y$-coordinate is bounded by $y_{2n+1} \ge y$.

    Without losing of generality, we assume that $n=0$. Since the area of $C$ is invariant under diagonal flow action, we apply Lemma \ref{recall} to restrict $A_0$ to simplify computation. Using the computation of Theorem \ref{corre}, we have
    $$
    \begin{aligned}
        A_0 &= \left(\frac{-1}{\sqrt{a_0+\beta_0+\alpha_0}}, \frac{1}{\sqrt{a_0+\beta_0+\alpha_0}}\right) \\
        A_{1} &= \left(\frac{\alpha_0}{\sqrt{a_0+\beta_0+\alpha_0}}, \frac{a_0+\beta_0}{\sqrt{a_0+\beta_0+\alpha_0}}\right) \\
        A_{-1} &= \left(\frac{a_0+\alpha_0}{\sqrt{a_0+\beta_0+\alpha_0}}, \frac{\beta_0}{\sqrt{a_0+\beta_0+\alpha_0}}\right) \\
    \end{aligned}
    $$
    The rectangle $C$ is then bounded by points of $\Lambda$ on the segment between $A_{-1}$ and $A_1$, which are
    $$K_t = \left(\frac{(a_0-t)+\alpha_0}{\sqrt{a_0+\beta_0+\alpha_0}}, \frac{t+\beta_0}{\sqrt{a_0+\beta_0+\alpha_0}}\right) \quad (t = 0, 1, \cdots, a_0)$$
    Since we already have $x_{1} \le x \le x_{-1}$, let $t$ be the largest number such that 
    $$\frac{(a_0-t)+\alpha_0}{\sqrt{a_0+\beta_0+\alpha_0}} \ge x > \frac{(a_0-t-1)+\alpha_0}{\sqrt{a_0+\beta_0+\alpha_0}} $$
    then, $y$ is controlled by $K_{t+1}$:
    $$y \le \frac{(t+1) + \beta_0}{\sqrt{a_0+\beta_0+\alpha_0}}$$
    So the area of $C$ is bounded by
    $$xy \le \max_{t = 0,1,\cdots, a_0-1} \frac{(a_0-t+\alpha_0)(t+1 + \beta_0)}{a_0+\beta_0+\alpha_0}$$
    The numerator is a quadratic equation for $t$, which is largest when $t$ is closest to $(a_0+\alpha_0-\beta_0-1)/2$. Furthermore, the bound is tight, because we can always use $C$ to approximate the rectangle determined by $K_t$ and $K_{t+1}$.

\end{proof}

\begin{theorem}
    The first quadrant Mordell-Gruber spectrum is a closed set.
\end{theorem}

\begin{proof}
    Note that the modified Mordell constant is obtained through
    $$\kappa^+(\Lambda) = \sup_{n \in \Z} P_{2n}^{[a]}$$
    where $a$ is the index sequence of $\Lambda$. Notice that, the subscript here is $2n$, not $n$, so Theorem \ref{infclosed} is not directly applicable. Nevertheless, we make only slight modifications to Theorem \ref{infclosed} in this proof.
    
    Assume the sequence $\{n_j\}$ satisfies that
    $$ \sup_{n \in \Z} P_{2n}^{[a]} = \lim _{j \to \infty}  P_{n_j}^{[a]}$$
    and WLOG, we assume that $P_{n_j}^{[a]}$ is monotonically increasing for $j \in \N^+$. We define $a(j)_k := a_{n_j+k}$ and define the sequence $\{c_n\}$ as its bidirectional accumulation sequence.

    The rest steps follow Theorem \ref{infclosed}.
\end{proof}

\begin{theorem}
    The first quadrant Mordell-Gruber spectrum contains an entire Hall's ray:
    $$ \exists \ K > 0, \quad [K, \infty] \subset \mathrm{MG}^+_2$$
\end{theorem}

\begin{proof}
    First, it's easy to verify that $f$ is symmetric: $f(\alpha_0,\beta_0, a_0) = f(\beta_0, \alpha_0, a_0)$.
    Also we have
    $$
    \begin{aligned}
        \frac{\partial f} {\partial \alpha_0}= \left\{ 
        \begin{aligned}
            &\frac{(a_0 + 2\beta_0 + 1) (a_0+2\beta_0-1) }{4(a_0+\alpha_0+\beta_0)^2}, \quad &(a_0 \text{ is odd})\\
            & \frac{(a_0 + 2\beta_0 + 2) (a_0 + 2\beta_0)}{4(a_0+\alpha_0+\beta_0)^2}, \quad &(a_0 \text{ is even and } \alpha_0 > \beta_0)\\
            & \frac{(a_0 + 2\beta_0) (a_0 + 2\beta_0 - 2)}{4(a_0+\alpha_0+\beta_0)^2}, \quad &(a_0 \text{ is even and } \alpha_0 < \beta_0)\\
        \end{aligned}
        \right.
    \end{aligned}
    $$
    and
    $$
    \begin{aligned}
        \frac{\partial f} {\partial \beta_0}= \left\{ 
        \begin{aligned}
            &\frac{(a_0 + 2\alpha_0 + 1) (a_0+2\alpha_0-1) }{4(a_0+\alpha_0+\beta_0)^2}, \quad &(a_0 \text{ is odd})\\
            & \frac{(a_0 + 2\alpha_0 - 2) (a_0 + 2\alpha_0)}{4(a_0+\alpha_0+\beta_0)^2}, \quad &(a_0 \text{ is even and } \alpha_0 > \beta_0)\\
            & \frac{(a_0 + 2\alpha_0) (a_0 + 2\alpha_0 + 2)}{4(a_0+\alpha_0+\beta_0)^2}, \quad &(a_0 \text{ is even and } \alpha_0 < \beta_0)\\
        \end{aligned}
        \right.
    \end{aligned}
    $$
    For $a_0 \ge 1$,
    $$ \frac{\partial f} {\partial \alpha_0} \ge 0,\quad \frac{\partial f} {\partial \beta_0} \ge 0 $$
    and for $a_0 \ge 3$,
    $$ \frac{\partial f} {\partial \alpha_0} > 0,\quad \frac{\partial f} {\partial \beta_0} > 0 $$
    For $a_{2n} \in \N^+$, we have
    $$f(\alpha_{2n}, \beta_{2n}, a_{2n}) \ge 1$$
    For $a_{2n} \le 4$, we have
    $$f(\alpha_{2n}, \beta_{2n}, a_{2n}) \le f(1,1,4) = 2$$
    For $a_{0} \ge 6$, we have
    $$f(\alpha_{0}, \beta_{0}, a_0) \ge f(0,0,6) = 2$$
    Therefore by taking $a_0 \ge 6, \ a_i \le 4$ for $i \ne 0$, we are certain that $P_0^{[a]}$ is the maximum.

    Claim 1: when $a_0 \ge 40$, we show that
    $$ \left. \frac{\partial f} {\partial \alpha_0} \middle/ \frac{\partial f} {\partial \beta_0} \right. \in [4/5,5/4] $$
    We only need to show the lower bound; the upper bound is automatically obtained by symmetry.

    Since
    $$
    \begin{aligned}
        \left. \frac{\partial f} {\partial \alpha_0} \middle/ \frac{\partial f} {\partial \beta_0} \right. = \left\{ 
        \begin{aligned}
            &\frac{(a_0 + 2\beta_0 + 1) (a_0+2\beta_0-1)}{(a_0 + 2\alpha_0 + 1) (a_0+2\alpha_0-1)}, \quad &(a_0 \text{ is odd})\\
            & \frac{(a_0 + 2\beta_0 + 2) (a_0 + 2\beta_0)}{(a_0 + 2\alpha_0 - 2) (a_0 + 2\alpha_0)}, \quad &(a_0 \text{ is even and } \alpha_0 \ge \beta_0)\\
            & \frac{(a_0 + 2\beta_0) (a_0 + 2\beta_0 - 2)}{(a_0 + 2\alpha_0) (a_0 + 2\alpha_0 + 2)}, \quad &(a_0 \text{ is even and } \alpha_0 \le \beta_0)\\
        \end{aligned}
        \right.
    \end{aligned}
    $$
    When $a_0$ is even and $\alpha_0 \le \beta_0$, 
    $$\frac{(a_0 + 2\beta_0) (a_0 + 2\beta_0 - 2)}{(a_0 + 2\alpha_0) (a_0 + 2\alpha_0 + 2)} \ge \frac{(a_0- 2)a_0}{(a_0+2) (a_0+4)} \ge \frac{38 \times 40}{42 \times 44} \ge \frac{4}{5}$$
    (the other two cases are proved similarly)

    Claim 2: consider
    $$
    \begin{aligned}
        t_1 &= [0; \ \overline{4, \ 1}] = \frac{\sqrt{2}-1}{2}\\
        t_2 &= [0; \ \overline{1, \ 4}] = 2\sqrt{2}-2\\
    \end{aligned}
    $$
    then $t_1 = \min F(4)$ and $t_2 = \max F(4)$. When $a_0 \ge 40$, we have $f(t_1, t_1, a_0 + 1) \le f(t_2, t_2, a_0)$.

    \begin{enumerate}
        \item If $a_0$ is even, then we show that
        $$\frac{(a_0 + 2t_2 + 2) (a_0 + 2t_2)}{4(a_0 + 2t_2)} \ge \frac{(a_0 + 2t_1 + 2)^2}{4(a_0 + 2t_1+1)}$$
        which is equivalent to
        $$(a_0 + 2t_2 + 2) (a_0 + 2t_1+1) \ge (a_0 + 2t_1 + 2)^2$$
        or
        $$(2t_2 + 2) (2t_1+1) + (2t_2) a_0 \ge (2t_1 + 2)^2 + (2t_1+1) a_0$$
        which is true when $a_0 > 3$.
        \item If $a_0$ is odd, then we show that
        $$\frac{(a_0 + 2t_2 + 1)^2}{4(a_0 + 2t_2)} \ge \frac{(a_0 + 2t_1 + 3)(a_0+2t_1+1)}{4(a_0 + 2t_1+1)}$$
        which is equivalent to
        $$(a_0 + 2t_2 + 1)^2 \ge (a_0 + 2t_1 + 3)(a_0+2t_2)$$
        or
        $$(2t_2+1)^2 + 2t_2 a_0 \ge 2t_2(2t_1 + 3) + (2t_1+1) a_0$$
        which is always true.
    \end{enumerate}

    Applying Theorem \ref{interval},
    $$\bigcup_{a_0=40}^{\infty} [f(t_1, t_1, a_0), f(t_2, t_2, a_0)] \subset \mathrm{MG}_2^+$$
    where
    $$\bigcup_{a_0=40}^{\infty} [f(t_1, t_1, a_0), f(t_2, t_2, a_0)] = [f(t_1, t_1, 40), \infty)$$
    so that
    $$f(t_1, t_1, 40) = \frac{42+2t_1}{4} < 10.61$$
    Also, $\infty$ is also in $\mathrm{MG}_2^+$ by taking the sequence of 
    $$a_{2n-1} = 1; \quad a_{2n} = |n| + 1$$
    so the conclusion is that
    $$[10.61, \infty] \subset\mathrm{MG}_2^+$$
\end{proof}

\subsection{Application 3: The $\ell^2$ Mordell-Gruber Spectrum}
We consider another setting in this application example: Let $\Lambda$ be a unimodular lattice, and let $g_t$ be the diagonal flow action. Instead of finding a rectangle, we try to find a maximal circle in $g_t \Lambda$ centered at the origin, that contains solely the origin point in its interior.

We define the $\ell^2$ Mordell constant as
$$\kappa_2(\Lambda) := \sup_{t\in \Z} \sup_{\{x^2+y^2< r^2\} \cap (g_t\Lambda) = \{(0,0)\}} r$$
and the $\ell^2$ Mordell-Gruber spectrum is the set of all possible values of $\ell^2$ Mordell constant
$$\mathrm{MGL}_2 := \left\{ \kappa_2(\Lambda)\  \middle|  \ \Lambda \subset \R^2 \text{ unimodular}\right\}$$
\begin{lemma}
    If $r_0$ attains the supremum, then at least one pivot on the circle $x^2+y^2=r_0^2$. 
\end{lemma}
\begin{proof}
    This is obvious, since there must be points $(a, b)$ on the circle, or $r$ can be larger. Note that the rectangle $[-a,a] \times [-b,b]$ is contained in the disk $x^2+y^2\le r^2$, so $(a,b)$ is a pivot.
\end{proof}
\begin{corollary}
    An alternative formula for $\ell^2$ Mordell constant is
    $$\kappa_2(\Lambda) = \sup_{t\in \Z} \min_{A \in \Pi(g_t\Lambda)} \lVert A \rVert _{2}$$
\end{corollary}
The right side looks very similar to the log-systole function defined before, so we call it $\ell^2$ log-systole function
\begin{definition}
    The $\ell^2$ log-systole function, written as $W_2(t): \R \to \R$, is defined as
    $$W_2(t) := \ln \left(\min_{A \in \Pi(g_t\Lambda)} \lVert A \rVert _{2}\right)$$
\end{definition}
We shall first provide the upper and lower bounds for the $\ell^2$ Mordell-Gruber Spectrum:
\begin{lemma}
    \label{mgl2bound}
    The upper and lower bounds for $\ell^2$ Mordell constant:
    $$\kappa_2(\Lambda) \in \left[ 1, \ \ (4/3)^{1/4} \right]$$
    Moreover, both bounds are attainable.
\end{lemma}
\begin{proof}
    For the lower part, we need to find a $t$ so that $W_2(t) \ge 0$. 
    Assume that at $t=0$, there is a point $A=(a,b)$ such that $\lVert A \rVert_2$ attains the minimum. Since we know that
    $$f(t) := \lVert g_t A \rVert_2 = \sqrt{a^2\exp(2t)+b^2 \exp(-2t)}$$
    WLOG assume $a^2\ge b^2$, so that $f(t)$ is monotonic increasing for $t\ge 0$.

    Claim: there exists a minimum $t_0 \ge 0$ such that $\lVert g_{t_0} A \rVert_2 = \lVert g_{t_0} B \rVert_2$.

    Proof: By Minkowski convex body theorem, we obtain a bound that if $\pi \lVert g_t A \rVert_2^2 > 4$ then there exists $B$ such that $\lVert g_t A \rVert_2 > \lVert g_t B \rVert_2$. Now set
    $$t_0 = \inf \left\{ t \ge 0 \ | \ \exists B \in g_t \Lambda \quad \lVert g_t A \rVert_2 \ge \lVert g_t B \rVert_2\right\}$$
    so, there is a decreasing sequence $\{t_n\}$ with limit $t_0$ such that for each $t_n$, there exists a point $B_n$ that satisfies $$\lVert g_{t_n} A \rVert_2 \ge \lVert g_{t_n} B_n \rVert_2$$
    Since $\{B_n\}$ is bounded, there is a subsequence converging to $B \in g_{t_0} \Lambda$, where $\lVert g_{t_0} A \rVert_2 \ge \lVert g_{t_0} B \rVert_2$. Also, by the definition of infimum, the equality must hold.

    Now for $g_t \Lambda$ consider the region $D=\{x^2+y^2<\lVert g_t A \rVert_2^2\}$, which does not contain other points of $g_t\Lambda$ than the origin. Three points $A, B, (0,0)$ determine a triangle whose interior and edges are contained in the region $D$, so there are no lattice points on its edge or in its interior. By the Pick's theorem, the triangle of $A, B, (0,0)$ has an area of $1/2$ and $\lvert \det(A;B) \rvert = 1$. This implies that 
    $$1=\lvert \det(A;B) \rvert = \lVert A \rVert_2 \cdot \lVert B \rVert_2 \cdot \lvert \sin \langle A, B \rangle \rvert \le \lVert A \rVert_2 \cdot \lVert B \rVert_2 = (f(t_0))^2$$
    So we have provided a $t_0$ that satisfies $W_2(t_0) \ge 0$.
    
    When $\Lambda = \Z^2$, the equality is obtained and lower bound is achieved.

    For the upper bound, it's equivalent to prove that the area of the circle is bounded by $2\pi/\sqrt{3}$. An equivalent formulation is that an ellipse $Ax^2+Bxy+Cy^2<1$ centered at origin containing no integral points (i.e., points in $\Z^2$) except the origin, has a maximum area of $2\pi/\sqrt{3}$, which is a well-known result. 

    The upper bound is obtained by setting $\Lambda = \Z (1,0) + \Z (-1/2, \sqrt{3}/2)$.
\end{proof}

Also, we prove some key properties related to the $\ell^2$ log systole function:
\begin{lemma}
    $W_2(t)$ is a continuous function for $t$. Moreover, it's $1$-Lipschitz.
\end{lemma}
\begin{proof}
    We prove that for any $t, u$ we have $W_2(t+u) - W_2(t)< |u|$. 
    
    Assume that at $t$, the $\ell^2$ log-systole is controlled by $g_{t}A$, and suppose $g_tA = (a,b)$, so that $W_2(t) = \ln\left(\sqrt{a^2+b^2}\right)$. Now, since $g_{t+u}A = g_u(a,b)$ does not fall inside the circle with radius $\exp(W_2(t+u))$, it gives 
    $$\ln\left(\sqrt{a^2\exp(2u)+b^2 \exp(-2u)}\right) \ge W_2(t+u)$$
    also $$\ln\left(\sqrt{a^2\exp(2u)+b^2 \exp(-2u)}\right) \le \ln\left(\sqrt{a^2+b^2}\right) + |u| = W_2(t)+|u|$$
\end{proof}

\begin{lemma}
    $W_2(t)$ is controlled both above and below by the $\ell^\infty$ log-systole function:
    $$W(t) \le W_2(t) \le W(t) + \frac{\ln(2)}{2}$$
\end{lemma}
\begin{proof}
    The first inequality is obvious since for any $A$, $\lVert A \rVert_2 \ge \lVert A \rVert_\infty$. The second inequality comes from the equivalence of $\ell^2$ and $\ell^\infty$ norm of $\R^2$:
    $$\lVert A \rVert_\infty \le \lVert A \rVert_2 \le \sqrt{2} \lVert A \rVert_\infty $$
\end{proof}

\begin{figure}
    \centering
    \includegraphics[width=\textwidth]{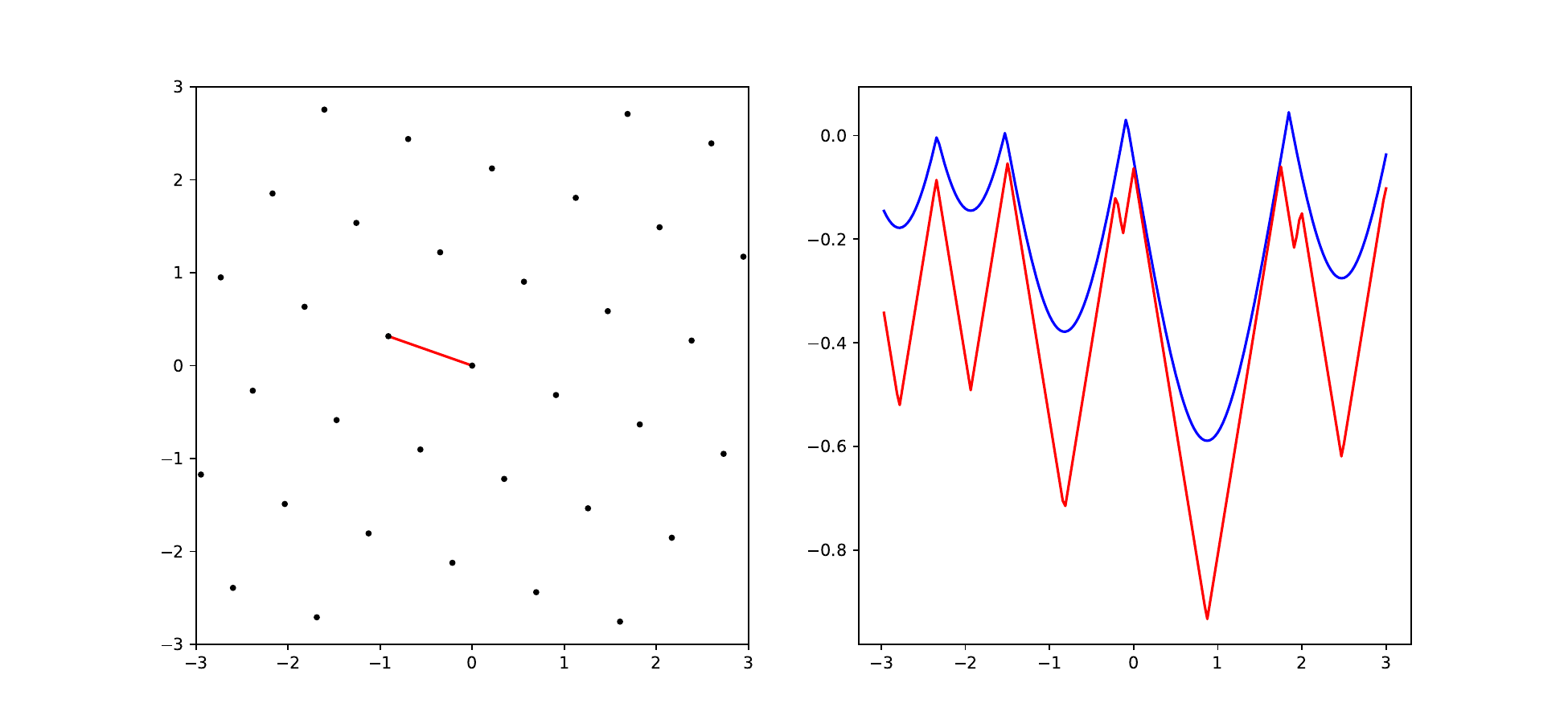}
    \caption{Visualization of a lattice and its corresponding log-systole functions $W(t)$ \textcolor{red}{(red)} and $W_2(t)$ \textcolor{blue}{(blue)} for $t \in [-3, 3]$. Note that the local minima of $W(t)$ and $W_2(t)$ does not strictly correspond to each other.}
    \label{fig:enter-label}
\end{figure}

The next step, we construct a Perron function without explicitly writing out the formula. 

\begin{lemma}
    For each local maximum $t_0$ of $W(t)$, consider the interval $[t_0-\ln (2), t_0+\ln (2)]$ and take 
    $$U = \bigcup_{t_0 \text{ is local maximum of } W(t)} [t_0-\ln (2), t_0+\ln (2)] $$
    Then
    $$\sup_{t \in \R} W_2(t) = \sup_{t \in U} W_2(t)$$
\end{lemma}
\begin{proof}
    We prove that when $t \notin U$, $W_2(t) < 0$. 

    When $t \notin U$, we find the closest of $t'$ such that $t'$ is a local maximum of $W(t)$. then $W$ is monotonic increasing on $[t,t']$ or monotonic decreasing on $[t',t]$ (if not, there are other local maxima on the interval, and closet to $t$, contradicting the assumption). We use Lemma \ref{systole} to show that $W(t') < 0$, so $W(t) \le -\ln(2)$, and $W_2(t) \le W(t) + \ln(2)/2 < 0$. 
\end{proof}

For each local maximal point of $W(t)$, following the computations in the projective invariance, there are lattice points lying on four sides of the maximal square box, which is one of the two cases:
$$
\begin{aligned}
    A &= \frac{(-\alpha, 1)}{\sqrt{1+\alpha \beta}},\quad B= \frac{(1, \beta)}{\sqrt{1+\alpha \beta}},\quad g_t \Lambda = \Lambda(\alpha, \beta) = \Z A + \Z B \\
    A &= \frac{(\alpha, 1)}{\sqrt{1+\alpha \beta}},\quad B= \frac{(-1, \beta)}{\sqrt{1+\alpha \beta}},\quad g_t\Lambda  = \Lambda(\alpha, \beta) = \Z A + \Z B \\
\end{aligned}
$$
The two cases are equivalent by quotient of a negation over $x$ axis which does not affect either $W$ or $W_2$ function, so without losing of generality, we discuss the first case. 
We define the Perron function as
$$f(\alpha, \beta) := \sup_{u \in [-\ln (2), \ln(2)]} W_2(t+u)$$
or an equivalent definition:
$$f(\alpha, \beta) := \sup_{u \in [-\ln (2), \ln(2)]} W_2(u, \Lambda(\alpha, \beta))$$

\begin{lemma}
    The Perron function $f$ is uniformly continuous with respect to $(\alpha, \beta) \in [0,1] \times [0,1]$.
\end{lemma}
\begin{proof}
    Define another map $F$ from $[0,1]^2$ to the Banach space of $C^0[-\ln(2), \ln(2)]$ with supremum norm \footnote{It's well known that the space of continuous functions on a bounded interval equipped with supremum norm is a Banach space.} by:
    $$F(\alpha, \beta) := W_2(\_, \Lambda(\alpha, \beta))$$
    We show that $F$ is a continuous map (not really linear). specifically, we show that for any $\epsilon > 0$, there exists $\delta$ so that for arbitrary $\lvert \alpha' - \alpha \rvert \le \delta$, $\lvert \beta' - \beta \rvert \le \delta$, and for $u \in [-\ln(2), \ln(2)]$ uniformly have
    $$\left\lvert W_2(u, \Lambda(\alpha', \beta')) - W_2(u, \Lambda(\alpha, \beta)) \right\rvert \le \epsilon$$
    We show that by a perturbation of $\alpha$ and $\beta$, the points of $\Lambda(\alpha, \beta)$ over a bounded region changes slightly.

    First, all points affecting $ W_2(u, \Lambda(\alpha, \beta))$ for $u \in [\ln (2), \ln(2)]$ are in the region of $[-4,4] \times [-4,4]$. This relation is obtained through that for any $P$ such that $\lVert g_u P\rVert = W_2(u; \Lambda(\alpha, \beta))$, $\lVert g_u P\rVert_2 \le (4/3)^{1/4}$. 

    Second, all points in $[-4,4] \times [-4,4]$ are uniformly continuous with respect to changes of $\alpha, \beta$. Since we know that $A = \frac{(-\alpha, 1)}{\sqrt{1+\alpha \beta}}$ and $B=\frac{(1, \beta)}{\sqrt{1+\alpha \beta}}$ are uniformly continuous with respect to changes of $\alpha, \beta$, that is to say, there exists a constant $C$, for any point $P = kA+lB \in [-4,4] \times [-4,4]$, it holds that $|k|, |l| \le C$. We choose $C=40$ which is large enough. If somehow $|k| > 40$, by simple calculation we show that $kA+lB \notin [-4,4] \times [-4,4]$, contradiction.
\end{proof}

\begin{corollary}
    The Perron function is a uniformly continuous function:
    $$f(\alpha, \beta) = \sup_{u \in [-\ln (2), \ln(2)]} F(\alpha, \beta)$$
    By taking the sequence $\{a_n\}$ such that $\alpha = [0; a_0, a_1, \cdots],\ \beta = [0; a_{-1}, a_{-2}, \cdots]$, one may also rewrite
    $$\kappa^+ (\Lambda)= \sup_{n \in \Z \text{ such that } \alpha_n^{[a]}, \beta_n^{[a]} \text{ are well-defined}} P_n^{[a]}$$
    Therefore, using Theorem \ref{infclosed}, the $\ell^2$ Mordell-Gruber spectrum is closed.
\end{corollary}

One may compute the Perron function explicitly, but it's not necessary in the proof of closedness. It might help, however, to prove or disprove that $\mathrm{MGL}_2$ contains an interval.
\begin{problem}
    Determine whether the $\ell^2$ Mordell-Gruber spectrum contains an interval.
\end{problem}

\section{Future Work}
So long, we have developed many tools to study the general spectra concerning the two-dimensional lattices (including the Mordell-Gruber spectrum) and one-dimensional Diophantine approximations. We have also generalized Perron's formula and Hall's ray to show that closedness and Hall's segment is actually ubiquitous across spectra.

The spectra discussed in this paper might share connections with number theory and quantum mechanics, since many quantum operators have closed spectrum. Discrete parts of such operators are usually called energy levels, and intervals in the spectra are called energy bands. It would be interesting to discover them if such connections exist, though it's also possible that it's just some coincidence.

While we have tried to generalize our results to higher dimensional lattices, as the article \cite{Shapira_Weiss_Mordell_Gruber} concerns, we fall short due to some obstacles:
\begin{itemize}
    \item Continued fractions are no longer applicable in lattices with dimension higher than three. The case of two-dimensional lattices is a very specific case.
    \item Pivots are no longer linearly ordered in higher dimensional lattices, they scatter near the union of all coordinate hyperplanes.
    \item Log-systole functions become multivariate. The log-systole function of an $n$-dimensional lattice depends on $n-1\ge2$ variables. This greatly increases the complexity for discussion.
\end{itemize}

Solving spectra for high-dimensional lattices still remains a challenge. Some high-dimensional lattices are closely related to open conjectures, such as Littlewood conjecture concerns a special case of three-dimensional lattices. More theories and tools are needed to solve spectra for high-dimensional lattices.

\section*{Acknowledgement}
I would like to thank my supervisor, Professor Yitwah Cheung, for his patient guidance, encouragement and fruitful discussions during my work. I would also thank Professor Nikolay Moshchevitin for his valuable comments and recommendations on some of my reference papers and books. I would especially thank Qiuzhen College for the cultivation and support during my four years of university study.

\printbibliography
\end{document}